\documentclass[11pt]{amsart}
\usepackage{amsmath,amssymb}
\usepackage{graphicx}
\usepackage{algorithm}
\usepackage{algpseudocode}
\usepackage{mathrsfs}
\usepackage{fullpage}
\usepackage{enumerate}
\usepackage{pgf,tikz}
\usepackage{todonotes}
\usepackage{appendix}
\usetikzlibrary{arrows}
\usetikzlibrary{positioning}
\usepackage{array}

\def\COMMENT#1{}
\let\COMMENT=\footnote
\def\TASK#1{}

\newdimen\margin   
\def\textno#1&#2\par{%
    \margin=\hsize
    \advance\margin by -4\parindent
           \setbox1=\hbox{\sl#1}%
    \ifdim\wd1 < \margin
       $$\box1\eqno#2$$%
    \else
       \bigbreak
       \hbox to \hsize{\indent$\vcenter{\advance\hsize by -3\parindent
       \sl\noindent#1}\hfil#2$}%
       \bigbreak
    \fi}

\newtheorem{thm}{Theorem}[section]
\newtheorem{define}[thm]{Definition}

\newtheorem{lem}[thm]{Lemma}
\newtheorem{claim}[thm]{Claim}
\newtheorem{fact}[thm]{Fact}

\newtheorem{conj}[thm]{Conjecture}
\newtheorem{prop}[thm]{Proposition}

\newtheorem*{thm*}{Theorem}
\newtheorem*{define*}{Definition}
\newtheorem*{examp*}{Example}
\newtheorem*{lem*}{Lemma}
\newtheorem*{claim*}{Claim}
\newtheorem*{fact*}{Fact}
\newtheorem*{col*}{Corollary}
\newtheorem*{conj*}{Conjecture}

\newtheorem*{case1}{Case 1}
\newtheorem*{case2}{Case 2}
\newtheorem*{case3}{Case 3}

\newtheorem*{prop:m2h1h2}{Proposition~\ref{prop:m2h1h2}}

\title{Towards the $0$-statement of the Kohayakawa-Kreuter Conjecture}
\author{Joseph Hyde}
\thanks{J.H. was supported by UKRI Future Leaders Fellowship grant MR/S016325/1 and ERC Advanced Grant 101020255.}

\begin{document}

\maketitle
\begin{abstract}
    In this paper, we study asymmetric Ramsey properties of the random graph $G_{n,p}$. 
    Let $r \in \mathbb{N}$ and $H_1, \ldots, H_r$ be graphs.
    We write $G_{n,p} \to (H_1, \ldots, H_r)$ to denote the property that whenever we colour the edges of $G_{n,p}$ with colours from the set $[r] := \{1, \ldots, r\}$ there exists $i \in [r]$ and a copy of $H_i$ in $G_{n,p}$ monochromatic in colour $i$. 
    There has been much interest in determining the asymptotic threshold function for this property. R\"{o}dl and Ruci\'{n}ski~\cite{rr1, rrrandom, rr2} determined a threshold function for the general symmetric case; that is, when $H_1 = \cdots = H_r$.
    A conjecture of Kohayakawa and Kreuter~\cite{kk}, if true, would fully resolve the asymmetric problem.
    Recently, the $1$-statement of this conjecture was confirmed by Mousset, Nenadov and Samotij~\cite{mns}.
    
    Building on work of Marciniszyn, Skokan, Sp\"{o}hel and Steger~\cite{msss}, we reduce the $0$-statement of Kohayakawa and Kreuter's conjecture to a certain deterministic subproblem.
    To demonstrate the potential of this approach, we show this subproblem can be resolved for almost all pairs of regular graphs. 
    This therefore resolves the $0$-statement for all such pairs of graphs.
\end{abstract}

\section{Introduction}\label{sec:intro}

Ramsey theory is one of the most studied areas in modern combinatorics. Let $r \in \mathbb{N}$ and let $G, H_1, \ldots, H_r$ be graphs.
We write $G \to (H_1, \ldots, H_r)$ to denote the property that whenever we colour the edges of $G$ with colours from the set $[r] := \{1, \ldots, r\}$ there exists $i \in [r]$ and a copy of $H_i$ in $G$ monochromatic in colour $i$. 
Thus, in this notation, the classical result of Ramsey~\cite{r} asserts that for $n$ sufficiently large $K_n \to (H_1, \ldots, H_r)$. 
One may posit that this is only true because $K_n$ is very dense, but Folkman~\cite{f}, and in a more general setting  Ne\v{s}et\v{r}il and R\"{o}dl~\cite{nr}, proved that for any graph $H$ there are locally sparse graphs $G = G(H)$ such that $G \to (H_1, \ldots, H_r)$ when $H_1 = \cdots = H_r = H$.

\subsection{Symmetric Ramsey properties: R\"{o}dl and Ruci\'{n}ski's theorem}

If we transfer our study of the Ramsey property to the random setting, we discover that such graphs $G$ are in fact very common.
Let $G_{n,p}$ be the binomial random graph with $n$ vertices and edge probability $p$. 
Improving on earlier work of Frankl and R\"{o}dl~\cite{fr}, \L{}uczak, Ruci\'{n}ski and Voigt~\cite{lrv} proved that $p = n^{-1/2}$ is a threshold for the property $G_{n,p} \to (K_3, K_3)$. 
Following this, R\"{o}dl and Ruci\'{n}ski~\cite{rr1, rrrandom, rr2} determined a threshold for the general symmetric case. 
For a graph $H$, we define \[d_2(H) := 
    \begin{cases}
        (e_H - 1)/(v_H - 2)         & \quad \text{if $H$ is non-empty with} \ v(H) \geq 3,\\
        1/2                         & \quad \text{if} \ H \cong K_2,\\    
        0                           & \quad \text{otherwise}
    \end{cases}\] and the \emph{$2$-density of $H$} to be \[m_2(H) := \max\{d_2(J): J \subseteq H\}.\] 
    
We say that a graph $H$ is \emph{$2$-balanced} if $d_2(H) = m_2(H)$, and \emph{strictly $2$-balanced} if for all proper subgraphs $J \subsetneq H$, we have $d_2(J) < m_2(H)$.

\begin{thm}[R\"{o}dl and Ruci\'{n}ski~\cite{rr2}]\label{thm:rr}
Let $r \geq 2$ and let $H$ be a non-empty graph such that at least one component of $H$ is not a star. If $r = 2$, then in addition restrict $H$ to having no component which is a path on $3$ edges. Then there exist positive constants $b, B > 0$ such that \[\lim\limits_{n \to \infty} \mathbb{P}[G_{n,p} \to (\underbrace{H, \ldots, H}_{r \ times})] = 
\begin{cases}
    0       & \quad \text{if } p \leq bn^{-1/m_2(H)},\\
    1       & \quad \text{if } p \geq Bn^{-1/m_2(H)}.
\end{cases}\] 
\end{thm}

The statement for $p \leq bn^{-1/m_2(H)}$ in Theorem~\ref{thm:rr} is known as a \emph{$0$-statement} and the statement for $p \geq Bn^{-1/m_2(H)}$ is known as a \emph{$1$-statement}.

The assumption on the structure of $H$ in Theorem~\ref{thm:rr} is necessary. If every component of $H$ is a star then $G_{n,p} \to (H, \ldots, H)$ as soon as sufficiently many vertices of degree $r(\Delta(H) - 1) + 1$ appear in $G_{n,p}$. 
A threshold for this property in $G_{n,p}$ is $p = n^{-1-1/(r(\Delta(H) - 1) + 1)}$, but $m_2(H) = 1$.
For the case when $r = 2$ and at least one component of $H$ is a path on 3 edges while the others are stars, the $0$-statement of Theorem~\ref{thm:rr} becomes false. 
Indeed, one can show that, if $p = cn^{-1/m_2(P_3)} = cn^{-1}$ for some $c > 0$, then the probability that $G_{n,p}$ contains a cycle of length 5 with an edge pending at every vertex is bounded from below by a positive constant $d = d(c)$.
One can check that every colouring of the edges of this augmented $5$-cycle with 2 colours yields a monochromatic path of length 3. 
This special case was missed in \cite{rr2}, and was eventually observed by Friedgut and Krivelevich~\cite{fk}, who corrected the $0$-statement to have the assumption $p = o(n^{-1/m_2(H)})$ instead. 
Note that Nenadov and Steger~\cite{ns} produced a short proof of Theorem~\ref{thm:rr} using the hypergraph container method.

The intuition behind the threshold in Theorem~\ref{thm:rr} is as follows: Firstly, assume $H$ is $2$-balanced. The expected number of copies of a graph $H$ in $G_{n,p}$ is $\Theta(n^{v(H)}p^{e(H)})$ and the expected number of edges is $\Theta(n^2p)$. For $p = n^{-1/m_2(H)}$ (the threshold in Theorem~\ref{thm:rr}), these two expectations are of the same order since $H$ is $2$-balanced. 
That is to say, if the expected number of copies of $H$ at a fixed edge is smaller than some small constant $c$, then we can hope to colour without creating a monochromatic copy of $H$: 
very roughly speaking, each copy will likely contain an edge not belonging to any other copy of $H$, so by colouring these edges with one colour and all other edges with a different colour we avoid creating monochromatic copies of $H$. 
If the expected number of copies of $H$ at a fixed edge is larger than some large constant $C$ then a monochromatic copy of $H$ may appear in any $r$-colouring since the copies of $H$ most likely overlap heavily.

Sharp thresholds for $G_{n,p} \to (H,H)$ have also been obtained when $H$ is a tree \cite{fk}, a triangle \cite{frrt} or, more generally, a strictly $2$-balanced graph that can be made bipartite by the removal of some edge \cite{ss}. 

\begin{thm}[\cite{fk, frrt, ss}]
Suppose that $H$ is either (i) a tree that is not a star or the path of length three or (ii) a strictly $2$-balanced graph with $e(H) \geq 2$ that can be made bipartite by the removal of some edge. 
Then there exist constants $c_0$, $c_1$, and a function $c: \mathbb{N} \to [c_0, c_1]$, such that \[\lim\limits_{n \to \infty} \mathbb{P}[G_{n,p} \to (H, H)] = 
\begin{cases}
    0       & \quad \text{if } p \leq (1 - \varepsilon)c(n)n^{-1/m_2(H)},\\
    1       & \quad \text{if } p \geq (1 + \varepsilon)c(n)n^{-1/m_2(H)}.
\end{cases}\]  for every positive constant $\varepsilon$.
\end{thm}

\subsection{Asymmetric Ramsey properties: The Kohayakawa-Kreuter Conjecture}

In this paper, we are interested in asymmetric Ramsey properties of $G_{n,p}$, that is, finding a threshold for when $G_{n,p} \to (H_1 \ldots, H_r)$ and $H_1, \ldots, H_r$ are not all the same graph. 
In classical Ramsey theory, the study of asymmetric Ramsey properties sparked off many interesting routes of research (see, e.g. \cite{cg}), including the seminal work of Kim~\cite{kim} on establishing an asymptotically sharp lower bound on the Ramsey number $R(3,t)$. 
In $G_{n,p}$, asymmetric Ramsey properties were first considered by Kohayakawa and Kreuter~\cite{kk}. 
For graphs $H_1$ and $H_2$ with $m_2(H_1) \geq m_2(H_2)$, we define \[d_2(H_1,H_2) := 
    \begin{cases}
        \frac{e(H_1)}{v(H_1) - 2 + \frac{1}{m_2(H_2)}}         & \quad \text{if $H_2$ is non-empty and} \ v(H_1) \geq 2,\\
        0                         & \quad \text{otherwise}
    \end{cases}\] and the \emph{asymmetric $2$-density of the pair $(H_1, H_2)$} to be \[m_2(H_1, H_2) := \max\left\{d_2(J,H_2): J \subseteq H_1\right\}.\]
    
We say that $H_1$ is \emph{balanced with respect to $d_2(\cdot, H_2)$} if we have $d_2(H_1, H_2) = m_2(H_1, H_2)$ and \emph{strictly balanced with respect to $d_2(\cdot, H_2)$} if for all proper subgraphs $J \subsetneq H_1$ we have $d_2(J, H_2) < m_2(H_1, H_2)$. Note that $m_2(H_1) \geq m_2(H_1, H_2) \geq m_2(H_2)$ (see Proposition~\ref{prop:m2h1h2}).

Kohayakawa and Kreuter~\cite{kk} conjectured the following generalisation of Theorem~\ref{thm:rr}. 
(We give here a slight rephrasing of the conjecture: we consider $r$ colours (instead of 2) and add the assumption of Kohayakawa, Schacht and Sp\"{o}hel~\cite{kss} that $H_1$ and $H_2$ are not forests.\footnote{This version of the conjecture is the same as that given in \cite{mns}.})

\begin{conj}[Kohayakawa and Kreuter~\cite{kk}]\label{conj:kk}
Let $r \geq 2$ and suppose that $H_1, \ldots, H_r$ are non-empty graphs such that $m_2(H_1) \geq \cdots \geq m_2(H_r)$ and $m_2(H_2) > 1$.
Then there exist constants $b, B > 0$ such that \[\lim\limits_{n \to \infty} \mathbb{P}[G_{n,p} \to (H_1, \ldots, H_r)] =   
\begin{cases}
    0       & \quad \text{if } p \leq bn^{-1/m_2(H_1,H_2)},\\
    1       & \quad \text{if } p \geq Bn^{-1/m_2(H_1,H_2)}.
\end{cases}\] 
\end{conj}

Observe that we would always need $m_2(H_2) \geq 1$ as an assumption, otherwise $m_2(H_2) = 1/2$ (that is, $H_2$ is the union of a matching and some isolated vertices) and we would have that $m_2(H_1, H_2) = e_J/v_J$ for some non-empty subgraph $J \subseteq H_1$.
For any constant $B > 0$, the probability that $G_{n,p}$ with $p = Bn^{-1/m_2(H_1, H_2)}$ contains no copy of $H_1$ exceeds a positive constant $C = C(B)$; see, e.g. \cite{jlr}. 
We include the assumption of Kohayakawa, Schacht and Sp\"{o}hel~\cite{kss}, that $m_2(H_2) > 1$, to avoid possible complications arising from $H_2$ (and/or $H_1$) being certain forests, such as those excluded in the statement of Theorem~\ref{thm:rr}.

The intuition behind the threshold in Conjecture~\ref{conj:kk} is most readily explained in the case of $r = 3$, $H_2 = H_3$ and when $m_2(H_1) > m_2(H_1, H_2)$. (The following explanation is adapted from \cite{gnpsst}.)
Firstly, observe that we can assign colour 1 to every edge that does not lie in a copy of $H_1$. 
Since $m_2(H_1) > m_2(H_1, H_2)$, we expect that the copies of $H_1$ in $G_{n,p}$ with $p = \Theta(n^{-1/m_2(H_1,H_2)})$ do not overlap much (by similar reasoning as in the intuition for the threshold in Theorem~\ref{thm:rr}). 
Hence the number of edges left to be coloured is of the same order as the number of copies of $H_1$, which is $\Theta(n^{v(H_1)}p^{e(H_1)})$. 
If we further assume that these edges are randomly distributed (which is not correct, but gives good intuition) then we get a random graph $G^*$ with edge probability $p^* =  \Theta(n^{v(H_1) - 2}p^{e(H_1)})$. 
Now we colour $G^*$ with colours 2 and 3, and apply the intuition from the symmetric case (as $H_2 = H_3$): if the copies of $H_2$ are heavily overlapping then we cannot hope to colour without getting a monochromatic copy of $H_2$, but if not then we should be able to colour.
As observed before, a threshold for this property is $p^* = n^{-1/m_2(H_2)}$. Solving $n^{v(H_1) - 2}p^{e(H_1)} = n^{-1/m_2(H_2)}$ for $p$ then yields $p = n^{-1/m_2(H_1,H_2)}$, the conjectured threshold. 


\smallskip
After earlier work (see e.g. \cite{gnpsst, hst, kk, kss, msss}), the $1$-statement of Conjecture~\ref{conj:kk} was proven by Mousset, Nenadov and Samotij~\cite{mns}.

\subsection{The $0$-statement of the Kohayakawa-Kreuter Conjecture}\label{sec:0statement}

In this paper, we are interested in the $0$-statement of Conjecture~\ref{conj:kk}, 
which has so far only been proven when $H_1$ and $H_2$ are both cycles \cite{kk}, both cliques \cite{msss} and, recently, when $H_1$ is a clique and $H_2$ is a cycle \cite{lmms}. We also note that the authors of \cite{gnpsst} prove, under certain balancedness conditions, the $0$-statement of a generalised version of Conjecture~\ref{conj:kk} which allows $H_1, \ldots, H_r$ to be uniform hypergraphs.

\begin{define}
\textnormal{Let $H_1$ and $H_2$ be non-empty graphs.
We say that a graph $G$ has a \emph{valid edge-colouring for $H_1$ and $H_2$} if there exists a red/blue colouring of the edges of $G$ that does not produce a red copy of $H_1$ or a blue copy of $H_2$.}
\end{define}

To prove the $0$-statement of Conjecture~\ref{conj:kk}, one only needs to show that $G = G_{n,p}$ has a valid edge-colouring for $H_1$ and $H_2$ asymptotically almost surely (a.a.s.)~(that is, we ignore $H_3, \ldots, H_r$ and colours $3, \ldots, r$). 
Further, when $m_2(H_1) > m_2(H_2)$ we can assume when trying to prove the $0$-statement of Conjecture~\ref{conj:kk} that $H_2$ is strictly $2$-balanced and $H_1$ is strictly balanced with respect to $d_2(\cdot, H_2)$.
Indeed, if either of these assumptions do not hold then one can replace $H_1$ and $H_2$ with subgraphs $H_1' \subseteq H_1$ and $H_2' \subseteq H_2$ such that $H_2'$ is strictly $2$-balanced and $H_1'$ is strictly balanced with respect to $d_2(\cdot, H_2')$.
Then we would aim to show that $G$ has a valid edge-colouring for $H_1'$ and $H_2'$ a.a.s.\footnote{Which would immediately imply that $G$ has a valid edge-colouring for $H_1$ and $H_2$.}~Similarly, when $m_2(H_1) = m_2(H_2)$, we can assume when proving the $0$-statement of Conjecture~\ref{conj:kk} that both $H_1$ and $H_2$ are strictly $2$-balanced.  

In past work on attacking $0$-statements of Ramsey problems (e.g.~Conjecture~\ref{conj:kk} and Theorem~\ref{thm:rr}), researchers have applied variants of a standard and natural approach (see e.g. \cite{kk, lmms, msss, rr1}). 
The main contribution of this paper is to prove that every step of this approach, except one, holds with respect to Conjecture~\ref{conj:kk}. 
That is, we reduce Conjecture~\ref{conj:kk} to a single subproblem. To state this subproblem we require a number of definitions adapted from \cite{msss}.

\begin{define}\label{def:c*}
\textnormal{For any graph $G$ we define the families \[\mathcal{R}_G := \{R \subseteq G : R \cong H_1\} \ \mbox{and}\ \mathcal{L}_G := \{L \subseteq G : L \cong H_2\}\] of all copies of $H_1$ and $H_2$ in $G$, respectively. 
Furthermore, we define \[\mathcal{L}^*_G := \{L \in \mathcal{L}_G : \forall e \in E(L) \ \exists R \in \mathcal{R}_G \ \mbox{s.t.} \ E(L) \cap E(R) = \{e\}\} \subseteq \mathcal{L}_G,\] the family of copies $L$ of $H_2$ in $G$ with the property that for every edge $e$ in $L$ there exists a copy $R$ of $H_1$ such that the edge sets of $L$ and $R$ intersect uniquely at $e$; 
\[\mathcal{C} = \mathcal{C}(H_1, H_2) := \{G = (V,E) : \forall e \in E\ \exists(L,R) \in \mathcal{L}_G \times \mathcal{R}_G\ \mbox{s.t.}\ E(L) \cap E(R) = \{e\}\},\] the family of graphs $G$ where every edge is the unique edge-intersection of some copy $L$ of $H_2$ and some copy $R$ of $H_1$;
and 
\[\mathcal{C}^* = \mathcal{C}^*(H_1, H_2) := \{G = (V,E): \forall e \in E \ \exists L \in \mathcal{L}^*_{G} \ \mbox{s.t.} \ e \in E(L)\},\] the family of graphs $G$ where every edge is contained in a copy $L$ of $H_2$ which has at each edge $e$ some copy $R$ of $H_1$ attached such that $E(L) \cap E(R) = \{e\}$.}
\end{define}

Note that $\mathcal{C}^*(H_1, H_2) \subseteq \mathcal{C}(H_1, H_2)$. Also, it is important to note that for $L \in \mathcal{L}_G^*$, the vertex sets of copies of $H_1$ appended at edges of $L$ may overlap with each other and/or intersect with more than $2$ vertices of $L$ (that is, more than the $2$ vertices of the edge of $L$ they are appended at).

Let us now discuss the relevance of these sets of graphs to proving the $0$-statement of Conjecture~\ref{conj:kk}. 
Recall that we aim to show that $G = G_{n,p}$ has a valid edge-colouring for $H_1$ and $H_2$. 
Now, certain obstacles relating to $\mathcal{C}$, $\mathcal{L}^*_G$ and $\mathcal{C}^*$ may appear while we are constructing such a valid edge-colouring. 
For instance, say there exists a subgraph $G' \subseteq G$ such that $G' \in \mathcal{C}$. 
Then each edge $e \in E(G')$ has a copy $R$ of $H_1$ and a copy $L$ of $H_2$ in $G$ that uniquely edge-intersect in that edge. 
Say during the construction of our colouring we come to a point where every edge of $E(R)\setminus \{e\}$ is coloured red,
every edge of $E(L)\setminus \{e\}$ is coloured blue and $e$ is yet to be coloured (see Figure~\ref{fig:cbad}). Then however we colour the edge $e$, we have a red copy of $H_1$ or a blue copy of $H_2$. 
Hence one must be careful when constructing a valid edge-colouring in graphs from $\mathcal{C}$. 

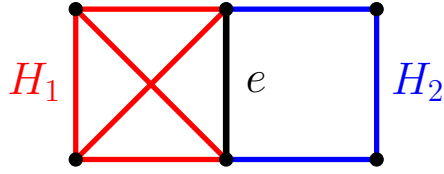
\begin{figure}[!ht]
\begin{center}
\begin{tikzpicture}[line cap=round,line join=round,>=triangle 45,x=1cm,y=1cm]
\draw [line width=2.5pt] (4,6)-- (4,4);
\draw[color=black] (4.4,5.035) node {\LARGE $e$};
\draw [line width=2pt,color=blue] (4,4)-- (6,4);
\draw [line width=2pt,color=blue] (6,4)-- (6,6);
\draw [line width=2pt,color=blue] (6,6)-- (4,6);
\draw[color=blue] (6.55,5.035) node {\LARGE $H_2$};
\draw [line width=2pt,color=red] (2,6)-- (4,4);
\draw [line width=2pt,color=red] (2,4)-- (4,6);
\draw [line width=2pt,color=red] (2,6)-- (4,6);
\draw [line width=2pt,color=red] (2,6)-- (2,4);
\draw [line width=2pt,color=red] (2,4)-- (4,4);
\draw[color=red] (1.45,5.035) node {\LARGE $H_1$};
\begin{scriptsize}
\draw [fill=black] (4,6) circle (2.5pt);
\draw [fill=black] (4,4) circle (2.5pt);
\draw [fill=black] (6,4) circle (2.5pt);
\draw [fill=black] (6,6) circle (2.5pt);
\draw [fill=black] (2,6) circle (2.5pt);
\draw [fill=black] (2,4) circle (2.5pt);
\end{scriptsize}
\end{tikzpicture}
\vspace{0.4cm}
\caption{A copy of $H_1$ intersecting a copy of $H_2$ at an edge $e$ where $H_1 = K_4$, $H_2 = C_4$ and $e$ is yet to be coloured.}\label{fig:cbad}
\end{center}
\end{figure} 

Now say there exists a subgraph $G' \subseteq G$ such that $G' \in \mathcal{C}^*$. 
Then each edge $e \in E(G')$ is contained in some copy $L$ of $H_2$ such that $L \in \mathcal{L}^*_{G'}$, that is, $L$ has a different copy $R_{e'}$ of $H_1$ appended at each edge $e' \in E(L)$ such that $E(L) \cap E(R_{e'}) = \{e\}$. 
Similarly as before, say during the construction of our colouring we come to a point where every edge of $E(R_{e'})\setminus \{e'\}$ is red (see Figure~\ref{fig:c*bad}).

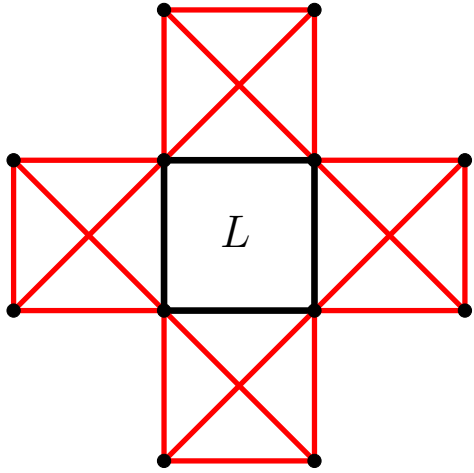
\begin{figure}[!ht]
\begin{center}
\begin{tikzpicture}[line cap=round,line join=round,>=triangle 45,x=1cm,y=1cm]
\draw [line width=2.5pt,color=black] (4,6)-- (4,4);
\draw [line width=2.5pt,color=black] (4,4)-- (6,4);
\draw [line width=2.5pt,color=black] (6,4)-- (6,6);
\draw [line width=2.5pt,color=black] (6,6)-- (4,6);
\draw[color=black] (4.95,5.05) node {\LARGE $L$};
\draw [line width=2pt,color=red] (8,6)-- (8,4);
\draw [line width=2pt,color=red] (2,6)-- (4,4);
\draw [line width=2pt,color=red] (2,4)-- (4,6);
\draw [line width=2pt,color=red] (4,6)-- (6,8);
\draw [line width=2pt,color=red] (4,8)-- (6,6);
\draw [line width=2pt,color=red] (4,4)-- (6,2);
\draw [line width=2pt,color=red] (4,2)-- (6,4);
\draw [line width=2pt,color=red] (6,6)-- (8,4);
\draw [line width=2pt,color=red] (6,4)-- (8,6);
\draw [line width=2pt,color=red] (4,6)-- (4,8);
\draw [line width=2pt,color=red] (4,8)-- (6,8);
\draw [line width=2pt,color=red] (6,8)-- (6,6);
\draw [line width=2pt,color=red] (2,6)-- (4,6);
\draw [line width=2pt,color=red] (2,6)-- (2,4);
\draw [line width=2pt,color=red] (2,4)-- (4,4);
\draw [line width=2pt,color=red] (4,4)-- (4,2);
\draw [line width=2pt,color=red] (4,2)-- (6,2);
\draw [line width=2pt,color=red] (6,2)-- (6,4);
\draw [line width=2pt,color=red] (6,4)-- (8,4);
\draw [line width=2pt,color=red] (6,6)-- (8,6);
\begin{scriptsize}
\draw [fill=black] (4,6) circle (2.5pt);
\draw [fill=black] (4,4) circle (2.5pt);
\draw [fill=black] (6,4) circle (2.5pt);
\draw [fill=black] (6,6) circle (2.5pt);
\draw [fill=black] (2,6) circle (2.5pt);
\draw [fill=black] (2,4) circle (2.5pt);
\draw [fill=black] (4,8) circle (2.5pt);
\draw [fill=black] (6,8) circle (2.5pt);
\draw [fill=black] (4,2) circle (2.5pt);
\draw [fill=black] (6,2) circle (2.5pt);
\draw [fill=black] (8,4) circle (2.5pt);
\draw [fill=black] (8,6) circle (2.5pt);
\end{scriptsize}
\end{tikzpicture}
\vspace{0.4cm}
\caption{A copy $L$ of $H_2$ such that $L \in \mathcal{L}^*_{G'}$ where $H_1 = K_4$, $H_2 = C_4$ and the edges of $L$ are yet to be coloured.}\label{fig:c*bad}
\end{center}
\end{figure}

However we colour the edges of $L$, we will produce a red copy of $H_1$ or a blue copy of $H_2$. 
Observe that the structure in Figure~\ref{fig:c*bad} is essentially a generalisation of the structure in Figure~\ref{fig:cbad}.
Indeed, $\mathcal{C}^* \subseteq \mathcal{C}$.
Hence graphs in $\mathcal{C}^*$ are also possibly difficult to construct a valid edge-colouring in.

For the next section we need also the following definition. For a graph $H$, we define $d(H) := e_H/v_H$ if $v(H) \geq 1$ and $d(H) := 0$ otherwise. Also, define $m(H) := \max\{d(J): J \subseteq H\}$.

\subsection{The family of graphs $\hat{\mathcal{A}}(H_1, H_2, \varepsilon)$}\label{sec:familya} Throughout the rest of this section, assume that $n$ is sufficiently large and for some constant $b > 0$ that $p = bn^{-1/m_2(H_1,H_2)}$. We now define a very important family of graphs which we will need to state our reduction of Conjecture~\ref{conj:kk}. These graphs present potentially significant obstacles to constructing a valid-edge colouring in $G_{n,p}$ a.a.s. After defining this family of graphs we will explain the thinking behind each part of the definition.

\begin{define}\label{def:a*}
\textnormal{Let $H_1$ and $H_2$ be non-empty graphs such that $m_2(H_1) \geq m_2(H_2) > 1$. Let $\varepsilon := \varepsilon(H_1, H_2) > 0$ be a constant. Define $\hat{\mathcal{A}} = \hat{\mathcal{A}}(H_1,H_2, \varepsilon)$ to be \begin{small}\[\hat{\mathcal{A}} := \begin{cases}
    \{A \in \mathcal{C}^*(H_1, H_2): m(A) \leq m_2(H_1, H_2) + \varepsilon \land A \ \mbox{is $2$-connected}\} \ \mbox{if} \ m_2(H_1) > m_2(H_2),\\
    \{A \in \mathcal{C}(H_1, H_2): m(A) \leq m_2(H_1, H_2) + \varepsilon \land A \ \mbox{is $2$-connected}\} \ \mbox{if} \ m_2(H_1) = m_2(H_2).
\end{cases}\]\end{small}}
\end{define}

\subsubsection{Why do we have $m(A) \leq m_2(H_1, H_2) + \varepsilon$?} We require the following definition: For any graph $F$, let \[\lambda(F) := v(F) - \frac{e(F)}{m_2(H_1,H_2)}.\] 
The definition of $\lambda(F)$ is motivated by the fact that the expected number of copies of $F$ in $G_{n,p}$ has order of magnitude \[n^{v(F)}p^{e(F)} = b^{e(F)}n^{\lambda(F)}.\] 
Let $A \in \hat{\mathcal{A}}(H_1, H_2, \varepsilon)$ for some pair of graphs $H_1$ and $H_2$ with $m_2(H_1) \geq m_2(H_2) > 1$ and a constant $\varepsilon := \varepsilon(H_1, H_2) > 0$.
Now, assuming $m(A) = d(A)\ (= |E(A)|/|V(A)|)$, observe that the following inequalities are equivalent.
\begin{eqnarray*}
    m(A) & < & m_2(H_1, H_2) \\
    \frac{|E(A)|}{m_2(H_1, H_2)} & < & |V(A)| \\
     0 & < &  \lambda(A).
\end{eqnarray*}
Thus, if $m(A) < m_2(H_1, H_2)$, then the expected number of copies of $A$ in $G_{n,p}$ has order of magnitude $\omega(1)$.
Thus such graphs $A$ can be expected to be found in $G_{n,p}$ a.a.s.

Now consider if instead we had $m(A) > m_2(H_1, H_2) + \varepsilon$. Then the following inequalities are equivalent.
\begin{eqnarray*}
    m(A) & > & m_2(H_1, H_2) + \varepsilon \\
    \lambda(A) & < & -\frac{|V(A)|\varepsilon}{m_2(H_1, H_2)}.
\end{eqnarray*}
Thus the expected number of copies of $A$ in $G_{n,p}$ has order of magnitude $o(1)$,
that is, there are no copies of $A$ in $G_{n,p}$ a.a.s. 

Hence graphs $A \in \hat{\mathcal{A}}$ have sufficiently low density to plausibly exist in $G_{n,p}$.
Note that we stipulate $m(A) \leq m_2(H_1, H_2) + \varepsilon$ rather than $m(A) < m_2(H_1, H_2)$ so that later we can argue\footnote{For instance, in the proof of Claim~\ref{claim:grow}.} that $2$-connected graphs $A \in C^*(H_1, H_2)$ that do not belong to $\hat{\mathcal{A}}$ have $m(A) > m_2(H_1, H_2) + \varepsilon$,
and so, as just noted, do not appear in $G_{n,p}$ a.a.s.

\subsubsection{Why is $A \in \mathcal{C}$ or $\mathcal{C}^*$?}\label{sec:whyaincorc^*}

As discussed at the end of Section~\ref{sec:0statement}, graphs in $\mathcal{C}$ and $\mathcal{C}^*$ contain certain structures which could be obstacles to constructing a valid edge-colouring. 

\subsubsection{Why is whether $A \in C^*$ or $A \in C$ dependent on $m_2(H_1)$ and $m_2(H_2)$?}\label{sec:ainc*orinc} Now let us consider why in the definition of $\hat{\mathcal{A}}$ we have that $A \in \mathcal{C}^*(H_1, H_2)$ when $m_2(H_1) > m_2(H_2)$ but $A \in \mathcal{C}(H_1, H_2)$ when $m_2(H_1) = m_2(H_2)$. 
In short, the defining structure of graphs in $\mathcal{C}^*(H_1, H_2)$ - that of a copy of $H_2$ with appended copies of $H_1$ at its edges (see Figure~\ref{fig:c*bad} for example) - 
is `meaningful' when $m_2(H_1) > m_2(H_2)$ but not when $m_2(H_1) = m_2(H_2)$. The following result will aid us in elaborating on this remark. It illuminates the relationship between the one and two argument $m_2$ measures and can be readily proven using elementary arguments.

\begin{prop}\label{prop:m2h1h2}
Suppose that $H_1$ and $H_2$ are non-empty graphs with $m_2(H_1) \geq m_2(H_2)$. 
Then we have \[m_2(H_1) \geq m_2(H_1, H_2) \geq m_2(H_2).\]
Moreover, \[m_2(H_1) > m_2(H_1, H_2) > m_2(H_2) \ \mbox{whenever} \ m_2(H_1) > m_2(H_2).\]
\end{prop}

Recall from earlier that we can assume when proving the $0$-statement of Conjecture~\ref{conj:kk} that $m_2(H_1) \geq m_2(H_2) \geq 1$, $H_2$ is strictly $2$-balanced, and $H_1$ is strictly balanced with respect to $d_2(\cdot, H_2)$ if $m_2(H_1) > m_2(H_2)$ and strictly $2$-balanced if $m_2(H_1) = m_2(H_2)$. 
Assuming these density conditions, one can view the values of $m_2(H_1)$,
$m_2(H_1, H_2)$ and $m_2(H_2)$ in a particular way which will be relevant across this paper.
Take a copy of $H_1$ and attach a copy of $H_2$ to get the same structure as in Figure~\ref{fig:cbad}.
Then, if we take the number of edges we added to the copy of $H_1$ and divide it by the number of vertices we added, this is precisely $$\frac{|E(H_2)| - 1}{|V(H_2)| - 2} = m_2(H_2).$$ 
Similarly, 
if one takes a copy of $H_1$ and attaches a copy of $H_2$ with $|E(H_2)| - 1$ appended copies of $H_1$ at each of its edges precisely as in Figure~\ref{fig:c*bad}, 
then the number of edges added to the copy of $H_1$ over the number of vertices added is \begin{eqnarray*}
  \frac{|E(H_1)|(|E(H_2)| - 1)}{(|V(H_1)| - 2)(|E(H_2)| - 1) + (|V(H_2)| - 2)} & = & \frac{|E(H_1)|}{|V(H_1)| - 2 + \frac{|V(H_2)| - 2}{|E(H_2)| - 1}}
  \\ & = & \frac{|E(H_1)|}{|V(H_1)| - 2 + \frac{1}{m_2(H_2)}} =  m_2(H_1, H_2),
\end{eqnarray*}
where in the $m_2(H_1) = m_2(H_2)$ case the final equality follows from $H_1$ being strictly $2$-balanced and Proposition~\ref{prop:m2h1h2}.\footnote{That is, $H_1$ is balanced with respect to $d_2(\cdot, H_2)$ when $m_2(H_1) = m_2(H_2)$ and $H_1$ is strictly $2$-balanced.}
Also, if we attach a copy of $H_1$ to a copy of $H_1$ or $H_2$ in a similar manner to the structure in Figure~\ref{fig:cbad} (at a single edge with no additional vertices or edges overlapping) and $H_1$ is strictly $2$-balanced, 
then the number of edges added over the number of vertices added is $m_2(H_1)$. For brevity, let $F_1$ be the graph in Figure~\ref{fig:cbad} and $F_2$ be the graph in Figure~\ref{fig:c*bad}, ignoring the colouring of the edges.
If $m_2(H_1) = m_2(H_2)$, then $m_2(H_1) = m_2(H_1, H_2) = m_2(H_2)$ by Proposition~\ref{prop:m2h1h2} and one can calculate that $\lambda(H_1) = \lambda(H_2) = \lambda(F_1) = \lambda(F_2)$, that is, the expected numbers of these graphs in $G_{n,p}$ are approximately the same; they have the same orders of magnitude. Moreover,
if we appended less than $|E(H_2)| - 1$ copies of $H_1$ to the copy of $H_2$\footnote{In the same manner as $F_2$ - attaching the copies at single edges and not overlapping additionally with any other vertices or edges of the copy of $H_2$ or any of the other copies of $H_1$.} - call such a graph $F_2'$ -  then we would still have $\lambda(F_2') = \lambda(H_1)$. 
However, when $m_2(H_1) > m_2(H_2)$ we have $m_2(H_1) > m_2(H_1, H_2) > m_2(H_2)$ by Proposition~\ref{prop:m2h1h2}, and one can calculate that $\lambda(H_1) = \lambda(F_2)$, but $\lambda(H_1) < \lambda(F_1)$. Thus, speaking broadly, $F_1$ is more likely to appear in $G_{n,p}$ than $H_1$. In fact, $\lambda(H_1) < \lambda(F_2')$, irrelevant of the position and number of the appended copies of $H_1$ in $F_2'$.

Thus when $m_2(H_1) = m_2(H_2)$, $A_2$ is not a particularly meaningful construction, but when $m_2(H_1) > m_2(H_2)$ we see that $A_2$ makes more sense to consider. One can observe that this accords with Proposition~\ref{prop:m2h1h2}, in that either $m_2(H_1) = m_2(H_1, H_2) = m_2(H_2)$, and so $m_2(H_1, H_2)$ does not have a distinct value, or $m_2(H_1) > m_2(H_1, H_2) > m_2(H_2)$ and $m_2(H_1, H_2)$ does have a distinct value.
See Section~\ref{sec:concludingremarks} for additional discussion on why for $A \in \hat{\mathcal{A}}$ we take $A \in \mathcal{C}$ when $m_2(H_1) = m_2(H_2)$.

\subsubsection{Why is $A$ a $2$-connected graph?} 
Let $G = G_{n,p}$ and consider the collection of $\boldsymbol{F}_2(G)$ of maximally $2$-connected subgraphs $A$ of $G$.\footnote{That is, if $e \in A$ for some graph $A$ in $\boldsymbol{F}_2(G)$, then there does not exist some $A' \subseteq G$ such that $A \subset A'$ and $A'$ is $2$-connected.} Thus $E(G)$ can be partitioned into $\boldsymbol{F}_2(G)$ and a forest $\boldsymbol{F}_1(G)$.
Later, we will prove that the density conditions we can assume for $H_1$ and $H_2$ when proving the $0$-statement of Conjecture~\ref{conj:kk} imply that $H_1$ and $H_2$ are $2$-connected (see Lemma~\ref{lemma:2connected}). 
Thus, assuming these density conditions, no copy of $H_1$ or $H_2$ has edges that lie in two different graphs in $\boldsymbol{F}_2(G)$. Importantly, this means that if each graph in $\boldsymbol{F}_2(G)$ has a valid edge-colouring for $H_1$ and $H_2$, then there exists a valid edge-colouring covering every graph in $\boldsymbol{F}_2(G)$. Further, the edges of the forest $\boldsymbol{F}_1(G)$ belong to no copies of $H_1$ and $H_2$ in $G$, thus we can colour them any way we want.
So with regard to finding a valid edge-colouring of $G$, we can reduce to looking at $2$-connected graphs. Hence we reduce to looking at $2$-connected graphs for $\hat{\mathcal{A}}$.

\subsection{Reduction of Conjecture~\ref{conj:kk}} We now state the subproblem we reduce Conjecture~\ref{conj:kk} to as the following conjecture.

\begin{conj}\label{conj:aconj}
Let $H_1$ and $H_2$ be non-empty graphs such that $H_1 \neq H_2$ and $m_2(H_1) \geq m_2(H_2)$. Assume $H_2$ is strictly $2$-balanced. Moreover, assume $H_1$ is strictly balanced with respect to $d_2(\cdot, H_2)$ if $m_2(H_1) > m_2(H_2)$ and strictly $2$-balanced if $m_2(H_1) = m_2(H_2)$.
Then there exists a constant $\varepsilon := \varepsilon(H_1, H_2) > 0$ such that the set $\hat{\mathcal{A}}$ is finite and every graph in $\hat{\mathcal{A}}$ has a valid edge-colouring for $H_1$ and $H_2$.
\end{conj}

Notice that we can assume $H_1 \neq H_2$ as the $H_1 = H_2$ case of Conjecture~\ref{conj:kk} is handled by Theorem~\ref{thm:rr}.

To be clear, the main purpose of this paper is to show that if Conjecture~\ref{conj:aconj} holds then the rest of a variant of a standard approach for attacking the $0$-statement of Conjecture~\ref{conj:kk} falls into place.
That is, Conjecture~\ref{conj:aconj} is a natural subproblem of Conjecture~\ref{conj:kk}. Thus we prove the following theorem.

\begin{thm}\label{thm:main}
If Conjecture~\ref{conj:aconj} is true then the $0$-statement of Conjecture~\ref{conj:kk} is true.
\end{thm}

We prove Conjecture~\ref{conj:aconj} for almost every pair of regular graphs, which, by Theorem~\ref{thm:main}, significantly extends the class of graphs for which the $0$-statement of Conjecture~\ref{conj:kk} is resolved. 

\begin{thm}\label{thm:regular}
Let $H_1$ and $H_2$ meet the criteria in Conjecture~\ref{conj:aconj}. In addition, let $H_1$ and $H_2$ be regular graphs, excluding the cases when (i) $H_1$ and $H_2$ are a clique and a cycle, (ii) $H_2$ is a cycle and $|V(H_1)| \geq |V(H_2)|$ or (iii) $(H_1, H_2) = (K_3, K_{3,3})$.
Then Conjecture~\ref{conj:aconj} is true for $H_1$ and $H_2$.
\end{thm}

As a natural subproblem of the $0$-statement of Conjecture~\ref{conj:kk}, we believe that Conjecture~\ref{conj:aconj} is a considerably more approachable problem than the $0$-statement of Conjecture~\ref{conj:kk}.
Indeed, the techniques used in the proof of Theorem~\ref{thm:regular} are elementary and uncomplicated. Thus, we hope that a full resolution of Conjecture~\ref{conj:kk} can be achieved via Theorem~\ref{thm:main}.

\section{Overview of the proof of Theorem~\ref{thm:main}}\label{sec:proofsketch}

As mentioned earlier, to prove Theorem~\ref{thm:main} we will employ a variant of a standard approach for attacking $0$-statements of Ramsey problems.
For attacking the $0$-statement of Conjecture~\ref{conj:aconj}, this standard approach is as follows: 

\begin{itemize}
    \item For $G = G_{n,p}$, assume $G \to (H_1, H_2)$;
    \item Use structural properties of $G$ (resulting from this assumption) to show that $G$ contains at least one of a sufficiently small collection of non-isomorphic graphs $\mathcal{F}$;
    \item Show that there exists a constant $b > 0$ such that for $p \leq bn^{-1/m_2(H_1,H_2)}$ we have that $G$ contains no graph in $\mathcal{F}$ a.a.s.;
    \item Conclude, by contradiction, that $G \not\to (H_1, H_2)$ a.a.s.
\end{itemize}  

The variant of this approach we will use is due to Marciniszyn, Skokan, Sp\"{o}hel and Steger~\cite{msss}, who proved Conjecture~\ref{conj:kk} for cliques.
In \cite{msss}, for $r > \ell \geq 3$, they employ an algorithm called \textsc{Asym-Edge-Col} which either produces a valid edge-colouring for $K_r$ and $K_{\ell}$ of $G$ (showing that $G \not\to (K_r, K_{\ell})$) or encounters an error.
Instead of assuming $G \to (K_r, K_{\ell})$, they assume algorithm \textsc{Asym-Edge-Col} encounters an error, and proceed with the standard approach from there.
One of the advantages of this approach is that it provides an algorithm for constructing a valid edge-colouring for $K_r$ and $K_{\ell}$, rather than just proving the existence of such a colouring. 

\subsection{On Conjecture~\ref{conj:aconj}} 
As mentioned earlier, we provide all but one step, Conjecture~\ref{conj:aconj}, of this approach. Let us consider how Conjecture~\ref{conj:aconj} relates to previous work on the $0$-statement of Conjecture~\ref{conj:kk}.
Firstly, Conjecture~\ref{conj:aconj} was implicitly proven for pairs of cliques in \cite{msss} and pairs of a clique and a cycle in \cite{lmms}.
More specifically, when $H_1$ and $H_2$ are both cliques (except when $H_1 = H_2 = K_3$)\footnote{The case $H_1 = H_2$ of Conjecture~\ref{conj:kk} is, of course, covered by Theorem~\ref{thm:rr}.},
the authors of \cite{msss} prove a slightly more general version of Conjecture~\ref{conj:aconj} (Lemma 8 in \cite{msss}) where $\hat{\mathcal{A}}(H_1, H_2, \varepsilon)$ is replaced with the set \[\mathcal{A}(H_1,H_2) := \{A \in \mathcal{C}(H_1, H_2) : m(A) \leq m_2(H_1, H_2) + 0.01 \land A \ \mbox{is $2$-connected}\}.\]
Note that the proof of Lemma 8 in \cite{msss} shows that $\mathcal{A}(H_1,H_2) \neq \emptyset$ for certain pairs of cliques $H_1$ and $H_2$. When $H_1$ is a clique, $H_2$ is a cycle and $H_1 \neq H_2$ (that is, excluding again the case when $H_1 = H_2 = K_3$), the proof of Lemma 3.3 in \cite{lmms} implies that there exists a constant $\varepsilon > 0$ such that $\hat{\mathcal{A}}(H_1, H_2, \varepsilon) = \emptyset$.

For reference, we note here the places in our proof of Theorem~\ref{thm:main} where we specifically need Conjecture~\ref{conj:aconj} to hold:

 \begin{itemize}
     \item the proof of Lemma~\ref{lemma:acolour};
     \item the proofs of Claims~\ref{claim:conclusion} and \ref{claim:conclusionsym};
     \item the definition of $\gamma = \gamma(H_1, H_2)$ in Section~\ref{sec:grow}.
 \end{itemize} 

\subsection{Proof sketch of Theorem~\ref{thm:main}}
Let us now proceed with describing the proof of Theorem~\ref{thm:main} in detail. 
In what follows, we write \emph{(Result~A; Result~B)} to mean that `Result~B in \cite{msss} fulfils the same role (in \cite{msss}) as Result~A does in our proof of Theorem~\ref{thm:main}'. 
This is to illustrate how we indeed provide every step bar one (Conjecture~\ref{conj:aconj}) of a proof of the $0$-statement of Conjecture~\ref{conj:kk}.

 Firstly, as in \cite{msss}, we give an algorithm \textsc{Asym-Edge-Col} that, assuming Conjecture~\ref{conj:aconj} holds, produces a valid edge-colouring for $H_1$ and $H_2$ of $G = G_{n,p}$ provided it does not encounter an error (Lemma~\ref{lemma:errorvalid}; Lemma 11).
Our aim then is to prove that \textsc{Asym-Edge-Col} does not encounter an error a.a.s.~(Lemma~\ref{lemma:noerror}; Lemma 12), that is, $G \not\to (H_1, H_2)$ a.a.s.
We split our proof of Lemma~\ref{lemma:noerror} into two cases: when $m_2(H_1) > m_2(H_2)$ and when $m_2(H_1) = m_2(H_2)$.

Suppose for a contradiction that \textsc{Asym-Edge-Col} encounters an error. Let $G' \subseteq G$ be the graph that \textsc{Asym-Edge-Col} got stuck on when it encountered this error.
In the $m_2(H_1) > m_2(H_2)$ case, we input $G'$ into an auxiliary algorithm \textsc{Grow} which always outputs a subgraph $F \subseteq G'$ (Claim~\ref{claim:grow}; Claim 13) belonging to a sufficiently small collection of non-isomorphic graphs $\mathcal{F}$. 
The definition of $\mathcal{F}$ will be such that with high probability no copy of any $F \in \mathcal{F}$ will be present in $G_{n,p}$, provided that $|\mathcal{F}|$ is sufficiently small. 

In order to show $|\mathcal{F}|$ is sufficiently small, we carefully analyse the possible outputs of \textsc{Grow}.
Assuming Conjecture~\ref{conj:aconj} holds, we show that only a constant number of graphs can be produced by \textsc{Grow} if one of two special cases occurs. 
If neither of these special cases occur, then, starting from a copy of $H_1$, in each step of \textsc{Grow} our subgraph $F$ is constructed iteratively by either (i) appending a copy of $H_1$ to $F$ or (ii) appending a `flower-like' structure to $F$, consisting of a central copy of $H_2$ with `petals' that are appended copies of $H_1$. We say an iteration is \emph{degenerate} if it is of type (i) or, loosely speaking, of type (ii) where `the flower is folded in on itself or into $F$'.
Otherwise an iteration is called \emph{non-degenerate}. Denote by $\lambda(F)$ the order of magnitude of the expected number of copies of $F$ in $G_{n,p}$ with $p = bn^{-1/m_2(H_1,H_2)}$.
Key to showing $|\mathcal{F}|$ is sufficiently small is proving that $\lambda(F)$ stays the same after a non-degenerate iteration (Claim~\ref{claim:non-degen}; Claim 14) and decreases by a \emph{constant} amount after a degenerate iteration (Claim~\ref{claim:degenfull}; Claim 15). 
Indeed, one of the termination conditions for \textsc{Grow} is that $\lambda(F) < -\gamma$ (where $\gamma = \gamma(H_1, H_2,\varepsilon) > 0$ is defined later in Section~\ref{sec:grow}, given $\varepsilon = \varepsilon(H_1, H_2) > 0$, the constant acquired from assuming Conjecture~\ref{conj:aconj} holds), that is, only a constant number of such \emph{degenerate} steps occur before \textsc{Grow} terminates (Claim~\ref{claim:q_1}; Claim 16).
Proving Claim~\ref{claim:degenfull} is the main work of this paper.
An important step in proving it is showing that if an iteration of type (ii) occurs where, loosely speaking, `the flower is folded in on itself', we get a helpful inequality comparing this iteration with a non-degenerate iteration (Lemma~\ref{lemma:21}; Lemma 21).
Indeed, the most novel work of this paper is the proof of Lemma~\ref{lemma:21}. 

The proof of Lemma~\ref{lemma:noerror} in the $m_2(H_1) = m_2(H_2)$ case is both similar and significantly simpler.
Notably, we use a different algorithm, \textsc{Grow-Alt}, to grow our subgraph $F \subseteq G'$.
Our analysis of \textsc{Grow-Alt} is much quicker than that of \textsc{Grow}, allowing us to easily prove a result analogous to Claim~\ref{claim:degenfull}.

\subsection{Differences between our work and \cite{msss}}
 As mentioned earlier, our approach to proving Theorem~\ref{thm:main} builds on the work of Marciniszyn, Skokan, Sp\"{o}hel and Steger in \cite{msss}. 
For readers familiar with \cite{msss}, we include the following list of differences between this paper and \cite{msss} (some of which we have already noted):

 \begin{itemize}
     \item We prove and employ a new result (Lemma~\ref{lemma:2connected}) concerning types of balancedness and $2$-connectivity;
     \item When $m_2(H_1) > m_2(H_2)$, we refine the proof in \cite{msss} to consider $\hat{\mathcal{A}}(H_1, H_2, \varepsilon)$ instead of $\mathcal{A}(H_1, H_2)$; 
     \item We generalise from considering triangle-sparse graphs to considering $(H_1, H_2)$-sparse graphs (see Section~\ref{sec:asymedgecol});
     \item Lemma~\ref{lemma:21} and its setup (see Sections~\ref{sec:claimdegenfull} and \ref{sec:lemma21}) are quite different to Lemma 21 and its setup in \cite{msss};
     \item Although Claim~\ref{claim:degen1} is analogous to Claim 19 in \cite{msss}, its proof is quite different, utilising the balancedness properties of $H_1$ and $H_2$;
     \item Although Claim~\ref{claim:degen2} is analogous to Claim 22 in \cite{msss}, our proof is slightly different, swapping the latter two steps of the proof of Claim 22 in order to apply our Lemma~\ref{lemma:21} in place of Lemma 21;
     \item To account for $H_1$ and $H_2$ possibly having less structure than cycles or cliques\footnote{In particular, for $i \in \{1,2\}$, $H_i$ may not have the property that when one removes any single edge (and no vertices) from $H_i$ one gets the same isomorphic graph, irrespective of which edge is removed.}, the statement of Claim~\ref{claim:polylog} and the proofs of Claims~\ref{claim:polylog} and \ref{claim:conclusion} differ somewhat from their counterparts (Claims 17 and 18) in \cite{msss};
     \item When $m_2(H_1) = m_2(H_2)$, we use a slightly different algorithm \textsc{Grow-Alt} (see Section~\ref{sec:altcase}) to algorithm \textsc{Grow}.
 \end{itemize}

\section{Organisation of paper}
The paper is organised as follows.
In Section~\ref{sec:notation}, we collect together notation, density measures and several useful results we will need. 
In Section~\ref{sec:asymedgecol}, we give our algorithm \textsc{Asym-Edge-Col} for producing a valid edge-colouring for $H_1$ and $H_2$ of $G = G_{n,p}$ provided it does not encounter an error (and Conjecture~\ref{conj:aconj} holds for $H_1$ and $H_2$).
In Sections~\ref{sec:grow}-\ref{sec:lemma21}, we prove that \textsc{Asym-Edge-Col} does not encounter an error a.a.s.~(Lemma~\ref{lemma:noerror}) in the case when $m_2(H_1) > m_2(H_2)$. 
In Section~\ref{sec:altcase}, we prove Lemma~\ref{lemma:noerror} in the case when $m_2(H_1) = m_2(H_2)$. 
In Section~\ref{sec:claimregular}, we prove Theorem~\ref{thm:regular}, before providing some concluding remarks in Section~\ref{sec:concludingremarks}.

\section{Notation, density measures and useful results}\label{sec:notation}

As far as possible we keep to the notation used in \cite{msss}.
Also, we repeat several definitions used earlier for ease of reference. 

Let $G = (V,E)$ be a graph. We denote the number of vertices in $G$ by $v(G) = v_G := |V(G)|$ and the number of edges in $G$ by $e(G) = e_G := |E(G)|$.
Moreover, for graphs $H_1$ and $H_2$ we let $v_1 := |V(H_1)|$, $e_1 := |E(H_1)|$, $v_2 := |V(H_2)|$ and $e_2 := |E(H_2)|$.  

Let $H$ be a graph. The most well-known density measure is \[d(H) := 
    \begin{cases}
        e_H/v_H         & \quad \text{if } v(H) \geq 1,\\
        0               & \quad \text{otherwise}.
    \end{cases}\]
Taking the maximum value of $d$ over all subgraphs $J \subseteq H$, we have the following measure \[m(H) := \max\{d(J): J \subseteq H\}.\]
(We say that a graph $H$ is \emph{balanced with respect to $d$}, or just \emph{balanced}, if we have $d(H) = m(H)$. 
Moreover, we say $H$ is \emph{strictly balanced} if for every proper subgraph $J \subsetneq H$, we have $d(J) < m(H)$.)

In \cite{rr1}, R\"{o}dl and Ruci\'{n}ski introduced the following so-called \emph{$2$-density measure}. \[d_2(H) := 
    \begin{cases}
        (e_H - 1)/(v_H - 2)         & \quad \text{if $H$ is non-empty with} \ v(H) \geq 3,\\
        1/2                         & \quad \text{if} \ H \cong K_2,\\    
        0                           & \quad \text{otherwise}.
    \end{cases}\]
As with $d$, we have an associated measure based on maximising $d_2$ over subgraphs of $H$: \[m_2(H) := \max\left\{d_2(J): J \subseteq H\right\}.\] 
Analogously to the notion of balancedness, we say that a graph $H$ is \emph{$2$-balanced} if $d_2(H) = m_2(H)$, and \emph{strictly $2$-balanced} if for all proper subgraphs $J \subsetneq H$, we have $d_2(J) < m_2(H)$.

Regarding asymmetric Ramsey properties, in \cite{kk}, Kohayakawa and Kreuter introduced the following generalisation of $d_2$. Let $H_1$ and $H_2$ be any graphs, and define \[d_2(H_1,H_2) := 
    \begin{cases}
        \frac{e_1}{v_1 - 2 + \frac{1}{m_2(H_2)}}         & \quad \text{if $H_2$ is non-empty and} \ v_1 \geq 2,\\
        0                         & \quad \text{otherwise}.
    \end{cases}\]
Similarly to before, we have the following measure based on maximising $d_2$ over all subgraphs $J \subseteq H_1$. \[m_2(H_1, H_2) := \max\left\{d_2(J,H_2): J \subseteq H_1\right\}.\]
We say that $H_1$ is \emph{balanced with respect to $d_2(\cdot, H_2)$} if we have $d_2(H_1, H_2) = m_2(H_1, H_2)$ and \emph{strictly balanced with respect to $d_2(\cdot, H_2)$} if for all proper subgraphs $J \subsetneq H_1$ we have $d_2(J, H_2) < m_2(H_1, H_2)$. 

Observe that $m_2(\cdot, \cdot)$ is not symmetric in both arguments. 
Recall Proposition~\ref{prop:m2h1h2}.

\begin{prop:m2h1h2}
Suppose that $H_1$ and $H_2$ are non-empty graphs with $m_2(H_1) \geq m_2(H_2)$. 
Then we have \[m_2(H_1) \geq m_2(H_1, H_2) \geq m_2(H_2).\]
Moreover, \[m_2(H_1) > m_2(H_1, H_2) > m_2(H_2) \ \mbox{whenever} \ m_2(H_1) > m_2(H_2).\]
\end{prop:m2h1h2}

Note that if $m_2(H_1) = m_2(H_2)$ and $H_1$ and $H_2$ are non-empty graphs, then $H_1$ cannot be strictly balanced with respect to $d_2(\cdot, H_2)$ unless $H_1 \cong K_2$. 
Indeed, otherwise, by Proposition~\ref{prop:m2h1h2} we would then have that \[m_2(H_2) = m_2(H_1, H_2) > d_2(K_2, H_2) = m_2(H_2).\] 
 
The following fact will be useful in the proofs of Lemmas~\ref{lemma:2connected} and \ref{lemma:21}.

\begin{fact}\label{fact:ineq}
For $a,c,C \in \mathbb{R}$ and $b,d > 0$, we have 
\[(i) \ \ \frac{a}{b} \leq C \ \land \ \frac{c}{d} \leq C \implies \frac{a+c}{b+d} \leq C \ \mbox{and} \ \ (ii) \ \ \frac{a}{b} \geq C \ \land \ \frac{c}{d} \ \geq C \implies \frac{a+c}{b+d} \geq C\] and similarly, if also $b > d$, 
\[(iii) \ \ \frac{a}{b} \leq C \ \land \ \frac{c}{d} \geq C \implies \frac{a-c}{b-d} \leq C \ \mbox{and} \ \ (iv) \ \ \frac{a}{b} \geq C \ \land \ \frac{c}{d} \ \leq C \implies \frac{a-c}{b-d} \geq C.\]
\end{fact}

The following result will be very useful for us, creating an important connection between types of balancedness and $2$-connectivity.

\begin{lem}\label{lemma:2connected}
Let $H_1$ and $H_2$ be graphs such that either (i) $m_2(H_1) > m_2(H_2) > 1$, $H_2$ is strictly $2$-balanced and $H_1$ is strictly balanced with respect to $d_2(\cdot, H_2)$; or (ii) $m_2(H_1) = m_2(H_2) > 1$ and $H_1$ and $H_2$ are both strictly $2$-balanced. Then $H_1$ and $H_2$ are both $2$-connected.
\end{lem}

\begin{proof}
By \cite[Lemma~3.3]{ns}, strictly $2$-balanced graphs are $2$-connected, hence $(ii)$ holds and $H_2$ is $2$-connected in case $(i)$. We now use a very similar method to that of the proof of Lemma~$3.3$ to show that $H_1$ is $2$-connected in case $(i)$. 
Since $H_1$ is strictly balanced with respect to $d_2(\cdot, H_2)$, we have that $H_1$ is connected.
Indeed, assume not. Let $H_1$ have $k \geq 2$ components and denote the number of vertices and edges in each component by $u_1, \ldots, u_k$ and $d_1, \ldots, d_k$, respectively. Then since $H_1$ is strictly balanced with respect to $d_2(\cdot, H_2)$, we must have that \[\frac{\sum_{i=1}^k d_i}{\sum_{i=1}^k u_i - 2 + \frac{1}{m_2(H_2)}} > \frac{d_1}{u_1 - 2 + \frac{1}{m_2(H_2)}}\] and \[\frac{\sum_{i=1}^k d_i}{\sum_{i=1}^k u_i - 2 + \frac{1}{m_2(H_2)}} > \frac{\sum_{i=2}^k d_i}{\sum_{i=2}^k u_i - 2 + \frac{1}{m_2(H_2)}}.\] Since $m_2(H_2) > 1$, by Fact~\ref{fact:ineq}(i) we get that \[\frac{\sum_{i=1}^k d_i}{\sum_{i=1}^k u_i - 2 + \frac{1}{m_2(H_2)}} \geq \frac{\sum_{i=1}^k d_i}{\sum_{i=1}^k u_i - 4 + \frac{2}{m_2(H_2)}} > \frac{\sum_{i=1}^k d_i}{\sum_{i=1}^k u_i - 2 + \frac{1}{m_2(H_2)}},\] a contradiction. 
 
Assume $H_1$ is not $2$-connected. Then there exists a cut vertex\footnote{That is, removing $v$ and its incident edges from $H_1$ produces a disconnected graph.} $v \in V(H_1)$.
Further, using Proposition~\ref{prop:m2h1h2} alongside that $H_1$ is strictly balanced with respect to $d_2(\cdot, H_2)$ and $m_2(H_1) > m_2(H_2) > 1$, we can show that $H_1$ does not contain any vertex of degree 1.
Indeed, otherwise $\frac{e_1 - 1}{v_1 - 3 + \frac{1}{m_2(H_2)}} > \frac{e_1}{v_1 - 2 + \frac{1}{m_2(H_2)}} = m_2(H_1, H_2)$, contradicting that $H_1$ is strictly balanced with respect to $d_2(\cdot, H_2)$. 
Thus there exist subgraphs $J_1$ and $J_2$ of $H_1$ such that $|E(J_1)|, |E(J_2)| \geq 1$, $J_1 \cup J_2 = H_1$ and $V(J_1) \cap V(J_2) = \{v\}$. Using Fact~\ref{fact:ineq}(i) and that $H_1$ is strictly balanced with respect to $d_2(\cdot, H_2)$, we have that 
\begin{eqnarray*}e_1 = e_{J_1} + e_{J_2} & < & m_2(H_1, H_2)\left(v_{J_1} - 2 + \frac{1}{m_2(H_2)} + v_{J_2} - 2 + \frac{1}{m_2(H_2)}\right) \\ & = & m_2(H_1,H_2)\left(v_1 - 3 + \frac{2}{m_2(H_2)}\right).\end{eqnarray*}
However, since $m_2(H_2) > 1$ we also have that \[\frac{e_1}{v_1 - 3 + \frac{2}{m_2(H_2)}} > \frac{e_1}{v_1 - 2 + \frac{1}{m_2(H_2)}} = m_2(H_1, H_2),\] contradicting the inequality above. Hence $H_1$ is $2$-connected.\end{proof}

\section{Algorithm for computing valid edge-colourings: {\sc Asym-Edge-Col}}\label{sec:asymedgecol}

To prove Theorem~\ref{thm:main}, we can clearly assume $H_1$ and $H_2$ are non-empty graphs satisfying the criteria of Conjecture~\ref{conj:aconj} and that Conjecture~\ref{conj:aconj} itself holds. 
Suppose $G = G_{n,p}$ and $p \leq bn^{-1/m_2(H_1, H_2)}$ where $b$ will be a small constant defined later. 
As noted earlier, to prove Conjecture~\ref{conj:kk} we can show that a.a.s.~$G$ has a valid edge-colouring for $H_1$ and $H_2$.
We construct our valid edge-colouring using an algorithm \textsc{Asym-Edge-Col} (see Figure~\ref{asymedgecolfig}). 
In order to state the algorithm succinctly, we need to define a considerable amount of notation, almost all of which we keep very similar to that in \cite{msss}. 

Recall Definition~\ref{def:c*}. In particular, recall that for any graph $G$ we define the families \[\mathcal{R}_G := \{R \subseteq G : R \cong H_1\} \ \mbox{and}\ \mathcal{L}_G := \{L \subseteq G : L \cong H_2\}\] of all copies of $H_1$ and $H_2$ in $G$, respectively.
Also, recall the family \[\mathcal{L}^*_G := \{L \in \mathcal{L}_G : \forall e \in E(L) \ \exists R \in \mathcal{R}_G \ \mbox{s.t.} \ E(L) \cap E(R) = \{e\}\}.\]
We highlight here that if $E(L) \cap E(R) = \{e\}$ for some $L \in \mathcal{L}_G$ and $R \in \mathcal{R}_G$ then it is still possible that $|V(L) \cap V(R)| > 2$.

Recall Definition~\ref{def:a*}. Intuitively, the graphs in $\hat{\mathcal{A}}$ are the building blocks of the graphs $\hat{G}$ which may remain after the edge deletion process in \textsc{Asym-Edge-Col} (described later). 

For any graph $G$, define \[\mathcal{S}_G := \{S \subseteq G: S \cong A \in \hat{\mathcal{A}} \land \nexists S' \supset S \ \mbox{with} \ S' \subseteq G, \ S' \cong A' \in \hat{\mathcal{A}}\},\] that is, the family $\mathcal{S}_G$ contains all maximal subgraphs of $G$ isomorphic to a member of $\hat{\mathcal{A}}$. 
Hence, there are no two members $S_1, S_2 \in \mathcal{S}_G$ such that $S_1 \subsetneq S_2$. For any edge $e \in E(G)$, let \[\mathcal{S}_G(e) := \{S \in \mathcal{S}_G : e \in E(S)\}.\]
\begin{define}\label{def:agraph}
    We call $G$ an \emph{$\hat{\mathcal{A}}$-graph} if, for all $e \in E(G)$, we have \[|\mathcal{S}_G(e)| = 1.\] 
\end{define}
In particular, an $\hat{\mathcal{A}}$-graph is an edge-disjoint union of graphs from $\hat{\mathcal{A}}$.
In an $\hat{\mathcal{A}}$-graph $G$, a copy of $H_1$ or $H_2$ can be a subgraph of $G$ in two particular ways: either it is a subgraph of an $S \in \mathcal{S}_G$ or it is a subgraph with edges in at least two different graphs from $\mathcal{S}_G$. 
The former we call \emph{trivial} copies of $H_1$ and $H_2$, and we define \[\mathcal{T}_G := \left\{T \subseteq G : (T \cong H_1 \vee T \cong H_2) \land \left|\bigcup_{\substack{e \in E(T)}}\mathcal{S}_G(e)\right| \geq 2\right\}\] to be the family of all \emph{non-trivial} copies of $H_1$ and $H_2$ in $G$.
\begin{define}\label{def:h1h2sparse}
    We say that a graph $G$ is \emph{$(H_1, H_2)$-sparse} if $\mathcal{T}_G = \emptyset$.\label{page:h1h2sparse}
\end{define} 
Our next lemma asserts that $(H_1, H_2)$-sparse $\hat{\mathcal{A}}$-graphs are easily colourable, provided Conjecture~\ref{conj:aconj} holds.

\begin{lem}\label{lemma:acolour}
There exists a procedure \textsc{A-Colour} that returns for any $(H_1, H_2)$-sparse $\hat{\mathcal{A}}$-graph $G$ a valid edge-colouring for $H_1$ and $H_2$.
\end{lem}

\begin{proof} 
By Conjecture~\ref{conj:aconj}, there exists a valid edge-colouring for $H_1$ and $H_2$ of every $A \in \hat{\mathcal{A}}$. Using this we define a procedure \textsc{A-Colour}$(G)$ as follows:
Assign a valid edge-colouring for $H_1$ and $H_2$ to every subgraph $S \in \mathcal{S}_G$ locally, that is, regardless of the structure of $G$. 
Since $G$ is an $(H_1, H_2)$-sparse $\hat{\mathcal{A}}$-graph, we assign a colour to each edge of $G$ without producing a red copy of $H_1$ or a blue copy of $H_2$, and the resulting colouring is a valid edge-colouring for $H_1$ and $H_2$ of $G$.\end{proof}

Note that we did not use that $\hat{\mathcal{A}}$ is finite, as given by Conjecture~\ref{conj:aconj}, in our proof of Lemma~\ref{lemma:acolour}, only that `every graph in $\hat{\mathcal{A}}$ has a valid edge-colouring for $H_1$ and $H_2$'. 
The finiteness of $\hat{\mathcal{A}}$ will be essential later for the proofs of Claims~\ref{claim:conclusion} and \ref{claim:conclusionsym}.

\begin{figure}
\begin{algorithmic}[1]
\Procedure{\sc Asym-Edge-Col}{$G = (V,E)$}
    \State $s\gets$ {\sc empty-stack}()
    \State $E'\gets E$
    \State $\mathcal{L}\gets \mathcal{L}_G$
    \While{$G' = (V, E')$ is not $(H_1, H_2)$-sparse or not an $\hat{\mathcal{A}}$-graph}\label{line:while1start}
            \If{$\exists e \in E' \ \mbox{s.t.} \ \nexists (L,R) \in \mathcal{L} \times \mathcal{R}_{G'}: E(L) \cap E(R) = \{e\}$}\label{line:edgeremoval}
                \ForAll {$L \in \mathcal{L}: e \in E(L)$}\label{line:Lall1}
                    \State $s$.{\sc push}($L$)\label{line:Lpush1}
                \State $\mathcal{L}$.{\sc remove}($L$) \label{line:Lremove1}
                \EndFor \label{line:Lend1}
            \State $s$.{\sc push}($e$)\label{line:edgepush}
            \State $E'$.{\sc remove}($e$)\label{line:edgeremove}
            \Else
                \If{$\exists L \in \mathcal{L}: \ \exists e \in E(L) \ \mbox{s.t.} \ \nexists R \in \mathcal{R}_{G'} \ \mbox{with} \ E(L) \cap E(R) = \{e\} $}\label{line:L*check}
                    \State $s$.{\sc push}($L$)\label{line:Lpush2}
                    \State $\mathcal{L}$.{\sc remove}($L$)\label{line:Lend2}
                \Else
                    \State {\bf error} {``stuck"}\label{line:error}
                \EndIf
            \EndIf
    \EndWhile\label{line:while1end}
    \State {\sc A-colour}($G' = (V,E')$)\label{line:acolourcall}
    \While{$s \neq \emptyset$}\label{line:while2start}
        \If{$s$.{\sc top}() is an edge}\label{line:edgeif}
            \State $e \gets s$.{\sc pop}()\label{line:edgepop}
            \State $E'$.{\sc add}($e$)\label{line:edgeadd}
            \State $e$.{\sc set-colour}(blue)\label{line:edgeblue}
        \Else
            \State $L \gets s$.{\sc pop}()\label{line:Lpop}
            \If{$L$ is entirely blue}\label{line:colourswapstart}
                \State $f \gets$ any $e \in E(L)$ s.t. $\nexists R \in \mathcal{R}_{G'}: E(L) \cap E(R) = \{e\}$\label{line:getf}
                \State $f$.{\sc set-colour}(red)\label{line:colourswapend}
                
            \EndIf
        \EndIf
    \EndWhile \label{line:while2end}
\EndProcedure
\end{algorithmic}
\caption{The implementation of algorithm \textsc{Asym-Edge-Col}.}\label{asymedgecolfig}
\end{figure}

Now let us describe the algorithm \textsc{Asym-Edge-Col} which if successful outputs a valid edge-colouring of $G$. 
In \textsc{Asym-Edge-Col}, edges are removed from and then inserted back into a working copy $G'~=~(V, E')$ of $G$.
Each edge is removed in the first while-loop only when it is not the unique intersection of the edge sets of some copy of $H_1$ and some copy of $H_2$ in $G'$ (line~\ref{line:edgeremoval}). 
It is then `pushed\footnote{For clarity, by `push' we mean that the object is placed on the top of the stack $s$.}' onto a stack $s$ such that when we reinsert edges (in reverse order) in the second while-loop we can colour them to construct a valid edge-colouring for $H_1$ and $H_2$ of $G$;
if at any point $G'$ is an $(H_1, H_2)$-sparse $\hat{\mathcal{A}}$-graph, then we combine the colouring of these edges with a valid edge-colouring for $H_1$ and $H_2$ of $G'$ provided by \textsc{A-Colour}. 
We also keep track of the copies of $H_2$ in $G$ and push abstract representations of some of them (or all of them if $G'$ is never an $(H_1, H_2)$-sparse $\hat{\mathcal{A}}$-graph during \textsc{Asym-Edge-Col}) onto $s$ (lines~\ref{line:Lpush1} and \ref{line:Lpush2}) to be used later in the colour swapping stage of the second while-loop (lines~\ref{line:colourswapstart}-\ref{line:colourswapend}). 

Let us consider algorithm \textsc{Asym-Edge-Col} in detail. In line~\ref{line:while1start}, we check whether $G'$ is an $(H_1, H_2)$-sparse $\hat{\mathcal{A}}$-graph or not.
If not, then we enter the first while-loop. In line~\ref{line:edgeremoval}, we choose an edge $e$ which is not the unique intersection of the edge sets of some copy of $H_1$ and some copy of $H_2$ in $G'$ (if such an edge $e$ exists).
Then in lines~\ref{line:Lall1}-\ref{line:edgeremove} we push each copy of $H_2$ in $G'$ that contains $e$ onto $s$ before pushing $e$ onto $s$ as well.
Now, if every edge $e \in E'$ is the unique intersection of the edge sets of some copy of $H_1$ and some copy of $H_2$ in $G'$,
then we push onto $s$ a copy $L$ of $H_2$ in $G'$ which contains an edge that is not the unique intersection of the edge set of $L$ and the edge set of some copy of $H_1$ in $G'$.
If no such copies $L$ of $H_2$ exist, then the algorithm has an error in line~\ref{line:error}.
If \textsc{Asym-Edge-Col} does not run into an error, then we enter the second while-loop with input $G'$.
Observe that $G'$ is either the empty graph on vertex set $V$ or some $(H_1, H_2)$-sparse $\hat{\mathcal{A}}$-graph.
By Lemma~\ref{lemma:acolour}, $G'$ has a valid edge-colouring for $H_1$ and $H_2$.
The second while-loop successively removes edges (line~\ref{line:edgepop}) and copies of $L$ (line~\ref{line:Lpop}) from $s$ in the reverse order in which they were added onto $s$, with the edges added back into $E'$.
Each time an edge is added back it is coloured blue, and if a monochromatic blue copy $L$ of $H_2$ is constructed, we make  one of the edges of $L$ red (lines~\ref{line:colourswapstart}-\ref{line:colourswapend}).
This colouring process is then repeated until we have a valid edge-colouring for $H_1$ and $H_2$ of $G$. 

The following lemma confirms that our colouring process in the second while-loop produces a valid edge-colouring for $H_1$ and $H_2$ of $G$. 

\begin{lem}\label{lemma:errorvalid}
Algorithm \textsc{Asym-Edge-Col} either terminates with an error in line~\ref{line:error} or finds a valid edge-colouring for $H_1$ and $H_2$ of $G$.
\end{lem}

\begin{proof}
Our proof is almost identical to the proof of Lemma 11 in \cite{msss}. We include it here for completeness.

Let $G^*$ denote the argument in the call to \textsc{A-Colour} in line~\ref{line:acolourcall}.
By Lemma~\ref{lemma:acolour}, there is a valid edge-colouring for $H_1$ and $H_2$ of $G^*$. 
It remains to show that no forbidden monochromatic copies of $H_1$ or $H_2$ are created when this colouring is extended to a colouring of $G$ in lines~\ref{line:while2start}-\ref{line:while2end}. 

Firstly, we argue that the algorithm never creates a blue copy of $H_2$. Observe that \emph{every} copy of $H_2$ that does not lie entirely in $G^*$ is pushed on the stack in the first while-loop (lines~\ref{line:while1start}-\ref{line:while1end}). 
Therefore, in the execution of the second loop, the algorithm checks the colouring of every such copy.
By the order of the elements on the stack, each such test is performed only after all edges of the corresponding copy of $H_2$ were inserted and coloured. 
For every blue copy of $H_2$, one particular edge $f$ (see line~\ref{line:getf}) is recoloured to red. Since red edges are never flipped back to blue, no blue copy of $H_2$ can occur.

We need to show that the edge $f$ in line~\ref{line:getf} always exists. 
Since the second loop inserts edges into $G'$ in the reverse order in which they were deleted during the first loop, when we select $f$ in line~\ref{line:getf}, $G'$ has the same structure as at the time when $L$ was pushed on the stack. 
This happened either in line~\ref{line:Lpush1} when there exists no copy of $H_1$ in $G'$ whose edge set intersects with $L$ on some particular edge $e \in E(L)$, or in line~\ref{line:Lpush2} when $L$ is not in $\mathcal{L}^*_{G'}$ due to the if-clause in line~\ref{line:L*check}. 
In both cases we have $L \notin \mathcal{L}^*_{G'}$, and hence there exists an edge $e \in E(L)$ such that the edge sets of all copies of $H_1$ in $G'$ do not intersect with $L$ exactly in $e$.

It remains to prove that changing the colour of some edges from blue to red by the algorithm never creates an entirely red copy of $H_1$.
By the condition on $f$ in line~\ref{line:getf} of the algorithm, at the moment $f$ is recoloured there exists no copy of $H_1$ in $G'$ whose edge set intersects $L$ exactly in $f$.
So there is either no copy of $H_1$ containing $f$ at all, or every such copy contains also another edge from $L$.
In the latter case, those copies cannot become entirely red since $L$ is entirely blue.\end{proof}

To prove Theorem~\ref{thm:main}, it now suffices to prove the following lemma.

\begin{lem}\label{lemma:noerror}
There exists a constant $b = b(H_1,H_2) > 0$ such that for $p \leq bn^{-1/m_2(H_1,H_2)}$ algorithm \textsc{Asym-Edge-Col} terminates on $G_{n,p}$ without error a.a.s.
\end{lem}
We split our proof of Lemma~\ref{lemma:noerror} into two cases: (1) when $m_2(H_1) > m_2(H_2)$ and (2) when $m_2(H_1) = m_2(H_2)$.
Notice that this accords with our definition of $\hat{\mathcal{A}}$. 

\section{Case~1: $m_2(H_1) > m_2(H_2)$.}\label{sec:grow}

We will prove Case~1 of Lemma~\ref{lemma:noerror} using an auxiliary algorithm \textsc{Grow} (see Figure~\ref{growfig}).
If \textsc{Asym-Edge-Col} has an error, then \textsc{Grow} computes a subgraph $F \subseteq G$ which is either too large in size or too dense to appear in $G_{n,p}$ a.a.s.~(with $p$ as in Lemma~\ref{lemma:noerror}).
Indeed, letting $\mathcal{F}$ be the class of all graphs that can possibly be returned by \textsc{Grow}, we will show that the expected number of copies of graphs from $\mathcal{F}$ contained in $G_{n,p}$ is $o(1)$, which with Markov's inequality implies that $G_{n,p}$ a.a.s.~contains no graph from $\mathcal{F}$. 
This in turn implies Lemma~\ref{lemma:noerror} by contradiction. Note that algorithm \textsc{Grow} is only used for proving Lemma~\ref{lemma:noerror} and hence does not add anything on to the run-time of \textsc{Asym-Edge-Col}. 

To state \textsc{Grow} we require the following definitions. Let \begin{equation}\label{eq:gamma}\gamma = \gamma(H_1,H_2) := \frac{1}{m_2(H_1, H_2)} - \frac{1}{m_2(H_1, H_2) + \varepsilon(H_1, H_2)} > 0,\end{equation} where $\varepsilon(H_1, H_2)$ is the constant in Conjecture~\ref{conj:aconj}.
Recall that for any graph $F$, we have \[\lambda(F) := v(F) - \frac{e(F)}{m_2(H_1,H_2)}\] and that 
this definition is motivated by the fact that the expected number of copies of $F$ in $G_{n,p}$ with $p = bn^{-1/m_2(H_1,H_2)}$ has order of magnitude \[n^{v(F)}p^{e(F)} = b^{e(F)}n^{\lambda(F)}.\] Also, recall that \[\mathcal{T}_G := \left\{T \subseteq G : (T \cong H_1 \vee T \cong H_2) \land \left|\bigcup_{\substack{e \in E(T)}}\mathcal{S}_G(e)\right| \geq 2\right\}.\]
For any graph $F$ and edge $e \in E(F)$, we say that $e$ is \emph{eligible for extension in \textsc{Grow}} if it satisfies \[\nexists L \in \mathcal{L}^*_F \ \mbox{s.t.} \ e \in E(L),\] and observe that $F$ is in $\mathcal{C}^*$ (see Definition~\ref{def:c*}) if and only if it contains no edge that is eligible for extension in \textsc{Grow}.

Algorithm \textsc{Grow} has as input the graph $G' \subseteq G$ that \textsc{Asym-Edge-Col} got stuck on. 
Let us consider the properties of $G'$ when \textsc{Asym-Edge-Col} got stuck. Because the condition in line~\ref{line:edgeremoval} of \textsc{Asym-Edge-Col} fails, $G'$ is in the family $\mathcal{C}$, where we recall \[\mathcal{C} = \mathcal{C}(H_1, H_2) := \{G = (V,E) : \forall e \in E\ \exists(L,R) \in \mathcal{L}_G \times \mathcal{R}_G\ \mbox{s.t.}\ E(L) \cap E(R) = \{e\}\}.\]
In particular, every edge of $G'$ is contained in a copy $L \in \mathcal{L}_{G'}$ of $H_2$, and, because the condition in line~\ref{line:L*check} fails, we can assume in addition that $L$ belongs to $\mathcal{L}^*_{G'}$.
Hence, $G'$ is actually in the family $\mathcal{C}^* = \mathcal{C}^*(H_1, H_2)$ where we recall \[\mathcal{C}^* = \mathcal{C}^*(H_1, H_2) := \{G = (V,E): \forall e \in E \ \exists L \in \mathcal{L}^*_{G} \ \mbox{s.t.} \ e \in E(L)\}.\]
Lastly, $G'$ is not $(H_1, H_2)$-sparse or not an $\hat{\mathcal{A}}$-graph because \textsc{Asym-Edge-Col} ended with an error.

We now outline algorithm \textsc{Grow}. Firstly, \textsc{Grow} checks whether either of two special cases occur (lines~\ref{line:growspecial1}-\ref{line:endif}). 
The first case corresponds to when $G'$ is an $\hat{\mathcal{A}}$-graph, which is not $(H_1, H_2)$-sparse as it is a graph that \textsc{Asym-Edge-Col} got stuck on.
The second case happens if there are $2$ graphs in $\mathcal{S}_{G'}(e)$ that overlap in (at least) the edge $e$. 
The outputs of these two special cases are graphs which the while loop of \textsc{Grow} could get stuck on. Indeed, if neither of these cases happen, then in line~\ref{line:growanye} we can choose an edge $e$ that does not belong to any graph in $\mathcal{S}_G$.
Crucially, algorithm \textsc{Grow} chooses a graph $R \in \mathcal{R}_{G'}$  which contains such an edge $e$ and makes it the seed $F_0$ for a growing procedure (line~\ref{line:growanyr}). This choice of $F_0$ will allow us to conclude later that there always exists an edge eligible for extension in \textsc{Grow} (see the proof of Claim~\ref{claim:grow}), that is, the while loop of \textsc{Grow} operates as desired and doesn't get stuck. 

 \begin{figure}
\begin{algorithmic}[1]
\Procedure{Grow}{$G' = (V,E)$}
    \If {$\forall e \in E: |\mathcal{S}_{G'}(e)| = 1$}\label{line:growspecial1}
        \State $T \gets$ any member of $\mathcal{T}_{G'}$\label{line:growt}
        \State \Return $\bigcup_{e \in E(T)}\mathcal{S}_{G'}(e)$\label{line:growspecial1end}
    \EndIf
    \If{$\exists e \in E : |\mathcal{S}_{G'}(e)| \geq 2$}\label{line:growspecial2}
        \State $S_1, S_2 \gets$ any two distinct members of $\mathcal{S}_{G'}(e)$\label{line:growanytwo}
        \State \Return $S_1 \cup S_2$\label{line:growspecial2end}
    \EndIf\label{line:endif}
    \State $e \gets$ any $e \in E : |\mathcal{S}_{G'}(e)| = 0$\label{line:growanye}
    \State $F_0 \gets$ any $R \in \mathcal{R}_{G'}:e \in E(R)$\label{line:growanyr}
    \State $i \gets 0$\label{line:growito0}
    \While {$(i < \ln(n)) \land (\forall \tilde{F} \subseteq F_i : \lambda(\tilde{F}) > -\gamma)$}\label{line:growwhileconditions}
        \If {$\exists R \in \mathcal{R}_{G'}\setminus \mathcal{R}_{F_i} : |V(R) \cap V(F_i)| \geq 2$}\label{line:growRline}
            \State $F_{i+1} \gets F_i \cup R$\label{line:growR-F_i+1}
        \Else   
            \State $e \gets \textsc{Eligible-Edge}(F_i)$\label{line:eligible-edge}
            \State $F_{i+1} \gets \textsc{Extend-L}(F_i, e, G')$\label{line:F_i+1getsextend}
        \EndIf
        \State $i \gets i + 1$
    \EndWhile
    \If {$i \geq \ln(n)$}
        \State \Return{$F_i$}\label{line:growreturnfi}
    \Else
        \State \Return{\textsc{Minimising-Subgraph}($F_i$)}\label{line:growreturnminsub}
    \EndIf
    
\EndProcedure
\end{algorithmic}\smallskip\smallskip

\begin{algorithmic}[1]
\Procedure{Extend-L}{$F,e,G'$}
    \State $L \gets$ any $L \in \mathcal{L}^*_{G'}$: $e \in E(L)$\label{line:linl*}
    \State $F' \gets F \cup L$\label{line:f'fl}
        \ForAll{$e' \in E(L)\setminus E(F)$}\label{line:e'inL}
                \State $R_{e'} \gets$ any $R \in \mathcal{R}_{G'} : E(L) \cap E(R) = \{e'\}$\label{line:re'}
                \State $F' \gets F' \cup R_{e'}$\label{line:f'f're'}
        \EndFor
        \State \Return {$F'$}
\EndProcedure 

\end{algorithmic}
\caption{The implementation of algorithm \textsc{Grow}.}\label{growfig}
\end{figure}
In each iteration $i$ of the while-loop, the growing procedure extends $F_i$ to $F_{i+1}$ in one of two ways.
The first (lines~\ref{line:growRline}-\ref{line:growR-F_i+1}) is by attaching a copy of $H_1$ in $G'$ that intersects $F_i$ in at least two vertices but is not contained in $F_i$.
The second is more involved and begins with calling a function \textsc{Eligible-Edge} which maps $F_i$ to an edge $e \in E(F_i)$ which is eligible for extension in \textsc{Grow} (we will show that such an edge always exists).
Importantly, \textsc{Eligible-Edge} selects this edge $e$ to be \textit{unique up to isomorphism of $F_i$}, that is, for any two isomorphic graphs $F$ and $F'$, there exists an isomorphism $\phi$ with $\phi(F) = F'$ such that \[\phi(\textsc{Eligible-Edge}(F)) = \textsc{Eligible-Edge}(F').\]
In particular, our choice of $e$ depends only on $F_i$ and not on the surrounding graph $G'$ or any previous graph $F_j$ with $j < i$ (indeed, there may be many ways that \textsc{Grow} could construct a graph isomorphic to $F_i$).
One could implement \textsc{Eligible-Edge} by having an enormous table of representatives for all isomorphism classes of graphs with up to $n$ vertices.
Since we do not care about complexity here, and only want to show the existence of certain structures in $G'$, the time \textsc{Eligible-Edge} would take to be implemented is unimportant.
What is important is that \textsc{Eligible-Edge} does not itself increase the number of graphs $F$ that \textsc{Grow} can output. 

Once we have our edge $e \in E(F_i)$ eligible for extension in \textsc{Grow}, we apply a procedure called \textsc{Extend-L} which attaches a graph $L \in \mathcal{L}^*_{G'}$ that contains $e$ to $F_i$ (line~\ref{line:F_i+1getsextend}).
We then attach to each new edge $e' \in E(L)\setminus E(F_i)$ a graph $R_{e'} \in \mathcal{R}_{G'}$ such that $E(L) \cap E(R_{e'}) = \{e'\}$ (lines~\ref{line:e'inL}-\ref{line:f'f're'} of \textsc{Extend-L}).
(We will show later that such a graph $L$ and graphs $R_{e'}$ exist and that $E(L)\setminus E(F_i)$ is non-empty.) The algorithm comes to an end when either $i \geq \ln(n)$ or $\lambda(\tilde{F}) \leq -\gamma$ for some subgraph $\tilde{F} \subseteq F_i$.
In the former, the algorithm returns $F_i$ (line~\ref{line:growreturnfi}); in the latter, the algorithm returns a subgraph $\tilde{F} \subseteq F_i$ that minimises $\lambda(\tilde{F})$ (line~\ref{line:growreturnminsub}).
For each graph $F$, the function \textsc{Minimising-Subgraph($F$)} returns such a minimising subgraph that is \textit{unique up to isomorphism}.
Once again, this is to ensure that \textsc{Minimising-Subgraph($F$)} does not itself artificially increase the number of graphs that \textsc{Grow} can output.
As with function \textsc{Eligible-Edge}, one could implement \textsc{Minimising-Subgraph} using an enormous look-up table.

\definecolor{aqaqaq}{rgb}{0.6274509803921569,0.6274509803921569,0.6274509803921569}
\definecolor{yqyqyq}{rgb}{0.5019607843137255,0.5019607843137255,0.5019607843137255}
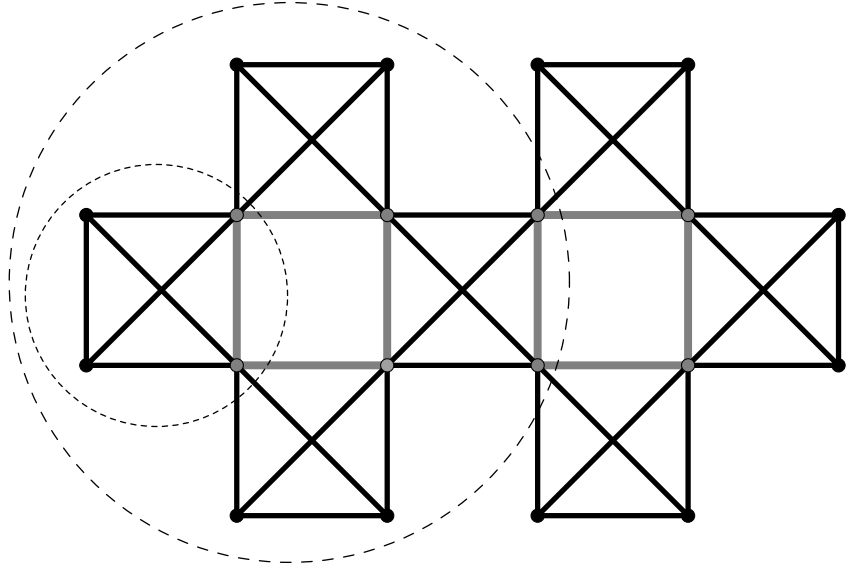
\begin{figure}[!ht]
\begin{center}
\begin{tikzpicture}[line cap=round,line join=round,>=triangle 45,x=1cm,y=1cm]
\draw [line width=3pt,color=yqyqyq] (4,6)-- (4,4);
\draw [line width=3pt,color=yqyqyq] (4,4)-- (6,4);
\draw [line width=3pt,color=yqyqyq] (6,4)-- (6,6);
\draw [line width=3pt,color=yqyqyq] (6,6)-- (4,6);
\draw [line width=3pt,color=yqyqyq] (8,4)-- (10,4);
\draw [line width=3pt,color=yqyqyq] (10,4)-- (10,6);
\draw [line width=3pt,color=yqyqyq] (10,6)-- (8,6);
\draw [line width=3pt,color=yqyqyq] (8,6)-- (8,4);
\draw [line width=2pt] (2,6)-- (4,4);
\draw [line width=2pt] (2,4)-- (4,6);
\draw [line width=2pt] (4,6)-- (6,8);
\draw [line width=2pt] (4,8)-- (6,6);
\draw [line width=2pt] (4,4)-- (6,2);
\draw [line width=2pt] (4,2)-- (6,4);
\draw [line width=2pt] (6,6)-- (8,4);
\draw [line width=2pt] (6,4)-- (8,6);
\draw [line width=2pt] (8,2)-- (10,4);
\draw [line width=2pt] (8,4)-- (10,2);
\draw [line width=2pt] (10,4)-- (12,6);
\draw [line width=2pt] (10,6)-- (12,4);
\draw [line width=2pt] (8,8)-- (10,6);
\draw [line width=2pt] (8,6)-- (10,8);
\draw [line width=2pt] (4,6)-- (4,8);
\draw [line width=2pt] (4,8)-- (6,8);
\draw [line width=2pt] (6,8)-- (6,6);
\draw [line width=2pt] (2,6)-- (4,6);
\draw [line width=2pt] (2,6)-- (2,4);
\draw [line width=2pt] (2,4)-- (4,4);
\draw [line width=2pt] (4,4)-- (4,2);
\draw [line width=2pt] (4,2)-- (6,2);
\draw [line width=2pt] (6,2)-- (6,4);
\draw [line width=2pt] (6,4)-- (8,4);
\draw [line width=2pt] (6,6)-- (8,6);
\draw [line width=2pt] (8,2)-- (8,4);
\draw [line width=2pt] (8,2)-- (10,2);
\draw [line width=2pt] (10,2)-- (10,4);
\draw [line width=2pt] (10,4)-- (12,4);
\draw [line width=2pt] (12,4)-- (12,6);
\draw [line width=2pt] (12,6)-- (10,6);
\draw [line width=2pt] (10,6)-- (10,8);
\draw [line width=2pt] (8,8)-- (10,8);
\draw [line width=2pt] (8,8)-- (8,6);
\draw [line width=0.5pt,dash pattern=on 2pt off 2pt] (2.9320423003311626,4.929673470131106) circle (1.7432259512012866cm);
\draw [line width=0.5pt,dash pattern=on 4pt off 4pt] (4.699041188101171,5.102672795114692) circle (3.72379990891749cm);
\begin{scriptsize}
\draw [fill=yqyqyq] (4,6) circle (2.5pt);
\draw [fill=yqyqyq] (4,4) circle (2.5pt);
\draw [fill=aqaqaq] (6,4) circle (2.5pt);
\draw [fill=yqyqyq] (6,6) circle (2.5pt);
\draw [fill=black] (2,6) circle (2.5pt);
\draw [fill=black] (2,4) circle (2.5pt);
\draw [fill=black] (4,8) circle (2.5pt);
\draw [fill=black] (6,8) circle (2.5pt);
\draw [fill=black] (4,2) circle (2.5pt);
\draw [fill=black] (6,2) circle (2.5pt);
\draw [fill=yqyqyq] (8,4) circle (2.5pt);
\draw [fill=yqyqyq] (8,6) circle (2.5pt);
\draw [fill=yqyqyq] (10,4) circle (2.5pt);
\draw [fill=yqyqyq] (10,6) circle (2.5pt);
\draw [fill=black] (8,8) circle (2.5pt);
\draw [fill=black] (10,8) circle (2.5pt);
\draw [fill=black] (12,6) circle (2.5pt);
\draw [fill=black] (12,4) circle (2.5pt);
\draw [fill=black] (8,2) circle (2.5pt);
\draw [fill=black] (10,2) circle (2.5pt);
\end{scriptsize}
\end{tikzpicture}
\vspace{0.4cm}
\caption{A graph $F_2$ resulting from two non-degenerate iterations for $H_1 = K_4$ and $H_2 = C_4$. The two central copies of $H_2$ are shaded.}
\end{center}
\end{figure}

We will now argue that \textsc{Grow} terminates without error, that is, \textsc{Eligible-Edge} always finds an edge eligible for extension in \textsc{Grow} and all `any'-assignments in \textsc{Grow} and \textsc{Extend-L} are always successful. 

\begin{claim}\label{claim:grow}
Algorithm \textsc{Grow} terminates without error on any input graph $G' \in \mathcal{C}^*$ that is not  $(H_1, H_2)$-sparse or not an $\hat{\mathcal{A}}$-graph.\footnote{See Definitions~\ref{def:agraph} and \ref{def:h1h2sparse}.} Moreover, for every iteration $i$ of the while-loop, we have $e(F_{i+1}) > e(F_i)$.
\end{claim}

\begin{proof}
Our proof is very similar to the proof of Claim 13 in \cite{msss}. 

We first show that the special cases in lines~\ref{line:growspecial1}-\ref{line:endif} always function as desired. 
The first case occurs if and only if $G'$ is an $\hat{\mathcal{A}}$-graph. By assumption, $G'$ is not $(H_1, H_2)$-sparse, hence the family $\mathcal{T}_{G'}$ is not empty. 
Hence the assignment in line~\ref{line:growt} is successful. Clearly, the assignment in line~\ref{line:growanytwo} is always successful due to the if-condition in line~\ref{line:growspecial2}.

One can also easily see that the assignments in lines~\ref{line:growanye} and \ref{line:growanyr} are successful. 
Indeed, neither of the two special cases occur so we must have an edge $e \in E$ that is not contained in any $S \in \mathcal{S}_{G'}$. 
Also, there must exist a member of $\mathcal{R}_{G'}$ that contains $e$ because $G'$ is a member of $\mathcal{C}^* \subseteq \mathcal{C}$.

Next, we show that the call to \textsc{Eligible-Edge} in line~\ref{line:eligible-edge} is always successful. Recall \eqref{eq:gamma} on page \pageref{eq:gamma}.
Indeed, suppose for a contradiction that no edge in $F_i$ is eligible for extension in \textsc{Grow} for some $i \geq 0$. Then every edge $e \in E(F_i)$ is in some $L \in \mathcal{L}^*_{F_i}$, by definition. Hence $F \in \mathcal{C}^*$. 
Recall that $H_1$ and $H_2$ satisfy the criteria of Conjecture~\ref{conj:aconj}. Hence $H_2$ is strictly $2$-balanced, $H_1$ is strictly balanced with respect to $d_2(\cdot, H_2)$) and $m_2(H_1) \geq m_2(H_2) > 1$.
Then, by Lemma~\ref{lemma:2connected}, $H_1$ and $H_2$ are $2$-connected, hence $F_i$ is $2$-connected by construction. However, our choice of $F_0$ in line~\ref{line:growanyr} guarantees that $F_i$ is not in $\hat{\mathcal{A}}$.
Indeed, the edge $e$ selected in line~\ref{line:growanye} satisfying $|\mathcal{S}_{G'}(e)| = 0$ is an edge of $F_0$ and $F_0 \subseteq F_i \subseteq G'$. 
Thus, by the definition of $\hat{\mathcal{A}}$ and that $m_2(H_1) > m_2(H_2)$, we have that $m(F_i) > m_2(H_1, H_2) + \varepsilon$.
Thus, there exists a non-empty graph $\tilde{F} \subseteq F_i$ with $d(\tilde{F}) = m(F_i)$ such that 
\begin{align*}
\lambda(\tilde{F}) & = v(\tilde{F}) - \frac{e(\tilde{F})}{m_2(H_1, H_2)} \\  
                                                 & = e(\tilde{F})\left(\frac{1}{m(F_i)} - \frac{1}{m_2(H_1, H_2)}\right) \\
                                                 & < e(\tilde{F})\left(\frac{1}{m_2(H_1, H_2) + \varepsilon} - \frac{1}{m_2(H_1, H_2)}\right) \\
                                                 & = -\gamma e(\tilde{F}) \leq -\gamma.
\end{align*}
Thus \textsc{Grow} terminates in line~\ref{line:growwhileconditions} without calling \textsc{Eligible-Edge}, and so every call to \textsc{Eligible-Edge} is successful and returns an edge $e$.
Since $G' \in \mathcal{C^*}$, the call to \textsc{Extend-L}$(F_i, e, G')$ is also successful and thus there exist suitable graphs $L \in \mathcal{L}^*_{G'}$ with $e \in E(L)$ and $R_{e'}$ for each $e' \in E(L) \setminus E(F_i)$. 

It remains to show that for every iteration $i$ of the while-loop, we have $e(F_{i+1}) > e(F_i)$. 
Since a copy $R$ of $H_1$ found in line~\ref{line:growRline} is a copy of $H_1$ in $G'$ but not in $F_i$ (and $H_1$ is connected), we must have that $F_{i+1} = F_i \cup R$ contains at least one more edge than $F_i$.

So assume lines~\ref{line:eligible-edge} and \ref{line:F_i+1getsextend} are called in iteration $i$ and let $e$ be the edge chosen in line~\ref{line:eligible-edge} and $L$ the subgraph selected in line~\ref{line:linl*} of \textsc{Extend-L}$(F_i, e, G')$.
By the definition of $\mathcal{L}^*_{G'}$, for each $e' \in E(L)$ there exists $R_{e'} \in \mathcal{R}_{G'}$ such that $E(L) \cap E(R_{e'}) = \{e'\}$. If $|E(L) \setminus E(F_i)| > 0$, then $e(F_{i+1}) \geq e(F_i \cup L) > e(F_i)$.
Otherwise, $L \subseteq F_i$. But since $e$ is eligible for extension in \textsc{Grow}, we must have $L \notin \mathcal{L}^*_{F_i}$.
Thus there exists $e' \in L$ such that $R_{e'} \in \mathcal{R}_{G'}\setminus \mathcal{R}_{F_i}$ and $|V(R_{e'}) \cap V(F_i)| \geq 2$, contradicting that lines~\ref{line:eligible-edge} and \ref{line:F_i+1getsextend} are called in iteration $i$.\end{proof}

\subsection{Proof of Lemma~\ref{lemma:noerror}}\label{sec:lemnoerror}

We consider the evolution of $F_i$ now in more detail. We call iteration $i$ of the while-loop in algorithm \textsc{Grow} \emph{non-degenerate} if all of the following hold:

\begin{itemize}
    \item The condition in line~\ref{line:growRline} evaluates to false (and \textsc{Extend-L} is called); 
    \item In line~\ref{line:f'fl} of \textsc{Extend-L}, we have $V(F) \cap V(L) = e$;
    \item In every execution of line~\ref{line:f'f're'} of \textsc{Extend-L}, we have $V(F') \cap V(R_{e'}) = e'$.
\end{itemize}
Otherwise, we call iteration $i$ \emph{degenerate}.
Note that, in non-degenerate iterations, there are only a \textit{constant} number of graphs $F_{i+1}$ that can result from any given $F_i$ since \textsc{Eligible-Edge} determines the exact position where to attach the copy $L$ of $H_2$, $V(F_i) \cap V(L) = e$ and for every execution of line~\ref{line:f'f're'} of \textsc{Extend-L} we have $V(F') \cap V(R_{e'}) = e'$ (recall that the edge $e$ found by \textsc{Eligible-Edge}($F_i$) is \emph{unique up to isomorphism of $F_i$}).

\begin{claim}\label{claim:non-degen}
If iteration $i$ of the while-loop in procedure \textsc{Grow} is non-degenerate, we have \[\lambda(F_{i+1}) = \lambda(F_i).\]
\end{claim}

\begin{proof}
In a non-degenerate iteration we add $v_2 - 2$ vertices and $e_2 - 1$ edges for the copy of $H_2$ and then $(e_2 - 1)(v_1 - 2)$ new vertices and $(e_2 - 1)(e_1 - 1)$ new edges to complete the copies of $H_1$. This gives \begin{align*}
  \lambda(F_{i+1}) - \lambda(F_i)   & = v_2 - 2 + (e_2  - 1)(v_1 - 2) - \frac{(e_2 - 1)e_1}{m_2(H_1,H_2)} \\
                                    & = v_2 - 2 + (e_2  - 1)(v_1 - 2) - (e_2 - 1)\left(v_1 - 2 + \frac{1}{m_2(H_2)}\right) \\
                                    & = 0,
\end{align*}
where we have used in the penultimate equality that $H_1$ is (strictly) balanced with respect to $d_2(\cdot, H_2)$ and in the final inequality that $H_2$ is (strictly) $2$-balanced.\end{proof}

When we have a degenerate iteration $i$, the structure of $F_{i+1}$ may vary considerably and also depend on the structure of $G'$.
Indeed, if $F_i$ is extended by a copy $R$ of $H_1$ in line~\ref{line:growR-F_i+1}, then $R$ could intersect $F_i$ in a multitude of ways.
Moreover, there may be several copies of $H_1$ that satisfy the condition in line~\ref{line:growRline}.
The same is true for graphs added in lines~\ref{line:f'fl} and \ref{line:f'f're'} of \textsc{Extend-$L$}. 
Thus, degenerate iterations cause us difficulties since they enlarge the family of graphs algorithm \textsc{Grow} can return.
However, we will show that at most a constant number of degenerate iterations can happen before algorithm \textsc{Grow} terminates, allowing us to bound from above sufficiently well the number of non-isomorphic graphs \textsc{Grow} can return.
Pivotal in proving this is the following claim.  

\begin{claim}\label{claim:degenfull}
There exists a constant $\kappa = \kappa(H_1,H_2) > 0$ such that if iteration $i$ of the while-loop in procedure \textsc{Grow} is degenerate then we have \[\lambda(F_{i+1}) \leq \lambda(F_i) - \kappa.\]
\end{claim}

We prove Claim~\ref{claim:degenfull} in Section~\ref{sec:claimdegenfull}.
Together, Claims~\ref{claim:non-degen} and \ref{claim:degenfull} yield the following claim.

\begin{claim}\label{claim:q_1}
There exists a constant $q_1 = q_1(H_1,H_2)$ such that algorithm \textsc{Grow} performs at most $q_1$ degenerate iterations before it terminates, regardless of the input instance $G'$.
\end{claim}

\begin{proof}
By Claim~\ref{claim:non-degen}, the value of the function $\lambda$ remains the same in every non-degenerate iteration of the while-loop of algorithm \textsc{Grow}. However, Claim~\ref{claim:degenfull} yields a constant $\kappa$, which depends solely on $H_1$ and $H_2$, such that \[\lambda(F_{i+1}) \leq \lambda(F_i) - \kappa\] for every degenerate iteration $i$.

Hence, after at most \[q_1 := \frac{\lambda(F_0) + \gamma}{\kappa}\] degenerate iterations, we have $\lambda(F_i) \leq -\gamma$, and algorithm \textsc{Grow} terminates.\end{proof}

For $0 \leq d \leq t < \lceil \ln(n) \rceil$, let $\mathcal{F}(t,d)$ denote a family of representatives for the isomorphism classes of all graphs $F_t$ that algorithm \textsc{Grow} can possibly generate after exactly $t$ iterations of the while-loop with exactly $d$ of those $t$ iterations being degenerate. Let $f(t,d) := |\mathcal{F}(t,d)|$.

\begin{claim}\label{claim:polylog}
There exist constants $C_0 = C_0(H_1, H_2)$ and $A = A(H_1,H_2)$ such that \[f(t,d)~\leq\lceil\ln(n)\rceil^{(C_0 +1)d}\cdot~A^{t-d}\] for $n$ sufficiently large.
\end{claim}

\begin{proof}
By Claim~\ref{claim:grow}, in every iteration $i$ of the while-loop of \textsc{Grow}, we add new edges onto $F_i$. These new edges span a graph on at most \[K := v_2 + (e_2 - 1)(v_1 - 2)\] vertices.  
Thus $v(F_t) \leq v_1 + Kt$. 
Let $\mathcal{G}_K$ denote the set of all graphs on at most $K$ vertices. 
In iteration $i$ of the while-loop, $F_{i+1}$ is uniquely defined if one specifies the graph $G \in \mathcal{G}_K$ with edges $E(F_{i+1})\setminus E(F_{i})$, the number $y$ of vertices in which $G$ intersects $F_i$, and two ordered lists of vertices from $G$ and $F_i$ respectively of length $y$, which specify the mapping of the intersection vertices from $G$ onto $F_i$. 
Thus, the number of ways that $F_i$ can be extended to $F_{i+1}$ is bounded from above by \[\sum_{G \in \mathcal{G}_K}\sum_{y = 2}^{v(G)}v(G)^yv(F_i)^y \leq |\mathcal{G}_K|\cdot K \cdot K^K(v_1 + Kt)^K \leq \lceil \ln(n) \rceil^{C_0},\] where $C_0$ depends only on $v_1$, $v_2$ and $e_2$, and $n$ is sufficiently large. 
The last inequality follows from the fact that $t < \ln(n)$ as otherwise the while-loop would have already ended. 

Recall that, since \textsc{Eligible-Edge} determines the exact position where to attach the copy of $H_2$, in non-degenerate iterations $i$ there are at most \[2e_2(2e_1)^{e_2-1} =: A\] ways to extend $F_i$ to $F_{i+1}$, where the coefficients of 2 correspond with the orientations of the edge of the copy of $H_2$ we attach to $F_i$ and the edges of the copies of $H_1$ we attach to said copy of $H_2$. 
Hence, for $0 \leq d \leq t < \lceil \ln(n) \rceil$, \[f(t,d) \leq \binom{t}{d}(\lceil \ln(n) \rceil^{C_0})^d \cdot A^{t-d} \leq \lceil \ln(n) \rceil^{(C_0 +1)d} \cdot A^{t-d},\] 
where the binomial coefficient corresponds to the choice of when in the $t$ iterations the $d$ degenerate iterations happen.\end{proof}

A reader of \cite{msss} may observe that Claim~\ref{claim:polylog} is not analogous to Claim 17 in \cite{msss}.
Indeed, since we have a constant number of non-degenerate iterations, instead of a unique non-degenerate iteration as in \cite{msss}, we truncated the proof of Claim 17 in order to have the appropriate bound to prove the following claim.
Let $\mathcal{F} = \mathcal{F}(H_1, H_2, n)$ be a family of representatives for the isomorphism classes of \emph{all} graphs that can be outputted by \textsc{Grow} (whether \textsc{Grow} enters the while-loop or not).
Note that the proof of the following claim requires Conjecture~\ref{conj:aconj} to be true; in particular, we need that $\hat{\mathcal{A}}(H_1, H_2, \varepsilon)$ is finite when $m_2(H_1) > m_2(H_2)$.

\begin{claim}\label{claim:conclusion}
There exists a constant $b = b(H_1,H_2) > 0$ such that for all $p \leq bn^{-1/m_2(H_1,H_2)}$, $G_{n,p}$ does not contain any graph from $\mathcal{F}(H_1,H_2,n)$ a.a.s.
\end{claim}

\begin{proof}
We first consider the two special cases in lines~\ref{line:growspecial1}-\ref{line:endif} of \textsc{Grow}.
Let $\mathcal{F}_0 = \mathcal{F}_0(H_1,H_2) \subseteq \mathcal{F}$ denote the class of graphs that can be outputted by \textsc{Grow} if one of these two cases happens. 
We can see that any $F \in \mathcal{F}_0$ is either of the form \[F = \bigcup_{\substack{e \in E(T)}}\mathcal{S}_{G'}(e)\] for some graph $T \in \mathcal{T}_{G'}$, or of the form \[F = S_1 \cup S_2\] for some edge-intersecting $S_1, S_2 \in \mathcal{S}_{G'}$.
Whichever of these forms $F$ has, since every element of $\mathcal{S}_{G'}$ is $2$-connected and in $\mathcal{C}^*$, and $T$ is $2$-connected\footnote{Since $T \cong H_1$ or $T \cong H_2$ and Lemma~\ref{lemma:2connected} holds.}, 
we have that $F$ is $2$-connected and in $\mathcal{C}^*$. On the other hand, $F \subseteq G'$ is not in $\mathcal{S}_{G'}$ and thus not isomorphic to a graph in $\hat{\mathcal{A}}$.
Indeed, otherwise the graphs $S$ forming $F$ would not be in $\mathcal{S}_{G'}$ due to the maximality condition in the definition of $\mathcal{S}_{G'}$. 
It follows that $m(F) > m_2(H_1, H_2) + \varepsilon(H_1, H_2)$. Since we assumed Conjecture~\ref{conj:aconj} holds, the family $\mathcal{F}_0$ is finite. Hence Markov's inequality yields that $G_{n,p}$ contains no graph from $\mathcal{F}_0$ a.a.s.

Let $\tilde{\mathcal{F}} = \tilde{\mathcal{F}}(H_1, H_2, n)$ denote a family of representatives for the isomorphism classes of all graphs that can be the output of \textsc{Grow} with parameters $n$ and $\gamma(H_1, H_2)$ on any input instance $G'$ for which it enters the while-loop. 
Observe that $\mathcal{F} = \mathcal{F}_0 \cup \tilde{\mathcal{F}}$. Let $\mathcal{F}_1$ and $\mathcal{F}_2$ denote the classes of graphs that algorithm \textsc{Grow} can output in lines~\ref{line:growreturnfi} and \ref{line:growreturnminsub}, respectively. 
For each $F \in \mathcal{F}_1$, we have that $e(F) \geq \ln(n)$, as $F$ was generated in $\lceil \ln(n) \rceil$ iterations, each of which introduces at least one new edge by Claim~\ref{claim:grow}.
Moreover, Claims~\ref{claim:non-degen} and \ref{claim:degenfull} imply that $\lambda(F_i)$ is non-increasing. Thus, we have that $\lambda(F) \leq \lambda(F_0)$ for all $F \in \mathcal{F}_1$. 
For all $F \in \mathcal{F}_2$, we have that $\lambda(F) \leq -\gamma$ due to the condition in line~\ref{line:growwhileconditions} of \textsc{Grow}.
Let $A := A(H_1, H_2)$ be the constant found in the proof of Claim~\ref{claim:polylog}.
Since we have chosen $F_0 \cong H_1$ as the seed of the growing procedure, it follows that for \[b := (Ae)^{-\lambda(F_0)-\gamma} \leq 1,\] the expected number of copies of graphs from $\tilde{\mathcal{F}}$ in $G_{n,p}$ with $p \leq bn^{-1/m_2(H_1,H_2)}$ is bounded by 

\begin{align}
\sum_{F \in \tilde{\mathcal{F}}}n^{v(F)}p^{e(F)} & \leq  \sum_{F \in \tilde{\mathcal{F}}}b^{e(F)}n^{\lambda(F)} \label{eq:finalstart}                           \\
                                                 & \leq \sum_{F \in \mathcal{F}_1}(eA)^{(-\lambda(F_0) - \gamma)\ln(n)}n^{\lambda(F_0)} + \sum_{F \in \mathcal{F}_2}b^{e(F)}n^{-\gamma}\nonumber  \\ 
                                                 & = \sum_{F \in \mathcal{F}_1}A^{(-\lambda(F_0) - \gamma)\ln(n)}n^{-\gamma} + \sum_{F \in \mathcal{F}_2}b^{e(F)}n^{-\gamma}. \nonumber                     
\end{align}
Observe that, since $m_2(F_2) \geq 1$, we have that \begin{equation}\label{eq:lambdaF_0}\lambda(F_0) = v_1 - \frac{e_1}{m_2(F_1,F_2)} = 2 - \frac{1}{m_2(F_2)} \geq 1\end{equation}

By Claims~\ref{claim:grow}, \ref{claim:q_1} and \ref{claim:polylog}, and \eqref{eq:lambdaF_0}, we have that
\begin{align}
     \sum_{F \in \mathcal{F}_1}A^{(-\lambda(F_0) - \gamma)\ln(n)}n^{-\gamma} & \leq \sum_{d = 0}^{\min\{t, q_1\}}f(\lceil \ln(n) \rceil,d)A^{(-\lambda(F_0) - \gamma)\ln(n)}n^{-\gamma}\label{eq:final1} \\ 
     & \leq (q_1 + 1)\lceil \ln(n) \rceil^{(C_0 + 1)q_1} \cdot A^{\lceil \ln(n) \rceil}A^{(-\lambda(F_0) - \gamma)\ln(n)}n^{-\gamma}\nonumber \\ 
     & \leq (\ln(n))^{2(C_0 + 1)q_1}n^{-\gamma}.\nonumber
\end{align}
Observe that, by Claim~\ref{claim:grow}, if some graph $F \in \mathcal{F}_2$ is the output of \textsc{Grow} after precisely $t$ iterations of the while-loop then $e(F) \geq t$. Since $b < 1$, this implies 
\begin{equation}\label{eq:befbt}
    b^{e(F)} \leq b^t
\end{equation}
for such a graph $F$.
Using \eqref{eq:befbt} and Claims~\ref{claim:grow}, \ref{claim:q_1} and \ref{claim:polylog}, we have that 
\begin{align}
    \sum_{F \in \mathcal{F}_2}b^{e(F)}n^{-\gamma} & \leq \sum_{t = 0}^{\lceil \ln(n)  \rceil}\sum_{d = 0}^{\min\{t, q_1\}}f(t,d)b^tn^{-\gamma}\label{eq:final2} \\ 
    & \leq \sum_{t = 0}^{\lceil \ln(n)  \rceil}\sum_{d = 0}^{\min\{t, q_1\}}\lceil \ln(n) \rceil^{(C_0 +1)d} \cdot A^{t-d}(Ae)^{(-\lambda(F_0)-\gamma)t}n^{-\gamma}\nonumber \\
    & \leq (\lceil \ln(n) \rceil + 1)(q_1 + 1)\lceil \ln(n) \rceil^{(C_0 +1)q_1} n^{-\gamma}\nonumber \\
    & \leq (\ln(n))^{2(C_0 + 1)q_1}n^{-\gamma}. \nonumber
\end{align}
Thus, by \eqref{eq:finalstart}, \eqref{eq:final1} and \eqref{eq:final2}, we have that $\sum_{F \in \tilde{\mathcal{F}}}n^{v(F)}p^{e(F)} = o(1)$. Consequently, Markov's inequality implies that $G_{n,p}$ a.a.s.~contains no graph from $\tilde{\mathcal{F}}$.

Combined with the earlier observation that $G_{n,p}$ a.a.s.~contains no graph from $\mathcal{F}_0$, we have that $G_{n,p}$ a.a.s.~contains no graph from $\mathcal{F} = \mathcal{F}_0 \cup \tilde{\mathcal{F}}$.\end{proof}

\begin{proofofnoerrorlemmacaseone}
Suppose that the call to \textsc{Asym-Edge-Col}$(G)$ gets stuck for some graph $G$, and consider $G' \subseteq G$ at this moment. 
Then \textsc{Grow}$(G')$ returns a copy of a graph $F \in \mathcal{F}(H_1, H_2, n)$ that is contained in $G' \subseteq G$. 
Provided Claim~\ref{claim:degenfull} holds, by Claim~\ref{claim:conclusion} this event a.a.s.~does not occur in $G = G_{n,p}$ with $p$ as claimed.
Thus \textsc{Asym-Edge-Col} does not get stuck a.a.s.~and, by Lemma~\ref{lemma:errorvalid}, finds a valid colouring for $H_1$ and $H_2$ of $G_{n,p}$ with $p \leq bn^{-1/m_2(H_1,H_2)}$ a.a.s.\qed
\end{proofofnoerrorlemmacaseone}

\subsection{Proof of Claim~\ref{claim:degenfull}}\label{sec:claimdegenfull}

Our strategy for proving Claim~\ref{claim:degenfull} revolves around comparing our degenerate iteration $i$ of the while-loop of algorithm \textsc{Grow} with any non-degenerate iteration which could have occurred instead.
In accordance with this strategy, we have the following technical lemma which will be crucial in proving Claim~\ref{claim:degenfull}.\footnote{More specifically, in proving Claim~\ref{claim:degen2}, stated later.} The lemma will play the same role as Lemma 21 does in \cite{msss}, but is considerably different.
In order to state our technical lemma, we define the following families of graphs.

\begin{define}\label{def:keyfam}
\textnormal{Let $F$, $H_1$ and $H_2$ be graphs and $\hat{e} \in E(F)$. We define $\mathcal{H}(F, \hat{e}, H_1, H_2)$ to be the family of graphs constructed from $F$ in the following way: Attach a copy $H_{\hat{e}}$ of $H_2$ to $F$ such that $E(H_{\hat{e}}) \cap E(F) = \{\hat{e}\}$ and $V(H_{\hat{e}}) \cap V(F) = \hat{e}$. 
Then, for each edge $f \in E(H_{\hat{e}})\setminus\{\hat{e}\}$, attach a copy $H_f$ of $H_1$ to $F \cup H_{\hat{e}}$ such that $E(F \cup H_{\hat{e}}) \cap E(H_f) = \{f\}$ and $(V(F)\setminus \hat{e}) \cap V(H_f) = \emptyset$. }
\end{define}

\begin{figure}[!ht]
\begin{center}
\definecolor{ffqqqq}{rgb}{1,0,0}
\definecolor{qqqqff}{rgb}{0,0,1}
\definecolor{qqzzqq}{rgb}{0,0.6,0}
\definecolor{ffxfqq}{rgb}{1,0.4980392156862745,0}
\definecolor{yqqqyq}{rgb}{0.5019607843137255,0,0.5019607843137255}
\begin{tikzpicture}[line cap=round,line join=round,>=triangle 45,x=1cm,y=1cm]
\draw [line width=2pt] (6,5)-- (5,2);
\draw [line width=2pt] (6,5)-- (6,2);
\draw [line width=2pt] (6,5)-- (7,2);
\draw [line width=2pt] (9,5)-- (8,2);
\draw [line width=2pt] (9,5)-- (9,2);
\draw [line width=2pt] (9,5)-- (10,2);
\draw (7.25,1.75) node[anchor=north west] {\Large $F$};
\draw [line width=2pt] (6,5)-- (9,5);
\draw [line width=2pt,color=yqqqyq] (6,5)-- (5,7);
\draw [line width=2pt,color=ffxfqq] (5,7)-- (6,9);
\draw [line width=2pt,color=qqzzqq] (6,9)-- (9,9);
\draw [line width=2pt,color=qqqqff] (9,9)-- (10,7);
\draw [line width=2pt,color=ffqqqq] (10,7)-- (9,5);
\draw [line width=2pt,color=ffqqqq, dash pattern= on 8pt off 8pt] (10,7)-- (13,7);
\draw [line width=2pt,color=qqqqff, dash pattern= on 8pt off 8pt,dash phase=8pt] (10,7)-- (13,7);
\draw [line width=2pt,color=ffqqqq] (9,5)-- (11,4);
\draw [line width=2pt,color=ffqqqq] (11,4)-- (13,5);
\draw [line width=2pt,color=ffqqqq] (13,7)-- (13,5);
\draw [line width=2pt,color=qqqqff] (13,7)-- (14,10);
\draw [line width=2pt,color=qqqqff] (14,10)-- (11,11);
\draw [line width=2pt,color=qqqqff] (11,11)-- (9,9);
\draw [line width=2pt,color=qqzzqq] (6,9)-- (6,5);
\draw [line width=2pt,color=qqzzqq] (6,5)-- (3,8);
\draw [line width=2pt,color=qqzzqq] (9,9)-- (5,11);
\draw [line width=2pt,color=qqzzqq] (3,8)-- (5,11);
\draw [line width=2pt,color=ffxfqq, dash pattern= on 8pt off 8pt] (5,7)-- (3,8);
\draw [line width=2pt,color=yqqqyq, dash pattern= on 8pt off 8pt,dash phase=8pt] (5,7)-- (3,8);
\draw [line width=2pt,color=ffxfqq] (3,8)-- (2,11);
\draw [line width=2pt,color=ffxfqq] (2,11)-- (5,11);
\draw [line width=2pt,color=ffxfqq] (5,11)-- (6,9);
\draw [line width=2pt,color=yqqqyq] (3,8)-- (2,6);
\draw [line width=2pt,color=yqqqyq] (3,4)-- (2,6);
\draw [line width=2pt,color=yqqqyq] (3,4)-- (6,5);
\draw [color=yqqqyq](3.25,6.25) node[anchor=north west] {\Large $H_{f_5}$};
\draw [color=qqqqff](11.08,9.25) node[anchor=north west] {\Large $H_{f_2}$};
\draw [color=ffxfqq](2.9,10.5) node[anchor=north west] {\Large $H_{f_4}$};
\draw (7.1,7.4) node[anchor=north west] {\Large $H_{\hat{e}}$};
\draw [color=ffqqqq](10.75,6) node[anchor=north west] {\Large $H_{f_1}$};
\draw [color=qqzzqq](4.35,9.15) node[anchor=north west] {\Large $H_{f_3}$};
\draw (7.25,5) node[anchor=north west] {$\hat{e}$};
\begin{scriptsize}
\draw [fill=black] (6,5) circle (2.5pt);
\draw [fill=black] (9,5) circle (2.5pt);
\draw [fill=black] (5,7) circle (2.5pt);
\draw[color=yqqqyq] (5.7,6.25) node {$f_5$};
\draw [fill=black] (6,9) circle (2.5pt);
\draw[color=ffxfqq] (5.7,7.9) node {$f_4$};
\draw [fill=black] (9,9) circle (2.5pt);
\draw[color=qqzzqq] (7.62,8.67) node {$f_3$};
\draw [fill=black] (10,7) circle (2.5pt);
\draw[color=qqqqff] (9.2,7.9) node {$f_2$};
\draw[color=ffqqqq] (9.29,6.25) node {$f_1$};
\draw [fill=black] (13,7) circle (2.5pt);
\draw [fill=black] (11,4) circle (2.5pt);
\draw [fill=black] (13,5) circle (2.5pt);
\draw [fill=black] (14,10) circle (2.5pt);
\draw [fill=black] (11,11) circle (2.5pt);
\draw [fill=black] (3,8) circle (2.5pt);
\draw [fill=black] (5,11) circle (2.5pt);
\draw [fill=black] (2,11) circle (2.5pt);
\draw [fill=black] (2,6) circle (2.5pt);
\draw [fill=black] (3,4) circle (2.5pt);
\end{scriptsize}
\end{tikzpicture}
\caption{A graph $J \in \mathcal{H}(F, \hat{e}, C_5, C_6)\setminus \mathcal{H}^*(F, \hat{e}, C_5, C_6)$.}\label{fig:jexample}
\end{center}
\end{figure}
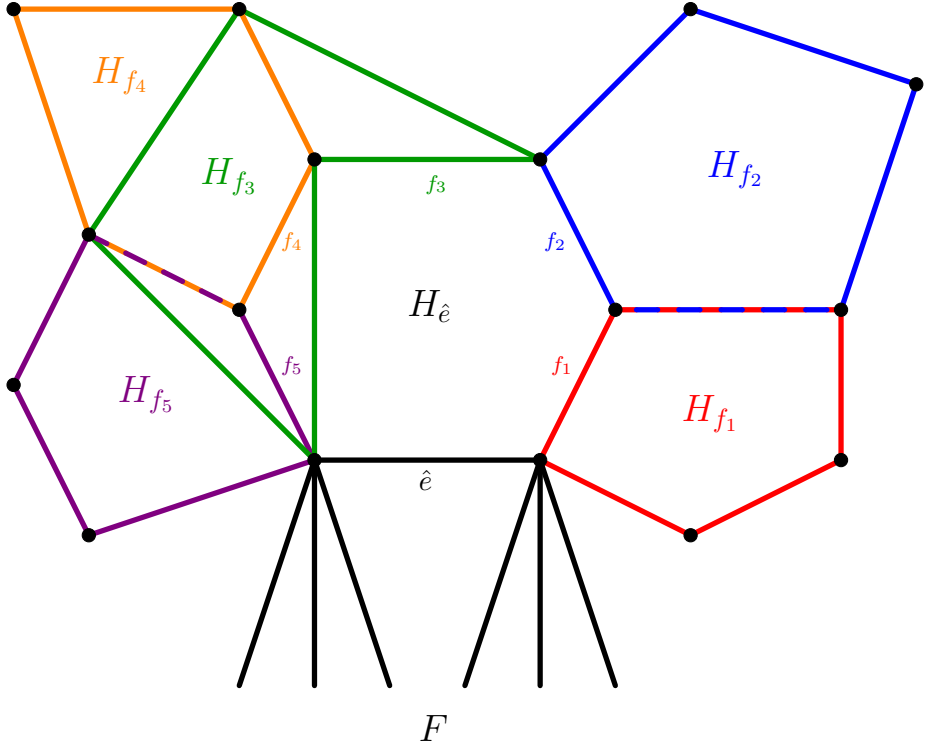

Notice that, during construction of a graph $J \in \mathcal{H}(F, \hat{e}, H_1, H_2)$, the edge of $H_{\hat{e}}$ intersecting at $\hat{e}$ and the edge of each copy $H_f$ of $H_1$ intersecting at an edge $f \in E(H_{\hat{e}})\setminus\{\hat{e}\}$
are not stipulated. 
That is, we may end up with different graphs after the construction process if we choose different edges of $H_{\hat{e}}$ to intersect $F$ at $\hat{e}$ and different edges of the copies $H_f$ of $H_1$ to intersect the edges in $E(H_{\hat{e}})\setminus\{\hat{e}\}$.
Observe that although $E(F \cup H_{\hat{e}}) \cap E(H_f) = \{f\}$ and $(V(F)\setminus \hat{e}) \cap V(H_f) = \emptyset$ for each $f \in E(H_{\hat{e}}) - \{\hat{e}\}$, the construction may result in one or more graphs $H_f$ intersecting $H_{\hat{e}}$ in more than two vertices, including possibly in vertices of $\hat{e}$ (e.g. $H_{f_3}$ in Figure~\ref{fig:jexample}). Also, the graphs $H_f$ may intersect with each other in vertices and/or edges (e.g.~$H_{f_1}$ and $H_{f_2}$ in Figure~\ref{fig:jexample}).

Borrowing notation and language from \cite{msss}, for any $J \in \mathcal{H}(F, \hat{e}, H_1, H_2)$ we call $V_J := V(H_{\hat{e}})\setminus \hat{e}$ the \emph{inner vertices of $J$} and $E_J := E(H_{\hat{e}})\setminus\{\hat{e}\}$ the \emph{inner edges of $J$}.\label{page:J*} 
Let $H_{\hat{e}}^{J}$ be the \emph{inner graph} on vertex set $V_J \ \dot{\cup} \ \hat{e}$ and edge set $E_J$, and observe that this graph $H_{\hat{e}}^{J}$ is isomorphic to a copy of $H_2$ minus some edge.
Further, for each copy $H_f$ of $H_1$, we define $U_J(f) := V(H_f)\setminus f$ and $D_J(f) := E(H_f) \setminus\{f\}$ and call \[U_J := \bigcup_{\substack{{f \in E_J}}} U_J(f)\] the set of \emph{outer vertices} of $J$ and \[D_J := \bigcup_{\substack{{f \in E_J}}} D_J(f)\] the set of \emph{outer edges} of $J$.
Observe that the sets $U_J(f)$ may overlap with each other and, as noted earlier, with $V(H_{\hat{e}}^{J})$. However, the sets $D_J(f)$ may overlap only with each other. 
Further, define $\mathcal{H}^*(F, \hat{e}, H_1, H_2) \subseteq \mathcal{H}(F, \hat{e}, H_1, H_2)$ such that for any $J^* \in \mathcal{H}^*(F, \hat{e}, H_1, H_2)$ we have $U_{J^*}(f_1) \cap U_{J^*}(f_2) = \emptyset$ and $D_{J^*}(f_1) \cap D_{J^*}(f_2) = \emptyset$ for all $f_1, f_2 \in E_{J^*}$, $f_1 \neq f_2$, and $U_{J^*}(f) \cap V(H_{\hat{e}}^{J^*}) = \emptyset$ for all $f \in E_{J^*}$; that is, the copies of $H_1$ are, in some sense, pairwise disjoint.
Note that each $J^* \in \mathcal{H}^*(F, \hat{e}, H_1, H_2)$ corresponds with a non-degenerate iteration $i$ of the while loop of algorithm \textsc{Grow} when $F = F_i$, $J^* = F_{i+1}$ and $\hat{e}$ is the edge chosen by \textsc{Eligible-Edge}($F_i$). This observation will be very helpful several times later.
For any $J \in \mathcal{H}(F, \hat{e}, H_1, H_2)$, define \[v^{+}(J) := |V(J)\setminus V(F)| = v(J) - v(F)\] and \[e^{+}(J) := |E(J)\setminus E(F)| = e(J) - e(F),\] and call $\frac{e^{+}(J)}{v^{+}(J)}$ the \emph{$F$-external density of $J$}.
The following lemma relates the $F$-external density of any $J^* \in \mathcal{H}^*(F, \hat{e}, H_1, H_2)$ to that of any $J \in \mathcal{H}(F, \hat{e}, H_1, H_2)\setminus \mathcal{H}^*(F, \hat{e}, H_1, H_2)$.

\begin{lem}\label{lemma:21}
Let $F$ be a graph and $\hat{e} \in E(F)$. 
Then for any $J \in \mathcal{H}(F, \hat{e}, H_1, H_2)\setminus \mathcal{H}^*(F, \hat{e}, H_1, H_2)$ and any $J^* \in \mathcal{H}^*(F, \hat{e}, H_1, H_2)$, we have \[\frac{e^{+}(J)}{v^{+}(J)} > \frac{e^{+}(J^*)}{v^{+}(J^*)}.\]
\end{lem}  We prove Lemma~\ref{lemma:21} in Section~\ref{sec:lemma21}.\smallskip

Claim~\ref{claim:degenfull} will follow from the next two claims. We say that algorithm \textsc{Grow} encounters a \emph{degeneracy of type 1} in iteration $i$ of the while-loop if line~\ref{line:growRline} returns true, that is, $\exists R \in \mathcal{R}_{G'}\setminus \mathcal{R}_{F_i} : |V(R) \cap V(F_i)| \geq 2$.
Note that the following claim requires that $m_2(H_1) > m_2(H_2)$.

\begin{claim}\label{claim:degen1}
There exists a constant $\kappa_1 = \kappa_1(H_1, H_2) > 0$ such that if procedure \textsc{Grow} encounters a degeneracy of type 1 in iteration $i$ of the while-loop, we have \[\lambda(F_{i+1}) \leq \lambda(F_i) - \kappa_1.\]
\end{claim}

\begin{proof}
Let $F := F_i$ be the graph before the operation in line~\ref{line:growR-F_i+1} is carried out (that is, before $F_{i+1} \gets F_i \cup R$), let $R$ be the copy of $H_1$ merged with $F$ in line~\ref{line:growR-F_i+1} and let $F' := F_{i+1}$ be the output from line~\ref{line:growR-F_i+1}.
We aim to show there exists a constant $\kappa_1 = \kappa_1(H_1, H_2) > 0$ such that
\begin{equation*}
    \lambda(F) - \lambda(F') = v(F) - v(F') - \frac{e(F) - e(F')}{m_2(H_1, H_2)} \geq \kappa_1.
\end{equation*}
Choose any edge $\hat{e} \in E(F)$ (the edge $\hat{e}$ need not be in the intersection of $R$ and $F$).
Choose any $F^* \in \mathcal{H}^*(F, \hat{e}, H_1, H_2)$.
Our strategy is to compare our degenerate outcome $F'$ with $F^*$. As noted earlier, $F^*$ corresponds to a non-degenerate iteration of the while loop of algorithm \textsc{Grow} (if $\hat{e}$ was the edge chosen by \textsc{Eligible-Edge}). Then Claim~\ref{claim:non-degen} gives us that $\lambda(F) = \lambda(F^*)$. Then 
\begin{eqnarray*}
      \lambda(F) - \lambda(F') & = \lambda(F^*) - \lambda(F')
     = v(F^*) - v(F') - \frac{e(F^*) - e(F')}{m_2(H_1, H_2)}.
\end{eqnarray*}
Hence we aim to show that there exists $\kappa_1 = \kappa_1(H_1, H_2) > 0$ such that
\begin{equation}\label{eq:kappa1bound}
    v(F^*) - v(F') - \frac{e(F^*) - e(F')}{m_2(H_1, H_2)} \geq \kappa_1.
\end{equation}
Define $R'$ to be the graph with vertex set $V' := V(R) \cap V(F)$ and edge set $E' :=  E(R) \cap E(F)$, and let $v' := |V'|$ and $e' := |E'|$. Observe that $R' \subsetneq R$.
Since $F^*$ corresponds with a non-degenerate iteration of the while-loop of algorithm \textsc{Grow}, $H_2$ is (strictly) $2$-balanced and $H_1$ is (strictly) balanced with respect to $d_2(\cdot, H_2)$, we have
\begin{eqnarray}
  v(F^*) - v(F') - \frac{e(F^*) - e(F')}{m_2(H_1, H_2)} & = &  (e_2 - 1)(v_1 - 2) + (v_2 - 2) - (v_1 - v')\nonumber \\ & &  - \frac{(e_2 - 1)e_1 - (e_1 - e')}{m_2(H_1, H_2)}\nonumber \\
  & = & (e_2 - 1)(v_1 - 2) + (v_2 - 2) \nonumber \\
  & & - (e_2 - 1)\left(v_1 - 2 + \frac{1}{m_2(H_2)}\right) \nonumber \\
  & & + \frac{e_1 - e'}{m_2(H_1, H_2)} - (v_1 - v')\nonumber \\
  & = &  \frac{e_1 - e'}{m_2(H_1, H_2)} - (v_1 - v')\nonumber \\ 
  & = &  v' - 2 + \frac{1}{m_2(H_2)} - \frac{e'}{m_2(H_1, H_2)}. \label{eq:f1end}
\end{eqnarray}
Also, since \textsc{Grow} encountered a degeneracy of type 1, we must have $v' \geq 2$. 
Hence, if $e' = 0$ and $v' \geq 2$, then \[v' - 2 + \frac{1}{m_2(H_2)} - \frac{e'}{m_2(H_1, H_2)} \geq \frac{1}{m_2(H_2)} > 0.\] 
If $e' \geq 1$, then since $R$ is a copy of $H_1$, $H_1$ is \emph{strictly} balanced with respect to $d_2(\cdot, H_2)$ and $R' \subsetneq R$ with $|E(R')| = e' \geq 1$, we have that $0 < d_2(R', H_2) < m_2(H_1, H_2)$, and so 
\begin{equation}\label{eq:m2d2}
-\frac{1}{m_2(H_1, H_2)} > -\frac{1}{d_2(R', H_2)}.
\end{equation}
Then by \eqref{eq:f1end} and \eqref{eq:m2d2}, we have that 
\begin{equation*}
    v(F^*) - v(F') - \frac{e(F^*) - e(F')}{m_2(H_1, H_2)} = v' - 2 + \frac{1}{m_2(H_2)} - \frac{e'}{m_2(H_1, H_2)} > v' - 2 + \frac{1}{m_2(H_2)} - \frac{e'}{d_2(R', H_2)} = 0,
\end{equation*}
using the definition of $d_2(R', H_2)$. Thus \eqref{eq:kappa1bound} holds for \[\kappa_1 = \min_{R' \subsetneq R}
\left\{ \frac{1}{m_2(H_2)}, \ v' - 2 + \frac{1}{m_2(H_2)} - \frac{e'}{m_2(H_1, H_2)}\right\}.\]\end{proof} 
We say that algorithm \textsc{Grow} encounters a \emph{degeneracy of type 2} in iteration $i$ of the while-loop if, when we call \textsc{Extend-L}$(F_i, e, G')$, the graph $L$ found in line~\ref{line:linl*} overlaps with $F_i$ in more than 2 vertices, 
or if there exists some edge $e' \in E(L)\setminus E(F_i)$ such that the graph $R_{e'}$ found in line~\ref{line:re'} overlaps in more than 2 vertices with $F'$.
The following result corresponds to Claim 22 in \cite{msss}. As in the proof of Claim 22 in \cite{msss}, we transform $F'$ into the output of a non-degenerate iteration $F^*$ in three steps.
However, we swap the order of the latter two steps in our proof. More precisely, we transform $F'$ into a graph $F^2 \in \mathcal{H}(F_i, e, H_1, H_2)$ in the first two steps, then transform $F^2$ into a graph $F^3 := F^* \in \mathcal{H}^*(F_i, e, H_1, H_2)$. In this last step we require Lemma~\ref{lemma:21}.

\begin{claim}\label{claim:degen2}
    There exists a constant $\kappa_2 = \kappa_2(H_1, H_2) > 0$ such that if procedure \textsc{Grow} encounters a degeneracy of type 2 in iteration $i$ of the while-loop, we have \[\lambda(F_{i+1}) \leq \lambda(F_i) - \kappa_2.\]
\end{claim}

\begin{proof}
Let $F := F_i$ be the graph passed to \textsc{Extend-L} and let $F' := F_{i+1}$ be its output.
We aim to show that there exists a constant $\kappa_2 = \kappa_2(H_1, H_2) > 0$ such that 
\begin{equation}\label{eq:kappa2full}
    \lambda(F) - \lambda(F') = v(F) - v(F') - \frac{e(F) - e(F')}{m_2(H_1, H_2)} \geq \kappa_2.
\end{equation}
Recall that $F'$ would be one of a constant number of graphs if iteration $i$ was non-degenerate.
Our strategy is to transform $F'$ into the output of such a non-degenerate iteration $F^*$ in three steps
\[F' =: F^0 \overset{(i)}{\to} F^1 \overset{(ii)}{\to} F^2 \overset{(iii)}{\to} F^3 := F^*,\] with each step carefully resolving a different facet of a degeneracy of type 2. By Claim~\ref{claim:non-degen}, we have $\lambda(F) = \lambda(F^*)$, hence we have that 
\begin{eqnarray*}
      \lambda(F) - \lambda(F') & = \lambda(F^*) - \lambda(F') = \sum_{j=1}^3\left(\lambda(F^j) - \lambda(F^{j-1})\right) \\
     & = \sum_{j=1}^3\left(v(F^j) - v(F^{j-1}) - \frac{e(F^j) - e(F^{j-1})}{m_2(H_1, H_2)}\right).
\end{eqnarray*}
We shall show that there exists $\kappa_2 = \kappa_2(H_1, H_2) > 0$ such that 
\begin{equation}\label{eq:kappa2partial}
    \left(v(F^j) - v(F^{j-1}) - \frac{e(F^j) - e(F^{j-1})}{m_2(H_1, H_2)}\right) \geq \kappa_2
\end{equation}
for each $j \in \{1,2,3\}$, whenever $F^j$ and $F^{j-1}$ are not isomorphic. 
In each step we will look at a different structural property of $F'$ that may result from a degeneracy of type 2. 
We do not know the exact structure of $F'$, and so, for each $j$, step $j$ may not necessarily modify $F^{j-1}$. 
However, since $F'$ is not isomorphic to $F^*$, as $F'$ resulted from a degeneracy of type 2, we know that for at least one $j$ that $F^j$ is not isomorphic to $F^{j-1}$. This will allow us to conclude \eqref{eq:kappa2full} from \eqref{eq:kappa2partial}.

We will now analyse the graph that \textsc{Extend-L} attaches to $F$ when a degeneracy of type 2 occurs. 
First of all, \textsc{Extend-L} attaches a graph $L \cong H_2$ to $F$ such that $L \in \mathcal{L}^*_{G'}$. 
Let $x$ be the number of new vertices that are added onto $F$ when $L$ is attached, that is, $x = |V(L)\setminus (V(F) \cap V(L))|$.
Since $L$ overlaps with the edge $e$ determined by \textsc{Eligible-Edge} in line~\ref{line:eligible-edge} of \textsc{Grow}, we must have that $x \leq v_2 - 2$. 
Further, as $L \in \mathcal{L}_{G'}^*$, every edge of $L$ is covered by a copy of $H_1$.
Thus, since the condition in line~\ref{line:growRline} of \textsc{Grow} came out as false in iteration $i$, we must have that 

\begin{equation}\label{eq:uv}
    \mbox{for all}\ u,v \in V(F) \cap V(L),\ \mbox{if}\ uv \in E(L)\ \mbox{then}\ uv \in E(F).
\end{equation} 
(By Claim~\ref{claim:grow}, \eqref{eq:uv} implies that $x \geq 1$ since $F$ must be extended by at least one edge.)

Let $L' \subseteq L$ denote the subgraph of $L$ obtained by removing every edge in $E(F) \cap E(L)$. Observe that $|V(L')| = |V(L)| = v_2$ and $|E(L')| \geq 1$ (see the remark above). 
\textsc{Extend-L} attaches to each edge $e' \in E(L')$ a copy $R_{e'}$ of $H_1$ in line~\ref{line:f'f're'} such that $E(L') \cap E(R_{e'}) = \{e'\}$.  
As the condition in line~\ref{line:growRline} of \textsc{Grow} came out as false, each graph $R_{e'}$ intersects $F$ in at most one vertex and, hence, zero edges. Let \[L_R' := L' \cup \bigcup_{e' \in E(L')} R_{e'}.\] 
Then $F'$ is the same as $F \cup L_R'$, and since every graph $R_{e'}$ contains at most one vertex of $F$, we have that $E(F') = E(F)\ \dot{\cup}\ E(L_R')$. Therefore, \[e(F') - e(F) = e(L_R').\]
Observe that $|V(F) \cap V(L')| = v_2 - x$ and so
\begin{alignat*}{2}
      v(F') - v(F)  & = v(L_R') - |V(F) \cap V(L_R')| \\
                    & = v(L_R') - (v_2 - x) - |V(F) \cap (V(L_R')\setminus V(L'))|. 
\end{alignat*}
\textbf{Transformation (i): $F^0 \to F^1$.} If $|V(F) \cap (V(L_R')\setminus V(L'))| \geq 1$, then we apply transformation (i), mapping $F^0$ to $F^1$:
For each vertex $v \in V(F) \cap (V(L_R')\setminus V(L'))$, transformation (i) introduces a new vertex $v'$. 
Every edge incident to $v$ in $E(F)$ remains connected to $v$ and all those edges incident to $v$ in $E(L_R')$ are redirected to $v'$. 
In $L_R'$ we replace the vertices in $V(F) \cap (V(L_R')\setminus V(L'))$ with the new vertices.
So now we have $|V(F) \cap (V(L_R')\setminus V(L'))| = 0$.  
Since $E(F) \cap E(L_R') = \emptyset$, the output of this transformation is uniquely defined. 
Moreover, the structure of $L_R'$ is completely unchanged. Hence, since $|V(F) \cap (V(L_R')\setminus V(L'))| \geq 1$, and  $|E(F')| = |E(F)\ \dot{\cup}\ E(L_R')|$ remained the same after transformation (i), we have that
\begin{equation*}
    v(F^1) - v(F^{0}) - \frac{e(F^1) - e(F^{0})}{m_2(H_1, H_2)} = |V(F) \cap (V(L_R')\setminus V(L'))| \geq 1.
\end{equation*}
\textbf{Transformation (ii): $F^1 \to F^2$.} Recall the definition of $\mathcal{H}(F, e, H_1, H_2)$. If $x \leq v_2 - 3$, then we apply transformation (ii), mapping $F^1$ to $F^2$ by replacing $L_{R}'$ with a graph $L_{R}''$ such that $F \cup L_{R}'' \in \mathcal{H}(F, e, H_1, H_2)$.

If $x = v_2 - 2$, observe that already $F \cup L_{R}' \in \mathcal{H}(F, e, H_1, H_2)$ and we continue to transformation (iii). So assume $x \leq v_2 - 3$.
Consider the proper subgraph $L_F := L[V(F) \cap V(L)] \subsetneq L$ obtained by removing all $x$ vertices in $V(L) \setminus V(F)$ and their incident edges from $L$.
Observe that $v(L_F) = v_2 - x \geq 3$ and also that $L_F \subseteq F$ by \eqref{eq:uv}.
Assign labels to $V(L_F)$ so that $V(L_F) = \{y,z,w_1, \ldots, w_{v_2 - (x+2)}\}$ where $e = \{y,z\}$ and $w_1, \ldots, w_{v_2 - (x+2)}$ are arbitrarily assigned.
At the start of transformation (ii), we create $v_2 - (x+2)$ new vertices $w_1', \ldots, w'_{v_2 - (x+2)}$ and also new edges such that $\{y,z,w_1', \ldots, w'_{v_2 - (x+2)}\}$ induces a copy $\hat{L}_F$ of $L_F$, and for all $i,j \in \{1, \ldots, v_2 - (x+2)\}$, $i \neq j$, \begin{alignat*}{2}
     & \mbox{if} \ w_iw_j \in E(L_F)\ \mbox{then} \ w'_iw'_j \in E(\hat{L}_F); \\
     & \mbox{if} \ w_iy \in E(L_F)\ \mbox{then} \ w'_iy \in E(\hat{L}_F); \\
     & \mbox{if} \ w_iz \in E(L_F)\ \mbox{then} \ w'_iz \in E(\hat{L}_F); \\
     & \mbox{and} \ e = yz \in E(\hat{L}_F).
\end{alignat*} We also transform $L_R'$. For each edge in $E(L_{R}')$ incident to a vertex $w_i$ in $L_F$, redirect the edge to $w'_i$, and remove $w_1, \ldots, w_{v_2 - (x+2)}$ from $V(L_R')$. 
Hence the structure of $L_R'$ remains the same except for the vertices $w_1, \ldots, w_{v_2 - (x+2)}$ that we removed.
Define $L'' := L_R' \cup \hat{L}_F$ and observe that $V(L'') \cap V(F) = e$.  

Continuing transformation (ii), for each $e' \in E(\hat{L}_{F})\setminus \{e\}$ , attach a copy $R_{e'}$ of $H_1$ to $L''$ such that $E(R_{e'}) \cap E(L'' \cup F) = \{e'\}$ and $V(R_{e'}) \cap V(L'' \cup F) = e'$.
That is, all these new copies $R_{e'}$ of $H_1$ are, in some sense, pairwise disjoint. Observe that $E(L'') \cap E(F) = \{e\}$ and
define $$L_R'' := L'' \cup \bigcup_{e' \in E(\hat{L}_F)\setminus \{e\}} R_{e'}$$ Then $F \cup L_R'' \in \mathcal{H}(F, e, H_1, H_2)$. (See Figure~\ref{fig:trans2} for an example of transformation (ii).)
Let $F^2 := F \cup L_R''$. Then, \begin{alignat*}{2}
        \                            & v(F^2) - v(F^1) - \frac{e(F^2) - e(F^1)}{m_2(H_1, H_2)} \\
    = \                               & (e(\hat{L}_F) - 1)(v_1 - 2) + v(\hat{L}_F) - 2 - \frac{(e(\hat{L}_F) - 1)e_1}{m_2(H_1, H_2)} \\
    = \                               & v(\hat{L}_F) - 2 - \frac{(e(\hat{L}_F) - 1)}{m_2(H_2)} \\
     =  \                              & \frac{(v(\hat{L}_F) - 2)\left(m_2(H_2) - \frac{e(\hat{L}_F) - 1}{v(\hat{L}_F) - 2}\right)}{m_2(H_2)} \\
    \geq \                               & \delta_1
\end{alignat*} for some $\delta_1 = \delta_1(H_1, H_2) > 0$,
where the second equality follows from $H_1$ being (strictly) balanced with respect to $d_2(\cdot, H_2)$, the third equality follows from $v(\hat{L}_F) = v(L_F) \geq 3$ and the last inequality follows from $\hat{L}_F$ being a copy of $L_F \subsetneq L \cong H_2$ and $H_2$ being \emph{strictly} $2$-balanced.

\begin{figure}[!ht]
\begin{center}
\definecolor{ffqqqq}{rgb}{1,0,0}
\definecolor{yqqqyq}{rgb}{0.8,0,0.8}
\begin{tikzpicture}[line cap=round,line join=round,>=triangle 45,x=1cm,y=1cm]
\draw [line width=2pt,color=yqqqyq] (-8,0)-- (-7,0);
\draw [line width=2pt] (-7,0)-- (-6,0);
\draw [line width=2pt,color=yqqqyq] (-6,0)-- (-5,-1);
\draw [line width=2pt,color=yqqqyq] (-5,-1)-- (-5,-2);
\draw [line width=2pt,color=yqqqyq] (-5,-2)-- (-7,0);
\draw [line width=2pt,color=ffqqqq] (-9,1)-- (-8,0);
\draw [line width=2pt,color=ffqqqq] (-9,1)-- (-9,2);
\draw [line width=2pt,color=ffqqqq] (-9,2)-- (-8,3);
\draw [line width=2pt,color=ffqqqq] (-8,3)-- (-6,0);
\draw [line width=2pt] (-10.5,1.5)-- (-9,1);
\draw [line width=2pt] (-10.5,1.5)-- (-9,2);
\draw [line width=2pt] (-10.5,1.5)-- (-8,0);
\draw [line width=2pt] (-10.5,1.5)-- (-8,3);
\draw [line width=2pt] (-8,3)-- (-7.5,4);
\draw [line width=2pt] (-7.5,4)-- (-6,0);
\draw [line width=2pt] (-2.5,1.5)-- (-1,2);
\draw [line width=2pt] (-2.5,1.5)-- (-1,1);
\draw [line width=2pt] (-2.5,1.5)-- (0,0);
\draw [line width=2pt] (-2.5,1.5)-- (0,3);
\draw [line width=2pt,color=ffqqqq] (0,3)-- (-1,2);
\draw [line width=2pt,color=ffqqqq] (-1,2)-- (-1,1);
\draw [line width=2pt,color=ffqqqq] (-1,1)-- (0,0);
\draw [line width=2pt,color=ffqqqq] (0,3)-- (1,3);
\draw [line width=2pt,color=yqqqyq] (1,3)-- (2,2);
\draw [line width=2pt,color=yqqqyq] (2,2)-- (2,1);
\draw [line width=2pt,color=yqqqyq] (2,1)-- (1,0);
\draw [line width=2pt,color=yqqqyq] (0,0)-- (1,0);
\draw [line width=2pt] (1,0)-- (2,0);
\draw [line width=2pt] (2,0)-- (3,-1);
\draw [line width=2pt] (3,-1)-- (3,-2);
\draw [line width=2pt] (3,-2)-- (1,0);
\draw [line width=2pt] (1,0)-- (2.5,0.5);
\draw [line width=2pt] (2,1)-- (2.5,0.5);
\draw [line width=2pt] (2,1)-- (3,1.5);
\draw [line width=2pt] (2,2)-- (3,1.5);
\draw [line width=2pt] (0,3)-- (0.5,4);
\draw [line width=2pt] (1,3)-- (0.5,4);
\draw [line width=2pt] (1,3)-- (2,3);
\draw [line width=2pt] (2,2)-- (2,3);
\draw (0.5,-1.8) node[anchor=north west] {\LARGE $F$};
\draw [color=yqqqyq](0.45,2) node[anchor=north west] {\LARGE $\hat{L}_{F}$};
\draw [color=yqqqyq](-5.25,0) node[anchor=north west] {\LARGE $L_{F}$};
\draw (-7.5,-1.8) node[anchor=north west] {\LARGE $F$};
\draw [line width=2pt] (-8,0)-- (-8.5,-1.5);
\draw [line width=2pt] (-8,0)-- (-8,-1.5);
\draw [line width=2pt] (-8,0)-- (-7.5,-1.5);
\draw [line width=2pt] (-7,0)-- (-7,-1.5);
\draw [line width=2pt] (-7,0)-- (-6.5,-1.5);
\draw [line width=2pt] (-5,-2)-- (-6,-2);
\draw[color=black] (-3.85,2.25) node {\bf (ii)};
\draw [->,line width=5.2pt] (-4.5,1.5) -- (-3,1.5);
\draw [line width=2pt] (0,0)-- (-0.5,-1.5);
\draw [line width=2pt] (0,0)-- (0,-1.5);
\draw [line width=2pt] (0,0)-- (0.5,-1.5);
\draw [line width=2pt] (1,0)-- (1,-1.5);
\draw [line width=2pt] (1,0)-- (1.5,-1.5);
\draw [line width=2pt] (3,-2)-- (2,-2);
\begin{scriptsize}
\draw [fill=black] (-8,0) circle (2.5pt);
\draw[color=black] (-8,0.4) node {$y$};
\draw [fill=black] (-7,0) circle (2.5pt);
\draw[color=black] (-7,0.4) node {$z$};
\draw [fill=black] (-6,0) circle (2.5pt);
\draw[color=black] (-5.55,0) node {$w_{1}$};
\draw [fill=black] (-5,-1) circle (2.5pt);
\draw[color=black] (-4.55,-1) node {$w_{2}$};
\draw [fill=black] (-5,-2) circle (2.5pt);
\draw[color=black] (-4.55,-2) node {$w_{3}$};
\draw [fill=black] (-9,1) circle (2.5pt);
\draw [fill=black] (-9,2) circle (2.5pt);
\draw [fill=black] (-8,3) circle (2.5pt);
\draw[color=black] (-8.1,3.3) node {$a$};
\draw [fill=black] (-10.5,1.5) circle (2.5pt);
\draw [fill=black] (-7.5,4) circle (2.5pt);
\draw[color=black] (-7.65,4.3) node {$b$};
\draw [fill=black] (-2.5,1.5) circle (2.5pt);
\draw [fill=black] (-1,2) circle (2.5pt);
\draw [fill=black] (-1,1) circle (2.5pt);
\draw [fill=black] (0,0) circle (2.5pt);
\draw[color=black] (0,0.4) node {$y$};
\draw [fill=black] (0,3) circle (2.5pt);
\draw[color=black] (-0.1,3.3) node {$a$};
\draw [fill=black] (1,3) circle (2.5pt);
\draw[color=black] (0.35,4.3) node {$b$};
\draw[color=black] (1.25,3.25) node {$w'_{1}$};
\draw [fill=black] (2,2) circle (2.5pt);
\draw[color=black] (1.55,2) node {$w'_{2}$};
\draw [fill=black] (2,1) circle (2.5pt);
\draw[color=black] (1.55,1) node {$w'_{3}$};
\draw [fill=black] (1,0) circle (2.5pt);
\draw[color=black] (1,0.4) node {$z$};
\draw [fill=black] (2,0) circle (2.5pt);
\draw[color=black] (2.45,0) node {$w_1$};
\draw [fill=black] (3,-1) circle (2.5pt);
\draw[color=black] (3.45,-1) node {$w_2$};
\draw [fill=black] (3,-2) circle (2.5pt);
\draw[color=black] (3.45,-2) node {$w_3$};
\draw [fill=black] (2.5,0.5) circle (2.5pt);
\draw [fill=black] (3,1.5) circle (2.5pt);
\draw [fill=black] (0.5,4) circle (2.5pt);
\draw [fill=black] (2,3) circle (2.5pt);
\end{scriptsize}
\end{tikzpicture}
\caption{An example of transformation (ii) where $H_1 = K_3$ and $H_2 = C_8$. Observe that edges $aw_1$ and $bw_1$ are replaced by edges $aw'_1$ and $bw'_1$.}\label{fig:trans2}
\end{center}
\end{figure}
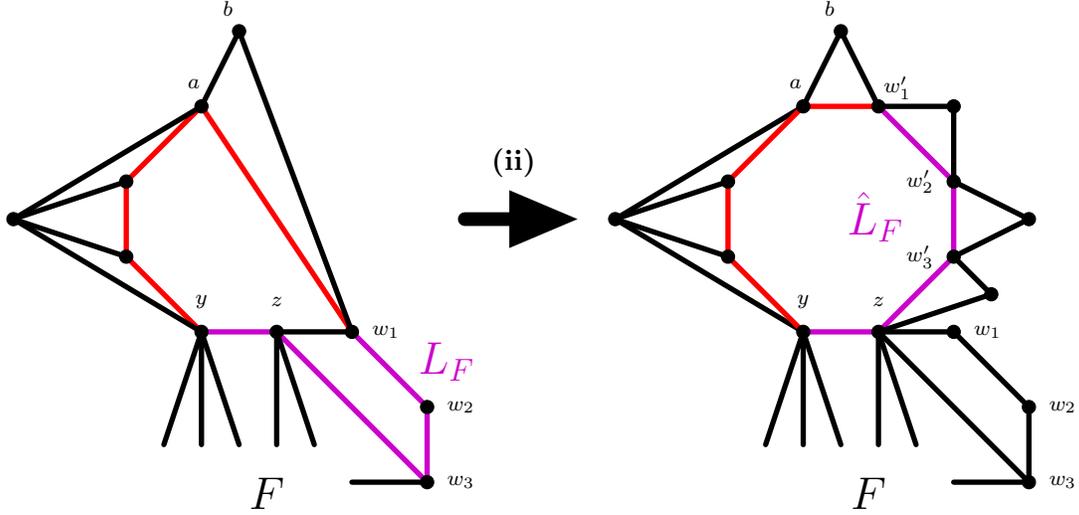

\textbf{Transformation (iii): $F^2 \to F^3$.} Recall that for any $J \in \mathcal{H}(F, \hat{e}, H_1, H_2)$, we define $v^{+}(J) := |V(J)\setminus V(F)| = v(J) - v(F)$ and $e^{+}(J) := |E(J)\setminus E(F)| = e(J) - e(F)$. Remove the edge $e$ from $E(L'')$ (and $E(L''_R)$) to give $E(L'') \cap E(F) = \emptyset$. Then \[e^+(F \cup L_R'') = e(L_R'')\] and \[v^+(F \cup L_R'') = v(L_R'') - 2.\]

If $F^2 = F \cup L_R'' \in \mathcal{H}^*(F, e, H_1, H_2)$, then transformation (iii) sets $F^3 := F^2$. Otherwise we have that $F \cup L_R'' \in \mathcal{H}(F, e, H_1, H_2)\setminus \mathcal{H}^*(F, e, H_1, H_2)$.
Let $F^3 := J^*$ where $J^*$ is any member of $\mathcal{H}^*(F, e, H_1, H_2)$ and recall that, indeed, $J^*$ is a possible output of a non-degenerate iteration of the while-loop of \textsc{Grow}. 

Then, in transformation (iii), we replace $F \cup L_R''$ with the graph $J^*$. Since $H_2$ is (strictly) $2$-balanced and $H_1$ is (strictly) balanced with respect to $d_2(\cdot, H_2)$, we have that 
\begin{equation}\label{eq:m2h1h2}
    m_2(H_1, H_2) = \frac{e_1}{v_1 - 2 + \frac{1}{m_2(H_2)}} = \frac{e_1(e_2 - 1)}{(v_1 - 2)(e_2 - 1) + v_2 - 2} = \frac{e^{+}(J^*)}{v^{+}(J^*)}.
\end{equation}
Using \eqref{eq:m2h1h2} and Lemma~\ref{lemma:21}, and that $H_2$ is (strictly) $2$-balanced and $H_1$ is (strictly) balanced with respect to $d_2(\cdot, H_2)$, we have that
\begin{alignat*}{2}
      \                              & v(F^3) - v(F^2) - \frac{e(F^3) - e(F^2)}{m_2(H_1, H_2)} \\
    = \                              & v(J^*) - v(F \cup L_R'') - \frac{e(J^*) - e(F \cup L_R'')}{m_2(H_1, H_2)}  \\
    = \                              & v^+(J^*) - v^+(F \cup L_R'') - \frac{e^+(J^*) - e^+(F \cup L_R'')}{m_2(H_1, H_2)}  \\
    \overset{L.\ref{lemma:21}}{>}  \  & v^+(J^*) - v^+(F \cup L_R'') - \frac{e^+(J^*) - e^+(J^*)\left(\frac{v^+(F \cup L_R'')}{v^+(J^*)}\right)}{m_2(H_1, H_2)}  \\
     = \                              & \left(v^+(J^*) - v^+(F \cup L_R'')\right) \left(1 - \frac{e^+(J^*)}{v^+(J^*)m_2(H_1, H_2)}\right) \\
    = \                              & 0.  \end{alignat*} Since $v^+(J^*)$, $v^+(F \cup L_R'')$, $e^+(J^*)$, $e^+(F \cup L_R'')$ and $m_2(H_1, H_2)$ only rely on $H_1$ and $H_2$, there exists $\delta_2 = \delta_2(H_1, H_2) > 0$ such that \[v(F^3) - v(F^2) - \frac{e(F^3) - e(F^2)}{m_2(H_1, H_2)} \geq \delta_2.\]
Taking \[\kappa_2 := \min\{1, \delta_1, \delta_2\}\] we see that \eqref{eq:kappa2partial} holds.\end{proof}

As stated earlier, Claim~\ref{claim:degenfull} follows from Claims~\ref{claim:degen1} and \ref{claim:degen2}. All that remains to prove Case~1 of Lemma~\ref{lemma:noerror} is to prove Lemma~\ref{lemma:21}.

\subsection{Proof of Lemma~\ref{lemma:21}}\label{sec:lemma21}

Let $J \in \mathcal{H}(F, \hat{e}, H_1, H_2)\setminus \mathcal{H}^*(F, \hat{e}, H_1, H_2)$. We choose the graph $J^* \in \mathcal{H}^*(F, \hat{e}, H_1, H_2)$ with the following properties.

\begin{itemize}
    \item[1)] The edge of the copy $H_{\hat{e}}$ of $H_2$ in $J$ attached at $\hat{e}$ and its orientation when attached are the same as the edge of the copy $H^*_{\hat{e}}$ of $H_2$ in $J^*$ attached at $\hat{e}$ and its orientation when attached; 
    \item[2)] for each $f \in E_{J}$, the edge of the copy $H_f$ of $H_1$ in $J$ attached at $f$ and its orientation when attached are the same as the edge of the copy $H^*_f$ of $H_1$ in $J^*$ attached at $f$ and its orientation when attached.
\end{itemize}

Then, recalling definitions from the beginning of Section~\ref{sec:claimdegenfull}, we have that $V_J = V_{J^*}$ and $E_J = E_{J^*}$; that is, $H_{\hat{e}}^{J} = H_{\hat{e}}^{J^*}$. 
From now on, let $V := V_J$, $E := E_J$ and $H_{\hat{e}}^{-} := H_{\hat{e}}^{J}$.
Observe for \emph{all} $J' \in \mathcal{H}^*(F, \hat{e}, H_1, H_2)$, that
\begin{equation}\label{eq:j'density}
    \frac{e^{+}(J')}{v^{+}(J')} = \frac{e_1(e_2 - 1)}{(v_1 - 2)(e_2 - 1) + v_2 - 2}.
\end{equation} 
Hence, to prove Lemma~\ref{lemma:21} it suffices to show
\begin{equation*}
\frac{e^{+}(J)}{v^{+}(J)} > \frac{e^{+}(J^*)}{v^{+}(J^*)}.
\end{equation*}
As in \cite{msss}, the intuition behind our proof is that $J^*$ can be transformed into $J$ by successively merging the copies $H^*_f$ of $H_1$ in $J^*$ with each other and vertices in $H^{-}_{\hat{e}}$.
We do this in $e_2 - 1$ steps, fixing carefully a total ordering of the inner edges $E$.
For every edge $f \in E$, we merge the attached outer copy $H^*_f$ of $H_1$ in $J^*$ with copies of $H_1$ (attached to edges preceding $f$ in our ordering) and vertices of $H^{-}_{\hat{e}}$.
Throughout, we keep track of the number of edges $\Delta_e(f)$ and the number of vertices $\Delta_v(f)$ vanishing in this process.
One could hope that the $F$-external density of $J$ increases in every step of this process, or, even slightly stronger,
that $\Delta_e(f) / \Delta_v(f) < e^{+}(J^*) / v^{+}(J^*)$. This does not necessarily hold, but we will show that there exists a collection of edge-disjoint subgraphs $A_i$ of $H^{-}_{\hat{e}}$ such that, for each $i$, the edges of $E(A_i)$ are `collectively good' for this process and every edge not belonging to one of these $A_i$ is also `good' for this process.

Recalling definitions from the beginning of Section~\ref{sec:claimdegenfull}, let $H^{-}_f := (U_J(f)\ \dot{\cup}\ f, D_J(f))$ denote the subgraph obtained by removing the edge $f$ from the copy $H_f$ of $H_1$ in $J$. 

Later, we will carefully define a (total) ordering $\prec$ on the inner edges $E$.\footnote{For clarity, for any $f \in E$, $f \not\prec f$ in this ordering $\prec$.} For such an ordering $\prec$
and each $f \in E$, define \[\Delta_E(f) := D_J(f) \cap \left(\bigcup_{f' \prec f} D_J(f')\right),\] and \[\Delta_V(f) := U_J(f) \cap \left(\left(\bigcup_{f' \prec f} U_J(f')\right) \cup V(H_{\hat{e}}^{-})\right),\] 
and set $\Delta_e(f) := |\Delta_E(f)|$ and $\Delta_v(f) := |\Delta_V(f)|$. We emphasise here that the definition of $\Delta_v(f)$ takes into account how vertices of outer vertex sets can intersect with the inner graph $H_{\hat{e}}^{-}$.
One can see that $\Delta_e(f)$ ($\Delta_v(f)$) is the number of edges (vertices) vanishing from $H^*_f$ when it is merged with preceding attached copies of $H_1$ and $V(H_{\hat{e}}^{-})$.

By our choice of $J^*$, one can quickly see that
\begin{equation}\label{eq:sume}
    e^{+}(J) = e^{+}(J^*) - \sum_{f \in E} \Delta_e(f)
\end{equation}
and
\begin{equation}\label{eq:sumv}
    v^{+}(J) = v^{+}(J^*) - \sum_{f \in E} \Delta_v(f).
\end{equation}
By \eqref{eq:sume} and \eqref{eq:sumv}, we have
\begin{equation*}
    \frac{e^{+}(J)}{v^{+}(J)} = \frac{e^{+}(J^*) - \sum_{f \in E} \Delta_e(f)}{v^{+}(J^*) - \sum_{f \in E} \Delta_v(f)}.
\end{equation*}
Then, by Fact~\ref{fact:ineq}, to show that 
\begin{equation*}
\frac{e^{+}(J)}{v^{+}(J)} > \frac{e^{+}(J^*)}{v^{+}(J^*)}
\end{equation*}
it suffices to prove that 
\begin{equation}\label{eq:conclusion2}
    \frac{\sum_{f \in E} \Delta_e(f)}{\sum_{f \in E} \Delta_v(f)} < \frac{e^{+}(J^*)}{v^{+}(J^*)}.\footnote{Note that $\sum_{f \in E} \Delta_v(f) \geq 1$ as otherwise $J = J^*$.}
\end{equation}
To show \eqref{eq:conclusion2}, we will now carefully order the edges of $E$ using an algorithm \textsc{Order-Edges} (Figure~\ref{orderedgesfig}).
The algorithm takes as input the graph $H_{\hat{e}}^{-} = (V \ \dot{\cup} \ \hat{e},E)$ and outputs a stack $s$ containing every edge from $E$ and a collection of edge-disjoint edge sets $E_i$ in $E$ and (not necessarily disjoint) vertex sets $V_i$ in $V \ \dot{\cup} \ \hat{e}$. 
We take our total ordering $\prec$ of $E$ to be that induced by the order in which edges of $E$ were placed onto the stack $s$ (that is, $f \prec f'$ if and only if $f$ was placed onto the stack $s$ before $f'$).
Also, for each $i$, we define $A_i := (V_i, E_i)$ to be the graph on vertex set $V_i$ and edge set $E_i$ and observe that $A_i \subsetneq H_2$.
We will utilise this ordering and our choice of $E_i$ and $V_i$ for each $i$, alongside that $H_2$ is (strictly) $2$-balanced and $H_1$ is \emph{strictly} balanced with respect to $d_2(\cdot, H_2)$, in order to conclude \eqref{eq:conclusion2}. 

\begin{figure}
\begin{algorithmic}[1]
\Procedure{\sc Order-Edges}{$H_{\hat{e}}^{-} = (V\  \dot{\cup}\  \hat{e},E)$}
    \State $s\gets$ {\sc empty-stack}()\label{line:oes}
    \ForAll {$i \in [\lfloor e_2/2 \rfloor]$}\label{line:oeforalli}
        \State $E_i \gets \emptyset$\label{line:oeej}
        \State $V_i \gets \emptyset$\label{line:oevj}
    \EndFor\label{line:oeendfor1}
    \State $j \gets 1$\label{line:oeigets1}
    \State $E'\gets E$\label{line:oee'getse}
    \While{$E' \neq \emptyset$}\label{line:oee'notempty}
            \If{$\exists f, f' \in E' \ \mbox{s.t.} \ (f \neq f') \land (D_J(f) \cap D_J(f') \neq \emptyset)$}\label{line:oeaiconstructionstart}
                \State $s$.{\sc push}($f$)\label{line:oespushf}
                \State $E_j$.{\sc push}($f$)\label{line:oeeipushf}
                \State $E'$.{\sc remove}($f$) \label{line:oee'removef}
                \State $V_j \gets f \cup (U_J(f) \cap V(H_{\hat{e}}^{-}))$\label{line:oeviupdate1}
                \While{$\exists \ uw \in E' \ \mbox{s.t.} \ \left(u,w \in V_j\right) \lor \left(D_J(uw) \cap \bigcup_{f \in E_j} D_J(f) \neq \emptyset\right)$}\label{line:oeaibadedgescollectedup}
                    \State $s$.{\sc push}($uw$)\label{line:oespushuw}
                    \State $E_j$.{\sc push}($uw$)\label{line:oeeipushuw}
                    \State $E'$.{\sc remove}($uw$) \label{line:oee'removeuw}
                    \State $V_j \gets \bigcup_{f \in E_j}\left(f \cup (U_J(f) \cap V(H_{\hat{e}}^{-})) \right)$\label{line:oeviupdate3}
                \EndWhile\label{line:oeendwhile1}
                \State $j \gets j + 1$\label{line:oeigetsiplus1}    
            \Else\label{line:oeelse2}
                \ForAll {$f \in E'$}\label{line:oeconclusion}
                \State $s$.{\sc push}($f$)\label{line:oespushlastf}
                 \State $E'$.{\sc remove}($f$) \label{line:oee'removelastf}
                \EndFor \label{line:oeendfor2}
            \EndIf\label{line:oeendif2}
    \EndWhile\label{line:oeendwhile2}
    \State\Return $s$\label{line:oereturns}
    \ForAll {$i \in [\lfloor e_2/2 \rfloor]$ s.t. $E_{i} \neq \emptyset$}\label{line:oelastforalli}
    \State    \Return $E_i$\label{line:oereturnej}
    \State    \Return $V_i$\label{line:oereturnvj}
    \EndFor\label{line:oelastendfor}
\EndProcedure\label{line:oeendprocedure}
\end{algorithmic}
\caption{The implementation of algorithm \textsc{Order-Edges}.}\label{orderedgesfig}
\end{figure}

Let us describe algorithm \textsc{Order-Edges} (Figure~\ref{orderedgesfig}) in detail.  In lines~\ref{line:oes}-\ref{line:oee'getse}, we initialise several parameters: a stack $s$, which we will place edges of $E$ on during our algorithm; sets $E_i$ and $V_i$ for each $i \in [\lfloor e_2/2 \rfloor]$,\footnote{\textsc{Order-Edges} can output at most $\lfloor e_2/2 \rfloor$ pairs of sets $E_i$ and $V_i$.} which we will add edges of $E$ and vertices of $V \ \dot{\cup} \ \hat{e}$ into, respectively;
an index $j$, which will correspond to whichever graph $A_j$ we consider constructing next; a set $E'$, which will keep track of those edges of $E$ we have not yet placed onto the stack $s$. 
Line~\ref{line:oee'notempty} ensures the algorithm continues until $E' = \emptyset$, that is, until all the edges of $E$ have been placed onto $s$.

In line~\ref{line:oeaiconstructionstart}, we begin constructing $A_j$ by finding a pair of distinct edges in $E'$ whose outer edge sets (in $J$) intersect.
In lines~\ref{line:oespushf}-\ref{line:oeviupdate1}, we place one of these edges, $f$, onto $s$, into $E_j$ and remove it from $E'$. We also set $V_j$ to be the two vertices in $f$ alongside any vertices in the outer vertex set $U_J(f)$ that intersect $V(H^{-}_{\hat{e}})$.

In lines~\ref{line:oeaibadedgescollectedup}-\ref{line:oeendwhile1}, we iteratively add onto $s$, into $E_j$ and remove from $E'$ any edge $uw \in E'$ which either connects two vertices previously added to $V_j$ or has an outer edge set $D_J(uw)$ that intersects the collection of outer edge sets of edges previously added to $E_j$.
We also update $V_j$ in each step of this process.

In line~\ref{line:oeigetsiplus1}, we increment $j$ in preparation for the next check at line~\ref{line:oeaiconstructionstart} (if we still have $E' \neq \emptyset$). 
If the condition in line~\ref{line:oeaiconstructionstart} fails then in lines~\ref{line:oeconclusion}-\ref{line:oeendfor2} we arbitrarily place the remaining edges of $E'$ onto the stack $s$. In line~\ref{line:oereturns}, we output the stack $s$ and in lines~\ref{line:oelastforalli}-\ref{line:oereturnvj} we output each non-empty $E_i$ and $V_i$.

We will now argue that each proper subgraph $A_i = (V_i, E_i)$ of $H_2$ and each edge placed onto $s$ in line~\ref{line:oespushlastf} are `good', in some sense, for us to conclude \eqref{eq:conclusion2}. 

For each  $f \in E$, define the graph \[T(f) := (\Delta_V(f) \ \dot{\cup} \ f, \Delta_E(f)) \subseteq H_f^- \subsetneq H_1.\]
Observe that one or both vertices of $f$ may be isolated in $T(f)$. This observation will be very useful later.

For each $i$ and $f \in E_i$, define \[(V(H_{\hat{e}}^{-}))_f := (V(H_{\hat{e}}^{-}) \cap U_J(f))\setminus \left(\bigcup_{\substack{f' \in E_i: \\ f' \prec f}}f' \cup \left(\bigcup_{\substack{f' \in E_i: \\ f' \prec f}} \left(V(H_{\hat{e}}^{-}) \cap U_J(f')\right)\right)\right) \subseteq \Delta_V(f)\] since $(V(H_{\hat{e}}^{-}))_f$ is a subset of $V(H_{\hat{e}}^{-}) \cap U_J(f)$. 
One can see that $(V(H_{\hat{e}}^{-}))_f$ consists of those vertices of $V(H_{\hat{e}}^{-})$ which are new to $V_i$ at the point when $f$ is added to $E_i$ but are not contained in $f$. Importantly for our purposes, every vertex in $(V(H_{\hat{e}}^{-}))_f$ is isolated in $T(f)$. 
Indeed, otherwise there exists $k < i$ and $f'' \in E_k$ such that $D_J(f) \cap D_J(f'') \neq \emptyset$ and $f$ would have been previously added to $E_k$ in line~\ref{line:oeeipushuw}. 

For each $f \in E$, let $T'(f)$ be the graph obtained from $T(f)$ by removing all isolated vertices from $V(T(f))$.
Crucially for our proof, since vertices of $f$ may be isolated in $T(f)$, one or more of them may not belong to $V(T'(f))$. 
Further, no vertex of $(V(H_{\hat{e}}^{-}))_f$ is contained in $V(T'(f))$.

For all $f \in E$ with $\Delta_e(f) \geq 1$, since $T'(f) \subsetneq H_1$ and $H_1$ is \emph{strictly} balanced with respect to $d_2(\cdot, H_2)$, we have that 
\begin{equation}\label{eq:centralobservation}
    m_2(H_1, H_2) > d_2(T'(f), H_2) = \frac{|E(T'(f))|}{|V(T'(f))| - 2 + \frac{1}{m_2(H_2)}}.
\end{equation} 
Recall \eqref{eq:m2h1h2}, that is, \[m_2(H_1, H_2) = \frac{e^+(J^*)}{v^+(J^*)}.\]
We now make the key observation of our proof: Since vertices of $f$ may be isolated in $T(f)$, and so not contained in $V(T'(f))$, and no vertex of $(V(H_{\hat{e}}^{-}))_f$ is contained in $V(T'(f))$, we have that 
\begin{equation}\label{eq:vt'fbound}
    |V(T'(f))| \leq \Delta_v(f) + |f \cap V(T'(f))| - |(V(H_{\hat{e}}))_f|.
\end{equation} 
Hence, from \eqref{eq:centralobservation} and \eqref{eq:vt'fbound} we have that \begin{alignat}{2}
    \Delta_e(f)     =  \ &|E(T'(f))| \nonumber\\
                    < \ &m_2(H_1, H_2)\left(\Delta_v(f) - (2 - |f \cap V(T'(f))|) - |(V(H_{\hat{e}}^{-}))_f| + \frac{1}{m_2(H_2)}\right).\label{eq:deltaef}
\end{alignat}
Edges $f$ such that $|f \cap V(T'(f))| = 2$ will be, in some sense, `bad' for us when trying to conclude \eqref{eq:conclusion2}.
Indeed, if $|(V(H_{\hat{e}}^{-}))_f| = 0$ then we may have that $\frac{\Delta_e(f)}{\Delta_v(f)} \geq m_2(H_1, H_2) = \frac{e^{+}(J^*)}{v^{+}(J^*)}$, by \eqref{eq:m2h1h2}.
However, edges $f$ such that $|f \cap V(T'(f))| \in \{0,1\}$ will be, in some sense, `good' for us when trying to conclude \eqref{eq:conclusion2}. 
Indeed, since $m_2(H_2) \geq 1$, we have for such edges $f$ that $\frac{\Delta_e(f)}{\Delta_v(f)} < m_2(H_1, H_2) = \frac{e^{+}(J^*)}{v^{+}(J^*)}$.

We show in the following claim that our choice of ordering $\prec$, our choice of each $A_i$ and the fact that $H_2$ is (strictly) $2$-balanced ensure that for each $i$ there are enough `good' edges in $A_i$ to compensate for any `bad' edges that may appear in $A_i$.
 
\begin{claim}\label{claim:inei}
For each $i$, 
\begin{equation*}
    \sum_{f \in E_i} \Delta_e(f) < m_2(H_1, H_2) \sum_{f \in E_i} \Delta_v(f).
\end{equation*}
\end{claim}

\begin{proof}
Fix $i$. Firstly, as observed before, each $A_i$ is a non-empty subgraph of $H_2$. Moreover, $|E_i| \geq 2$ (by the condition in line~\ref{line:oeaiconstructionstart}). Since $H_2$ is (strictly) $2$-balanced, we have that 
\begin{equation}\label{eq:f22balancedbound}
    m_2(H_2) \geq d_2(A_i) = \frac{|E_i| - 1}{|V_i| - 2}.
\end{equation}
Now let us consider $\Delta_e(f)$ for each $f \in E_i$. For the edge $f$ added in line~\ref{line:oeeipushf}, observe that $\Delta_e(f) = 0$. 
Indeed, otherwise $\Delta_e(f) \geq 1$, and there exists $k < i$ such that $D_J(f) \cap \left(\bigcup_{f' \in E_k}D_J(f')\right) \neq \emptyset$. That is, $f$ would have been added to $E_k$ in line~\ref{line:oeeipushuw} previously in the algorithm. Since $(V(H_{\hat{e}}^{-}))_f \subseteq \Delta_V(f)$, we have that 
\begin{equation}\label{eq:aif}
    \Delta_e(f) = 0 \leq m_2(H_1, H_2)\left(\Delta_v(f) - |(V(H_{\hat{e}}^{-}))_f|\right).    
\end{equation}
Now, for each edge added in line~\ref{line:oeeipushuw}, either $v,w \in V_i$ when $uw$ was added to $E_i$, or we had $D_J(uw) \cap \left(\bigcup_{\substack{f' \in E_i: \\ f' \prec uw}}D_J(f')\right) \neq \emptyset$, that is, $\Delta_e(uw) \geq 1$, and at least one of $u$, $w$ did not belong to $V_i$ when $uw$ was added to $E_i$. In the former case, if $\Delta_e(uw) \geq 1$, then by \eqref{eq:deltaef} and that $-(2 - |f \cap V(T'(f))|) \leq 0$, we have that
\begin{equation}\label{eq:aiuwbad}
    \Delta_e(uw) < m_2(H_1, H_2)\left(\Delta_v(uw) - |(V(H_{\hat{e}}^{-}))_{uw}| + \frac{1}{m_2(H_2)}\right).
\end{equation}
Observe that \eqref{eq:aiuwbad} also holds when $\Delta_{e}(uw) = 0$. In the latter case, observe that for all $k < i$, we must have $D_{J}(uw) \cap \left(\bigcup_{f \in E_k}D_J(f)\right) = \emptyset$.
Indeed, otherwise $uw$ would have been added to some $E_k$ in line~\ref{line:oeeipushuw}. Combining this with knowing that at least one of $u$, $w$ did not belong to $V_i$ before $uw$ was added to $E_i$, we must have that one or both of $u$, $w$ are isolated in $T(uw)$.
That is, one or both do not belong to $T'(uw)$, and so $|uw \cap V(T'(uw))| \in \{0,1\}$. Thus, since $\Delta_e(uw) \geq 1$, by \eqref{eq:deltaef} we have that 
\begin{equation}\label{eq:aiuwgood1}
    \Delta_e(uw) < m_2(H_1, H_2)\left(\Delta_v(uw) - 1 - |(V(H_{\hat{e}}^{-}))_{uw}| + \frac{1}{m_2(H_2)}\right)
\end{equation} if $|uw \cap V(T'(uw))| = 1$, and

\begin{equation}\label{eq:aiuwgood2}
    \Delta_e(uw) < m_2(H_1, H_2)\left(\Delta_v(uw) - 2 - |(V(H_{\hat{e}}^{-}))_{uw}| + \frac{1}{m_2(H_2)}\right)
\end{equation} if $|uw \cap V(T'(uw))| = 0$.

In conclusion, except for the two vertices in the edge added in line~\ref{line:oeeipushf}, every time a new vertex $x$ was added to $V_i$ when some edge $f$ was added to $E_i$, either $x \in (V(H_{\hat{e}}^{-}))_{f}$, or $x \in f$ and \eqref{eq:aiuwgood1} or \eqref{eq:aiuwgood2} held, dependent on whether one or both of the vertices in $f$ were new to $V_i$. Indeed, $x$ was isolated in $T(f)$. Moreover, after the two vertices in the edge added in line~\ref{line:oeeipushf} there are $|V_i| - 2$ vertices added to $V_i$.

Hence, by \eqref{eq:aif}-\eqref{eq:aiuwgood2}, we have that 
\begin{equation}\label{eq:aifinal1}
    \sum_{f \in E_i} \Delta_e(f) < m_2(H_1, H_2)\left(\sum_{f \in E_i} \Delta_v(f) - (|V_i| - 2) + \frac{|E_i| - 1}{m_2(H_2)}\right).
\end{equation}
By \eqref{eq:f22balancedbound}, 
\begin{equation}\label{eq:aifinal2}
    - (|V_i| - 2) + \frac{|E_i| - 1}{m_2(H_2)} \leq 0.
\end{equation}
Thus, by \eqref{eq:aifinal1} and \eqref{eq:aifinal2}, we have that
\[\sum_{f \in E_i} \Delta_e(f) < m_2(H_1, H_2) \sum_{f \in E_i} \Delta_v(f)\] as desired.\end{proof}

\begin{claim}\label{claim:notinei}
For each edge $f$ placed onto $s$ in line~\ref{line:oespushlastf}, we have $\Delta_e(f) = 0$.
\end{claim}

\begin{proof}
Assume not. Then $\Delta_e(f) \geq 1$ for some edge $f$ placed onto $s$ in line~\ref{line:oespushlastf}. 
Observe that $D_J(f) \cap \left(\bigcup_{f'' \in E_i} D_J(f'')\right) = \emptyset$ for any $i$, otherwise $f$ would have been added to some $E_i$ in line~\ref{line:oeeipushuw} previously.

Thus we must have that $D_J(f) \cap D_J(f') \neq \emptyset$ for some edge $f' \neq f$ where $f'$ was placed onto $s$ also in line~\ref{line:oespushlastf}.
But then $f$ and $f'$ satisfy the condition in line~\ref{line:oeaiconstructionstart} and would both be contained in some $E_i$, contradicting that $f$ was placed onto $s$ in line~\ref{line:oespushlastf}.\end{proof}

Since $J \in \mathcal{H}(F, \hat{e}, H_1, H_2)\setminus \mathcal{H}^*(F, \hat{e}, H_1, H_2)$, we must have that $\Delta_v(f) \geq 1$ for some $f \in E$. 
Thus, if $\Delta_e(f) = 0$ for all $f \in E$ then \eqref{eq:conclusion2} holds trivially. If $\Delta_e(f) \geq 1$ for some $f \in E$, then $E_1 \neq \emptyset$ and $A_1$ is a non-empty subgraph of $H_{\hat{e}}^{-}$. 
Then, by \eqref{eq:m2h1h2} and Claims~\ref{claim:inei} and \ref{claim:notinei}, we have that 
\begin{alignat*}{2}
    \frac{\sum_{f \in E} \Delta_e(f)}{\sum_{f \in E} \Delta_v(f)} & = \frac{\sum_i\sum_{f \in E_i} \Delta_e(f) + \sum_{f \in E\setminus \cup_i E_i} \Delta_e(f)}{\sum_{f \in E} \Delta_v(f)} \\
    & < \frac{m_2(H_1, H_2)\sum_i\sum_{f \in E_i} \Delta_v(f)}{\sum_i\sum_{f \in E_i} \Delta_v(f)} \\
    & = m_2(H_1, H_2) \\
    & = \frac{e^{+}(J^*)}{v^{+}(J^*)}.
\end{alignat*}
Thus \eqref{eq:conclusion2} holds and we are done. 

\section{Case~2: $m_2(H_1, H_2) = m_2(H_2)$.}\label{sec:altcase}

In this section we prove Lemma~\ref{lemma:noerror} when $m_2(H_1) = m_2(H_2)$. 
Our proof follows that of Case~1 significantly, but uses a different algorithm \textsc{Grow-Alt}.
All definitions and notation are the same as previously unless otherwise stated. 

For any graph $F$ and edge $e \in E(F)$, we say that $e$ is \emph{eligible for extension in \textsc{Grow-Alt}} if it satisfies \[\nexists L \in \mathcal{L}_F, R \in \mathcal{R}_F \ \mbox{s.t.} \ E(L) \cap E(R) = \{e\}.\]
We note here that this is substantially different to Case~1; indeed, the set $\mathcal{C}^*$ will not feature in what follows. 
Algorithm \textsc{Grow-Alt} is shown in Figure~\ref{growsymfig}.
As with \textsc{Grow}, it has input $G' \subseteq G$, the graph that \textsc{Asym-Edge-Col} got stuck on.
\textsc{Grow-Alt} operates in a similar way to \textsc{Grow}.
In line~\ref{line:eligible}, the function \textsc{Eligible-Edge-Alt} is called which maps $F_i$ to an edge $e \in E(F_i)$ which is eligible for extension in \textsc{Grow-Alt}. 
As with \textsc{Eligible-Edge} in Case~1, this edge $e$ is selected to be unique up to isomorphism. 
We then apply a new procedure \textsc{Extend} which attaches either a graph $L \in \mathcal{L}_{G'}$ or a graph $R \in \mathcal{R}_{G'}$ that contains $e$ to $F_i$.
As in Case~1, because the condition in line~\ref{line:edgeremoval} of \textsc{Asym-Edge-Col} fails, $G' \in \mathcal{C}$.\footnote{Note that we could also conclude $G' \in \mathcal{C}^*$, however this will not be necessary as noted earlier. Also, see Section~\ref{sec:concludingremarks} for discussion on why we use \textsc{Grow-Alt} instead of \textsc{Grow} when $m_2(H_1) = m_2(H_2)$.} 

We now show that the number of edges of $F_i$ increases by at least one and that \textsc{Grow-Alt} operates as desired with a result analogous to Claim~\ref{claim:grow}. 

\begin{claim}\label{claim:growsym}
Algorithm \textsc{Grow-Alt} terminates without error on any input graph $G' \in \mathcal{C}$
that is not $(H_1, H_2)$-sparse or not an $\hat{\mathcal{A}}$-graph.\footnote{See Definitions~\ref{def:agraph} and \ref{def:h1h2sparse}.} Moreover, for every iteration $i$ of the while-loop, we have $e(F_{i+1}) > e(F_i)$.
\end{claim}

\begin{proof}

The special cases in lines~\ref{line:growaltforalle}-\ref{line:endifgrowalt} and the assignments in lines~\ref{line:growaltanye} and \ref{line:growaltanyr} operate successfully for the exact same reasons as given in the proof of Claim~\ref{claim:grow}. 

Next, we show that the call to \textsc{Eligible-Edge-Alt} in line~\ref{line:eligible} is always successful. 
Indeed, suppose for a contradiction that no edge in $F_i$ is eligible for extension in \textsc{Grow-Alt} for some $i \geq 0$. 
Then for every edge $e \in E(F_i)$ there exist $L \in \mathcal{L}_{F_i}$ and $R \in \mathcal{R}_{F_i}$ s.t. $E(L) \cap E(R) = \{e\}$, by definition.
Hence $F \in \mathcal{C}$. 
Recall that $H_1$ and $H_2$ satisfy the criteria of Conjecture~\ref{conj:aconj}. Hence $H_1$ and $H_2$ are strictly $2$-balanced and $m_2(H_1) = m_2(H_2) > 1$.
Then, by Lemma~\ref{lemma:2connected}, we have that $H_1$ and $H_2$ are both $2$-connected. Hence $F_i$ is $2$-connected by construction.
However, our choice of $F_0$ in line~\ref{line:growaltanyr} guarantees that $F_i$ is not in $\hat{\mathcal{A}}$. 
Indeed, the edge $e$ selected in line~\ref{line:growaltanye} satisfying $|\mathcal{S}_{G'}(e)| = 0$ is an edge of $F_0$ and $F_0 \subseteq F_i \subseteq G'$.
Thus, by the definition of $\hat{\mathcal{A}}$ and that $m_2(H_1) = m_2(H_2)$, we have that $m(F_i) > m_2(H_1, H_2) + \varepsilon$.

\begin{figure}[!ht]
\begin{algorithmic}[1]
\Procedure{Grow-Alt}{$G' = (V,E)$}
 \If {$\forall e \in E: |\mathcal{S}_{G'}(e)| = 1$}\label{line:growaltforalle}
     \State $T \gets$ any member of $\mathcal{T}_{G'}$\label{line:growaltanyt}
     \State \Return $\bigcup_{e \in E(T)}\mathcal{S}_{G'}(e)$\label{line:growaltreturnsg'e}
 \EndIf
 \If{$\exists e \in E : |\mathcal{S}_{G'}(e)| \geq 2$}\label{line:growaltsg'e2}
     \State $S_1, S_2 \gets$ any two distinct members of $\mathcal{S}_{G'}(e)$\label{line:growalt2members}
     \State \Return $S_1 \cup S_2$\label{line:growalts1s2}
 \EndIf\label{line:endifgrowalt}
 \State $e \gets$ any $e \in E : |\mathcal{S}_{G'}(e)| = 0$\label{line:growaltanye}
 \State $F_0 \gets$ any $R \in \mathcal{R}_{G'}:e \in E(R)$\label{line:growaltanyr}
 \State $i \gets 0$\label{line:growalti0}
 \While {$(i < \ln(n)) \land (\forall \tilde{F} \subseteq F_i : \lambda(\tilde{F}) > - \gamma)$}\label{line:growaltwhileconditions}
         \State $e \gets \textsc{Eligible-Edge-Alt}(F_i)$\label{line:eligible}
         \State $F_{i+1} \gets \textsc{Extend}(F_i, e, G')$\label{line:extend}    
         \State $i \gets i + 1$
 \EndWhile 
 \If {$i \geq \ln(n)$}\label{line:growaltilnn}
     \State \Return{$F_i$}\label{line:growaltreturnfi}
 \Else
     \State \Return{\textsc{Minimising-Subgraph}($F_i$)}\label{line:growaltreturnminsub}
 \EndIf
\EndProcedure
\end{algorithmic}\smallskip\smallskip

\begin{algorithmic}[1]
\Procedure{Extend}{$F,e,G'$}
 \State $\{L,R\} \gets$ any pair $\{L,R\}$ such that $L \in \mathcal{L}_{G'}$, $R \in \mathcal{R}_{G'}$ and $E(L) \cap E(R) = e$\label{line:lr}
 \If{$L \nsubseteq F$}\label{line:ifl}
 \State $F' \gets F \cup L$\label{line:fl}
 \Else
 \State $F' \gets F \cup R$\label{line:fr}
 \EndIf
 \State \Return {$F'$}\label{line:growaltreturnf'}
\EndProcedure
\end{algorithmic}
\caption{The implementation of algorithm \textsc{Grow-Alt}.}\label{growsymfig}
\end{figure}

Thus, there exists a non-empty graph $\tilde{F} \subseteq F_i$ with $d(\tilde{F}) = m(F_i)$ such that 
\begin{align*}
\lambda(\tilde{F}) = v(\tilde{F}) - \frac{e(\tilde{F})}{m_2(H_1, H_2)}   
                                                 & = e(\tilde{F})\left(\frac{1}{m(F_i)} - \frac{1}{m_2(H_1, H_2)}\right) \\
                                                 & < e(\tilde{F})\left(\frac{1}{m_2(H_1, H_2) + \varepsilon} - \frac{1}{m_2(H_1, H_2)}\right) 
                                                  = -\gamma e(\tilde{F}) \leq -\gamma.
\end{align*}
Thus \textsc{Grow} terminates in line~\ref{line:growaltwhileconditions} without calling \textsc{Eligible-Edge-Alt}. Thus every call that is made to \textsc{Eligible-Edge-Alt} is successful and returns an edge $e$. 
Since $G' \in \mathcal{C}$, the call to \textsc{Extend}$(F_i, e, G')$ is also successful and thus there exist suitable graphs $L \in \mathcal{L}_{G'}$ and $R \in \mathcal{R}_{G'}$ such that $E(L) \cap E(R) = \{e\}$, that is, line~\ref{line:lr} is successful. 

It remains to show that for every iteration $i$ of the while-loop, we have $e(F_{i+1}) > e(F_i)$. Since $e$ is eligible for extension in \textsc{Grow-Alt} for $F_i$ and $E(L) \cap E(R) = \{e\}$, we must have that either $L \nsubseteq F_i$ or $R \nsubseteq F_i$. 
Hence \textsc{Extend} outputs $F' := F \cup L$ such that $e(F_{i+1}) = e(F') > e(F_i)$ or $F' := F \cup R$ such that $e(F_{i+1}) = e(F') > e(F_i)$.\end{proof}

We consider the evolution of $F_i$ now in more detail.
We call iteration $i$ of the while-loop in algorithm \textsc{Grow-Alt} \emph{non-degenerate} if the following hold:

\begin{itemize}
 \item If $L \nsubseteq F_i$ (that is, line~\ref{line:ifl} is true), then in line~\ref{line:fl} we have $V(F_i) \cap V(L) = e$;
 \item If $L \subseteq F_i$ (that is, line~\ref{line:ifl} is false), then in line~\ref{line:fr} we have $V(F_i) \cap V(R) = e$.
\end{itemize}
Otherwise, we call iteration $i$ \emph{degenerate}. Note that, in non-degenerate iterations $i$, there are only a constant number of graphs that $F_{i+1}$ can be for any given $F_i$; indeed, \textsc{Eligible-Edge-Alt} determines the exact position where to attach $L$ or $R$ (recall that the edge $e$ found by \textsc{Eligible-Edge-Alt}($F_i$) is \emph{unique up to isomorphism of $F_i$}).

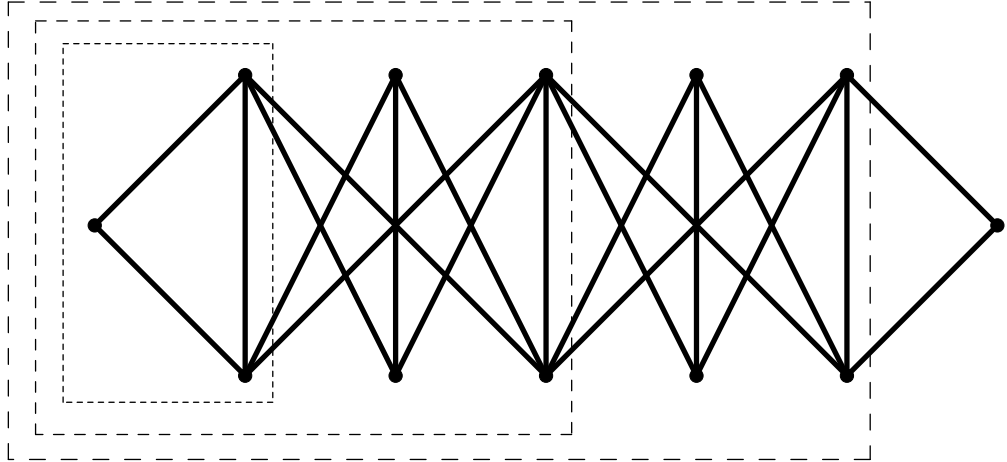
\begin{figure}[!ht]
\begin{center}
\begin{tikzpicture}[line cap=round,line join=round,>=triangle 45,x=1cm,y=1cm]
\draw [line width=2pt] (4,7)-- (4,3);
\draw [line width=2pt] (4,7)-- (2,5);
\draw [line width=2pt] (2,5)-- (4,3);
\draw [line width=2pt] (6,7)-- (6,3);
\draw [line width=2pt] (8,7)-- (8,3);
\draw [line width=2pt] (4,7)-- (6,3);
\draw [line width=2pt] (4,7)-- (8,3);
\draw [line width=2pt] (4,3)-- (6,7);
\draw [line width=2pt] (4,3)-- (8,7);
\draw [line width=2pt] (6,3)-- (8,7);
\draw [line width=2pt] (8,3)-- (6,7);
\draw [line width=2pt] (10,7)-- (10,3);
\draw [line width=2pt] (12,7)-- (12,3);
\draw [line width=2pt] (8,7)-- (10,3);
\draw [line width=2pt] (8,7)-- (12,3);
\draw [line width=2pt] (8,3)-- (10,7);
\draw [line width=2pt] (8,3)-- (12,7);
\draw [line width=2pt] (10,3)-- (12,7);
\draw [line width=2pt] (12,3)-- (10,7);
\draw [line width=2pt] (12,7)-- (14,5);
\draw [line width=2pt] (12,3)-- (14,5);
\draw [line width=0.5pt,dash pattern=on 2pt off 2pt] (1.579665672009385,7.418079578918211)-- (4.3670982412967865,7.418079578918211);
\draw [line width=0.5pt,dash pattern=on 2pt off 2pt] (4.3670982412967865,7.418079578918211)-- (4.3670982412967865,2.6470220520129017);
\draw [line width=0.5pt,dash pattern=on 2pt off 2pt] (4.3670982412967865,2.6470220520129017)-- (1.5796656720093858,2.6470220520129017);
\draw [line width=0.5pt,dash pattern=on 2pt off 2pt] (1.5796656720093858,2.6470220520129017)-- (1.579665672009385,7.418079578918211);
\draw [line width=0.5pt,dash pattern=on 4pt off 4pt] (1.2136693391219346,7.720617180014802)-- (8.34061651214996,7.720617180014802);
\draw [line width=0.5pt,dash pattern=on 4pt off 4pt] (8.34061651214996,7.720617180014802)-- (8.34061651214996,2.218214623372848);
\draw [line width=0.5pt,dash pattern=on 4pt off 4pt] (8.34061651214996,2.218214623372848)-- (1.2136693391219358,2.218214623372847);
\draw [line width=0.5pt,dash pattern=on 4pt off 4pt] (1.2136693391219358,2.218214623372847)-- (1.2136693391219346,7.720617180014802);
\draw [line width=0.5pt,dash pattern=on 6pt off 6pt] (0.851603920027077,7.971645770046527)-- (12.309167214867147,7.971645770046527);
\draw [line width=0.5pt,dash pattern=on 6pt off 6pt] (12.309167214867147,7.971645770046527)-- (12.309167214867147,1.886130908716004);
\draw [line width=0.5pt,dash pattern=on 6pt off 6pt] (12.309167214867147,1.886130908716004)-- (0.8516039200270787,1.886130908716004);
\draw [line width=0.5pt,dash pattern=on 6pt off 6pt] (0.8516039200270787,1.886130908716004)-- (0.851603920027077,7.971645770046527);
\begin{scriptsize}
\draw [fill=black] (4,7) circle (2.5pt);
\draw [fill=black] (4,3) circle (2.5pt);
\draw [fill=black] (2,5) circle (2.5pt);
\draw [fill=black] (6,7) circle (2.5pt);
\draw [fill=black] (6,3) circle (2.5pt);
\draw [fill=black] (8,7) circle (2.5pt);
\draw [fill=black] (8,3) circle (2.5pt);
\draw [fill=black] (10,7) circle (2.5pt);
\draw [fill=black] (10,3) circle (2.5pt);
\draw [fill=black] (12,7) circle (2.5pt);
\draw [fill=black] (12,3) circle (2.5pt);
\draw [fill=black] (14,5) circle (2.5pt);
\end{scriptsize}
\end{tikzpicture}
\caption{A graph $F_3$ resulting from three non-degenerate iterations for $H_1 = K_3$ and $H_2 = K_{3,3}$.}
\end{center}
\end{figure}

We now prove a result analogous to Claim~\ref{claim:non-degen} for \textsc{Grow-Alt}.

\begin{claim}\label{claim:non-degensym}
If iteration $i$ of the while-loop in procedure \textsc{Grow-Alt} is non-degenerate, we have \[\lambda(F_{i+1}) = \lambda(F_i).\]
\end{claim}

\begin{proof}
In a non-degenerate iteration, we either add $v_2 - 2$ vertices and $e_2 - 1$ edges for the copy $L$ of $H_2$ to $F_i$ or add $v_1 - 2$ vertices and $e_1 - 1$ edges for the copy $R$ of $H_1$ to $F_i$.
In the former case,
\begin{align*}
  \lambda(F_{i+1}) - \lambda(F_i)   & = v_2 - 2 - \frac{e_2 - 1}{m_2(H_1,H_2)} \\
                                 & = v_2 - 2 - \frac{e_2 - 1}{m_2(H_2)} \\
                                 & = 0,
\end{align*} where the second equality follows from $m_2(H_1,H_2) = m_2(H_2)$ (see Proposition~\ref{prop:m2h1h2}) and the last equality follows from $H_2$ being (strictly) $2$-balanced.

In the latter case,
\begin{align*}
  \lambda(F_{i+1}) - \lambda(F_i)   & = v_1 - 2 - \frac{e_1 - 1}{m_2(H_1,H_2)} \\
                                 & = v_1 - 2 - \frac{e_1 - 1}{m_2(H_1)} \\
                                 & = 0.
\end{align*} where the second equality follows from $m_2(H_1,H_2) = m_2(H_1)$ (see Proposition~\ref{prop:m2h1h2}) and the last equality follows from $H_1$ being (strictly) $2$-balanced.\end{proof} 

As in Case~1, when we have a degenerate iteration $i$, the structure of $F_{i+1}$ depends not just on $F_i$ but also on the structure of $G'$.
Indeed, if $F_i$ is extended by a copy $L$ of $H_2$ in line~\ref{line:fl} of \textsc{Extend}, then $L$ could intersect $F_i$ in a multitude of ways.
Moreover, there may be several copies of $H_2$ that satisfy the condition in line~\ref{line:lr} of \textsc{Extend}.
One could say the same for graphs added in line~\ref{line:fr} of \textsc{Extend}.
Thus, as in Case~1, degenerate iterations cause us difficulties since they enlarge the family of graphs algorithm \textsc{Grow-Alt} can return.
However, we will show that at most a constant number of degenerate iterations can happen before algorithm \textsc{Grow-Alt} terminates, allowing us to bound from above sufficiently well the number of non-isomorphic graphs \textsc{Grow-Alt} can return.  
Pivotal in proving this is the following claim, analogous to Claim~\ref{claim:degenfull}.

\begin{claim}\label{claim:degenfullsym}
There exists a constant $\kappa = \kappa(H_1,H_2) > 0$ such that if iteration $i$ of the while-loop in procedure \textsc{Grow} is degenerate then we have \[\lambda(F_{i+1}) \leq \lambda(F_i) - \kappa.\]
\end{claim}

Compared to the proof of Claim~\ref{claim:degenfull}, the proof of Claim~\ref{claim:degenfullsym} is relatively straightforward.

\begin{proof}
Let $F := F_i$ be the graph before the operation in line~\ref{line:extend} (of \textsc{Grow-Alt}) is carried out and let $F' := F_{i+1}$ be the output from line~\ref{line:extend}. We aim to show there exists a constant $\kappa = \kappa(H_1, H_2) > 0$ such that
\begin{equation*}
    \lambda(F) - \lambda(F') = v(F) - v(F') - \frac{e(F) - e(F')}{m_2(H_1, H_2)} \geq \kappa
\end{equation*} whether \textsc{Extend} attached a graph $L \in \mathcal{L}_{G'}$ or $R \in \mathcal{R}_{G'}$ to $F$. We need only consider the case when $L \in \mathcal{L}_{G'}$ is the graph added by \textsc{Extend} in this non-degenerate iteration $i$ of the while-loop, as the proof of the $R \in \mathcal{R}_{G'}$ case is identical. 
Let $V_{L'} := V(F) \cap V(L)$ and $E_{L'} := E(F) \cap E(L)$ and set $v_{L'} := |V_{L'}|$ and $e_{L'} := |E_{L'}|$.
Let ${L'}$ be the graph on vertex set $V_{L'}$ and edge set $E_{L'}$. By Claim~\ref{claim:growsym}, we have $e(F') > e(F)$, hence ${L'}$ is a proper subgraph of $H_2$. Also, observe that
since iteration $i$ was degenerate, we must have that $v_{L'} \geq 3$. 
Let $F_{\hat{L}}$ be the graph produced by a non-degenerate iteration at $e$ with a copy $\hat{L}$ of $H_2$, that is, $F_{\hat{L}} := F \cup \hat{L}$ and $V(F) \cap V(\hat{L}) = e$. Our strategy is to compare $F'$ with $F_{\hat{L}}$. 
By Claim~\ref{claim:non-degensym}, $\lambda(F) = \lambda(F_{\hat{L}})$. Thus, since $m_2(H_1, H_2) = m_2(H_2)$, we have 
\begin{equation}\label{eq:f1beginsymL}
 \lambda(F) - \lambda(F') = \lambda(F_{\hat{L}}) - \lambda(F') = v_2 - 2 - (v_2 - v_{L'}) - \frac{(e_2 - 1) - (e_2 - e_{L'})}{m_2(H_1, H_2)} = v_{L'} - 2 - \frac{e_{L'} - 1}{m_2(H_1, H_2)}. 
\end{equation}
If $e_{L'} = 1$, then since $v_{L'} \geq 3$ we have $\lambda(F) - \lambda(F') \geq 1$. So assume $e_{L'} \geq 2$. 
Since $H_2$ is \emph{strictly} $2$-balanced and $L'$ is a proper subgraph of $H_2$ with $e_{L'} \geq 2$ (that is, $L'$ is not an edge), we have that \begin{equation}\label{eq:f1l'sym}
 \frac{v_{L'} - 2}{e_{L'} - 1} > \frac{1}{m_2(H_2)}.
\end{equation}
Using \eqref{eq:f1beginsymL}, \eqref{eq:f1l'sym} and that $e_{L'} \geq 2$ and $m_2(H_1, H_2) = m_2(H)$, we have 
\begin{equation*}
 \lambda(F) - \lambda(F') = (e_{L'} - 1)\left(\frac{v_{L'} - 2}{e_{L'} - 1} - \frac{1}{m_2(H_2)}\right) > 0.
\end{equation*}
Letting $\delta := \frac{1}{2} \min\left\{(e_{L'} - 1)\left(\frac{v_{L'} - 2}{e_{L'} - 1} - \frac{1}{m_2(H_2)}\right): L' \subset H_2, e_{L'} \geq 2\right\}$, we take $$\kappa := \min\{1, \delta\}.$$\end{proof}

Together, Claims~\ref{claim:non-degensym} and \ref{claim:degenfullsym} yield the following claim (analogous to Claim~\ref{claim:q_1}). 

\begin{claim}\label{claim:q_2}
There exists a constant $q_2 = q_2(H_1,H_2)$ such that algorithm \textsc{Grow-Alt} performs at most $q_2$ degenerate iterations before it terminates, regardless of the input instance $G'$.
\end{claim}

\begin{proof}
Analogous to the proof of Claim~\ref{claim:q_1}.\end{proof}

For $0 \leq d \leq t \leq \lceil \ln(n) \rceil$, let $\mathcal{F}_{\textsc{Alt}}(t,d)$ denote a family of representatives for the isomorphism classes of all graphs $F_t$ that algorithm \textsc{Grow-Alt} can possibly generate after exactly $t$ iterations of the while-loop with exactly $d$ of those $t$ iterations being degenerate. 
Let $f_{\textsc{Alt}}(t,d) := |\mathcal{F}_{\textsc{Alt}}(t,d)|$.

\begin{claim}\label{claim:polylogsym}
There exist constants $C_0 = C_0(H_1, H_2)$ and $A = A(H_1,H_2)$ such that $$f_{\textsc{Alt}}(t,d) \leq \lceil \ln(n) \rceil^{(C_0 +1)d} \cdot A^{t-d}$$ for $n$ sufficiently large.
\end{claim}

\begin{proof}
Analogous to the proof of Claim~\ref{claim:polylog}.\end{proof}

Let $\mathcal{F}_{\textsc{Alt}} = \mathcal{F}_{\textsc{Alt}}(H_1, H_2, n)$ be a family of representatives for the isomorphism classes of \emph{all} graphs that can be outputted by \textsc{Grow-Alt} (whether \textsc{Grow-Alt} enters the while-loop or not). 
Note that the proof of the following claim requires Conjecture~\ref{conj:aconj} to be true; in particular, we need that $\hat{\mathcal{A}}$ is finite when $m_2(H_1) = m_2(H_2)$.

\begin{claim}\label{claim:conclusionsym}
There exists a constant $b = b(H_1,H_2) > 0$ such that for all $p \leq bn^{-1/m_2(H_1,H_2)}$, $G_{n,p}$ does not contain any graph from $\mathcal{F}_{\textsc{Alt}}(H_1,H_2,n)$ a.a.s.
\end{claim}

\begin{proof}
Analogous to the proof of Claim~\ref{claim:conclusion}.\end{proof}

\begin{proofofnoerrorlemmacasetwo}
Suppose that the call to \textsc{Asym-Edge-Col}$(G)$ gets stuck for some graph $G$, and consider $G' \subseteq G$ at this moment. 
Then \textsc{Grow-Alt}$(G')$ returns a copy of a graph $F \in \mathcal{F}_{\textsc{Alt}}(H_1, H_2, n)$ that is contained in $G' \subseteq G$. 
By Claim~\ref{claim:conclusionsym}, this event  a.a.s.~ does not occur in $G = G_{n,p}$ with $p$ as claimed. Thus \textsc{Asym-Edge-Col} does not get stuck a.a.s.~and, by Lemma~\ref{lemma:acolour}, finds a valid edge-colouring for $H_1$ and $H_2$ of $G_{n,p}$ with $p \leq bn^{-1/m_2(H_1,H_2)}$ a.a.s.\qed
\end{proofofnoerrorlemmacasetwo}

\section{Proof of Theorem~\ref{thm:regular}}\label{sec:claimregular}

Since $H_1$ and $H_2$ are regular graphs, let $\ell_1$ be the degree of every vertex in $H_1$ and $\ell_2$ be the degree of every vertex in $H_2$.
We begin using an approach employed in \cite{kk} and \cite{msss}.
Let $A$ be a $2$-connected graph such that $A \in \mathcal{C}^*(H_1, H_2)$ if $m_2(H_1) > m_2(H_2)$ and $A \in \mathcal{C}(H_1, H_2)$ if $m_2(H_1) = m_2(H_2)$. In both cases, $A \in \mathcal{C}(H_1, H_2)$.
Then, since every vertex is contained in a copy of $H_1$ and a copy of $H_2$ whose edge-sets intersect in exactly one edge, we must have that $\delta(A) \geq \ell_1 + \ell_2 - 1$. Hence
\begin{equation}\label{eq:da}
    d(A) = \frac{e_A}{v_A} \geq \frac{\ell_1 + \ell_2 - 1}{2}.
\end{equation} 
We aim to show $d(A) > m_2(H_1, H_2) + \varepsilon$ for some $\varepsilon = \varepsilon(H_1, H_2) > 0$.
Indeed, if there exists $\varepsilon = \varepsilon(H_1, H_2) > 0$ such that for all such graphs $A$ we have $d(A) > m_2(H_1, H_2) + \varepsilon$ then $\hat{\mathcal{A}}(H_1, H_2, \varepsilon) = \emptyset$, and so Conjecture~\ref{conj:aconj} holds trivially for $H_1$ and $H_2$. 

Since $m_2(H_1) \geq m_2(H_2) > 1$, we have that $H_1$ and $H_2$ cannot be matchings. Hence $\ell_1, \ell_2 \geq 2$. Also, observe that 
\begin{equation}\label{eq:m2f1f21.7}
m_2(H_1, H_2) = \frac{e_1}{v_1 - 2 + \frac{1}{m_2(H_2)}} = \frac{\frac{v_1\ell_1}{2}}{v_1 - 2 + \frac{v_2 - 2}{\frac{v_2\ell_2}{2} - 1}} = \frac{v_1\ell_1}{2v_1 - 4 + \frac{4v_2 - 8}{v_2\ell_2 - 2}}.
\end{equation} 
Furthermore,
\begin{align}\label{eq:implications}
        & \frac{\ell_1 + \ell_2 - 1}{2}     > m_2(H_1, H_2) \\
        \iff  & \ell_1 + \ell_2 - 1               > \frac{v_1v_2\ell_1\ell_2 - 2v_1\ell_1}{2v_1v_2\ell_2 - 4v_2\ell_2 - 4v_1 + 4v_2}\nonumber \\
  \iff  & \ell_1 + \ell_2 - 1               > \frac{2v_1v_2\ell_1\ell_2 - 4v_1\ell_1}{2v_1v_2\ell_2 - 4v_2\ell_2 - 4v_1 + 4v_2}\nonumber \\
   \iff  &  0 < v_1v_2\ell_2(\ell_2 - 1) - 2v_1(\ell_2 - 1) - 2v_2\ell_1(\ell_2 - 1) - 2v_2(\ell_2 - 1)^2\nonumber \\
  \iff  & 0 < v_1v_2\ell_2 - 2v_1 - 2v_2\ell_1 - 2v_2\ell_2 + 2v_2,\nonumber
\end{align} where in the last implication we used that $\ell_2 \geq 2$.

Let $f(v_1, v_2, \ell_1, \ell_2) := v_1v_2\ell_2 - 2v_1 - 2v_2\ell_1 - 2v_2\ell_2 + 2v_2$.
Observe that $\ell_1 \leq v_1 - 1$. 
Hence $-2v_2\ell_1 \geq -2v_1v_2 + 2v_2$ and we have \[f(v_1, v_2, \ell_1, \ell_2) \geq v_1v_2(\ell_2 - 2) - 2v_1 + 4v_2 - 2v_2\ell_2.\] 
Let $g(v_1, v_2, \ell_2) := v_1v_2(\ell_2 - 2) - 2v_1 + 4v_2 - 2v_2\ell_2$ so $f(v_1, v_2, \ell_1, \ell_2) \geq g(v_1, v_2, \ell_2)$. 
Observe that, since $\ell_1 ,\ell_2 \geq 2$, we have that $v_1, v_2 \geq 3$. If $\ell_2 = 2$, then $g(v_1, v_2, 2) = -2v_1 < 0$.
However, if $\ell_2 \geq 3$, then since $v_1, v_2 \geq 3$, we have that 

\begin{equation*}
    \frac{dg}{dv_1} =  v_2(\ell_2 - 2) - 2  >  0;
\end{equation*} 
\begin{equation*}
    \frac{dg}{dv_2} =  v_1(\ell_2 - 2) + 4 - 2\ell_2  =  (v_1 - 2)(\ell_2 - 2)  >  0;
\end{equation*}    
\begin{equation*}
    \frac{dg}{d\ell_2} =  v_1v_2 - 2v_2  =  v_2(v_1 - 2)  >  0.
\end{equation*}
Further, \[g(4,5,3) = 20 - 8 + 20 - 30 = 2 > 0.\]

Thus, for all $v_1 \geq 4, v_2 \geq 5, \ell_2 \geq 3$, we have that \[f(v_1, v_2, \ell_1, \ell_2) \geq g(v_1, v_2, \ell_2) \geq g(4,5,3) = 2 > 0.\]  
Hence, by \eqref{eq:da}-\eqref{eq:implications}, we have that there exists a constant $\varepsilon := \varepsilon(H_1, H_2) > 0$ such that \[d(A) > m_2(H_1, H_2) + \varepsilon.\] 
Thus we only have left the cases when $v_1 = 3$, $v_2 \leq 4$ or $\ell_2 = 2$. 

\begin{case1} $\ell_2 = 2$.\end{case1}

Then $H_2$ is a cycle and \[f(v_1, v_2, \ell_1, 2) = 2v_1v_2 - 2v_1 - 2v_2\ell_1 - 2v_2.\] 
Observe that $\ell_1 \leq v_1 - 2$, as otherwise $H_1$ is a clique, contradicting that $(H_1, H_2)$ is not a pair of a clique and a cycle. Thus $-\ell_1 \geq -(v_1 - 2)$.
Then, since we excluded considering when $H_2$ is a cycle and $H_1$ is a graph with $v_1 = |V(H_1)| \geq |V(H_2)| = v_2$, we have that $v_2 > v_1$, and so \[f(v_1, v_2, \ell_1, 2) \geq 2(v_2 - v_1) \geq 2 > 0.\]
Hence by \eqref{eq:da}-\eqref{eq:implications}, we have that there exists a constant $\varepsilon := \varepsilon(H_1, H_2) > 0$ such that $d(A) > m_2(H_1, H_2) + \varepsilon$. 

\begin{case2}  $v_1 = 3$. \end{case2}

Then $H_1 = K_3$ and $\ell_1 = 2$. Thus \[f(3, v_2, 2, \ell_2) = v_2(\ell_2 - 2) - 6.\]
We can assume $\ell_2 \geq 3$, as otherwise we are in Case~1.
Observe that we cannot have that $v_2 = 6$ and $\ell_2 \geq 3$.
Indeed, when $\ell_2 = 3$, one can check that the only strictly $2$-balanced $3$-regular graph on 6 vertices is $K_{3,3}$. 
But then $(H_1, H_2) = (K_{3}, K_{3,3})$, which is a pair of graphs we excluded from consideration.

When $\ell_2 \geq 4$, we have that $m_2(H_2) > m_2(H_1)$, contradicting that our choice of $H_1$ and $H_2$ meet the criteria in Conjecture~\ref{conj:aconj}.
If $v_2 \leq 5$, then since also $\ell_2 \geq 3$ we must have that $H_2$ is a copy of $K_4$ or $K_5$.\footnote{There exists no graph with both odd regularity and odd order.} But then $m_2(H_2) > m_2(H_1)$. 

Hence $v_2 \geq 7$ and $\ell_2 \geq 3$. Thus $f(3, v_2, 2, \ell_2) > 0$ and, as before, we conclude that there exists a constant $\varepsilon := \varepsilon(H_1, H_2) > 0$ such that $d(A) > m_2(H_1, H_2) + \varepsilon.$

\begin{case3} $v_2 \leq 4$. \end{case3}

Still assuming $\ell_2 \geq 3$, we have that $H_2 = K_4$, $v_2 = 4$ and $\ell_2 = 3$.  If $v_1 \geq \ell_1 + 2$, then \[f(v_1, 4, \ell_1, 3) = 10v_1 - 8(\ell_1 + 2) \geq 2(\ell_1 + 2) > 0.\]
If $v_1 = \ell_1 + 1$, then $H_1$ is a clique. Since $H_1$ and $H_2$ meet the criteria in Conjecture~\ref{conj:aconj}, we must have that $v_1 \geq 5$. Therefore \[f(v_1, 4, \ell_1, 3) = 10v_1 - 8(\ell_1 + 2) = 2v_1 - 8 \geq 2 > 0.\]

Then, as before, there exists a constant $\varepsilon := \varepsilon(H_1, H_2) > 0$ such that $d(A) > m_2(H_1, H_2) + \varepsilon$, as desired.

\section{Concluding remarks}\label{sec:concludingremarks}

In \cite{msss}, the value of $\varepsilon$ was set at $0.01$ for \emph{every} pair of cliques, that is, the value of $\varepsilon$ did not depend explicitly on the graphs $H_1$ and $H_2$.
It would be interesting to know if there exists a single value of $\varepsilon$ satisfying Conjecture~\ref{conj:aconj} for \emph{all} suitable pairs of graphs $H_1$ and $H_2$.

Let us discuss why we use \textsc{Grow-Alt} when $m_2(H_1) = m_2(H_2)$ instead of \textsc{Grow}.
The main reason is that when attaching a copy of $H_1$ precisely at an edge\footnote{That is, $e_1 - 1$ edges and $v_1 - 2$ vertices are added.} of $F_i$ to create $F_{i+1}$ - for some iteration $i$ during \textsc{Grow} or \textsc{Grow-Alt} - we have $\lambda(F_i) = \lambda(F_{i+1})$ if $m_2(H_1) = m_2(H_2)$ and $\lambda(F_i) > \lambda(F_{i+1})$ if $m_2(H_1) > m_2(H_2)$. We would say that iteration $i$ was a non-degenerate iteration if $m_2(H_1) = m_2(H_2)$ and a degenerate iteration if $m_2(H_1) > m_2(H_2)$. 
That is, if we were to use \textsc{Grow} instead of \textsc{Grow-Alt} then the number of possible non-degenerate iterations would depend on the size of $F_i$ and grow too quickly. 

Connected to this observation is why for $A \in \hat{\mathcal{A}}$ we take $A \in \mathcal{C}$ when $m_2(H_1) = m_2(H_2)$, rather than $A \in \mathcal{C}^*$.\footnote{As mentioned at the end of Section~\ref{sec:ainc*orinc}.} 
In the proof of Claim~\ref{claim:growsym}, we need to conclude that $F_i$ is not in $\hat{\mathcal{A}}$ in order to conclude \textsc{Eligible-Edge-Alt} is always successful.
As part of concluding this, we need that $F_i \in \mathcal{C}$, which we get by assuming for a contradiction that there is no edge in $F_i$ eligible for extension in \textsc{Grow-Alt}. 
Hence, if for all $A \in \hat{\mathcal{A}}$ one sets $A \in \mathcal{C}^*$ when $m_2(H_1) = m_2(H_2)$, then one needs to conclude that $F_i \in \mathcal{C}^*$. 
This means that we must assume for a contradiction that for every edge $e \in E(F_i)$ there exists $L \in L_{F_i}^*$ such that $e \in E(L)$. 
Such an argument then requires us to change from using \textsc{Eligible-Edge-Alt} to \textsc{Eligible-Edge} in \textsc{Grow-Alt} in order to arrive at the correct contradiction in the proof. 
But now \textsc{Extend$(F_i, e, G')$} in \textsc{Grow-Alt} may fail to add an edge as it is possible that every $L \in \mathcal{L}_{G'}^*$ containing $e$ has its copies of $H_1$ and $H_2$ intersecting at $e$ both contained fully in $F_i$,
that is, we cannot guarantee that in every iteration $i$ of the while-loop of \textsc{Grow-Alt} that $e(F_{i+1}) > e(F_i)$. 

In order to overcome this we would need to replace procedure Extend with a procedure similar to procedure \textsc{Extend-L}, which is used in \textsc{Grow}. 
If we used \textsc{Extend-L}, then the copy $L$ of $H_1$ could be fully contained in $F_i$ and one of its copies $R'$ of $H_2$ appended to it could contain precisely one edge of $F_i$ and two vertices of $F_i$, with every other appended copy of $H_2$ fully contained in $F_i$.
This would then have to be considered a non-degenerate iteration for \textsc{Grow-Alt} as $\lambda(F_i) = \lambda(F_{i+1})$. 
But, importantly, \textsc{Extend-L$(F_i, e, G')$} depends on the structure of $G'$, not on the structure of $F_i$. 
Thus $R'$ could be attached at a number of edges of $F_i$. 
This would increase the number of possible non-degenerate iterations too much. 

What procedure similar to \textsc{Extend-L} do we choose then? One possibility would be to attach a copy $L$ of $H_2$ with copies of $H_1$ appended to \emph{every one} of its edges $e'$, but this procedure runs into the same problem as \textsc{Extend-L}. 
Overall, this discussion shows how difficult - and maybe impossible - it would be to prove the $m_2(H_1) = m_2(H_2)$ case of Lemma~\ref{lemma:noerror} while stipulating that when $A \in \hat{\mathcal{A}}$ we have $A \in \mathcal{C}^*$.

Also, in \cite{msss}, Marciniszyn, Skokan, Sp\"{o}hel and Steger note that they do not know whether $\mathcal{C}^*(H_1, H_2) = \mathcal{C}(H_1, H_2)$ or not for any pair of cliques.
The author is unaware if this has been resolved for any pair of graphs.
According with the definition of $\hat{\mathcal{A}}$, perhaps it is the case that $\mathcal{C}^*(H_1, H_2) = \mathcal{C}(H_1, H_2)$ whenever $m_2(H_1) = m_2(H_2)$?

Note that our method in this paper does not completely extend to when $m_2(H_2) = 1$.
Indeed, in this case $H_2$ is a forest and so not $2$-connected,
a property which is specifically used in the proofs of Claims~\ref{claim:grow}, \ref{claim:conclusion}, \ref{claim:growsym} and \ref{claim:conclusionsym}. 
In the proofs of Claims~\ref{claim:grow} and \ref{claim:growsym} we require that $F_i$ is $2$-connected for each $i > 0$. 
However, if the last iteration of the while-loop of either \textsc{Grow} or \textsc{Grow-Alt} was non-degenerate when constructing $F_i$, 
then $F_i$ is certainly not $2$-connected if $H_2$ is a tree. A similar problem occurs at the beginning of the proofs of Claims~\ref{claim:conclusion} and \ref{claim:conclusionsym} if $T \cong H_2$.

\section{Acknowledgements}

The author is grateful to Andrew Treglown for suggesting working on the $0$-statement of the Kohayakawa-Kreuter Conjecture and to Robert Hancock and Andrew Treglown for providing many helpful comments on a draft of this paper. The author also thanks the referee for their comments and suggestions.

\end{document}